\documentclass[11pt]{smfart}
\usepackage{amssymb,latexsym}
\usepackage{graphicx}
\usepackage{version}
\usepackage[all]{xy} 
\usepackage{lscape}
\entrymodifiers={!!<0pt,0.7ex>+}

\usepackage{geometry}
\usepackage[applemac]{inputenc} 
\usepackage{amsfonts}
\usepackage{amssymb}
\usepackage{epstopdf}

\newtheorem{result}{Result}{\bf}{\it}
\newtheorem{theorem}{Theorem}[section]

\newtheorem{proposition}{Proposition}[section]

\newtheorem{lemma}{Lemma}[section]
\newtheorem{definition}{Definition} [section]
\newtheorem{notation}{Notation}[section]
\newtheorem{remark}{Remark}[section]

\newtheorem{exemple}{Exemple}[section]

\def\O{{\mathfrak{O}}}
\begin{document}
\title[Partitions and geometry]{Combinatorial theory of permutation-invariant random matrices~II: \\ Cumulants, freeness and Levy processes.}
\author[Franck Gabriel]{Franck Gabriel \\ E-mail: \parbox[t]{0.45\linewidth}{\texttt{franck.gabriel@normalesup.org}}}
\email{franck.gabriel@normalesup.org}
\address{Université Pierre et Marie Curie (Paris 6)\\ Laboratoire de Probabilités et Modèles Aléatoires\\ 4, Place Jussieu\\  F-75252 Paris Cedex 05. \emph{Present address:}  Mathematics Institute, \\ University of Warwick,\\ Gibbet Hill Rd, Coventry CV4 7AL }

\begin{abstract}
The $\mathcal{A}$-tracial algebras are algebras endowed with multi-linear forms, compatible with the product, and indexed by partitions. Using the notion of $\mathcal{A}$-cumulants, we define and study the $\mathcal{A}$-freeness property which generalizes the independence and freeness properties, and some invariance properties which model the invariance by conjugation for random matrices. A central limit theorem is given in the setting of $\mathcal{A}$-tracial algebras. A generalization of the normalized moments for random matrices is used to define convergence in $\mathcal{A}$-distribution: this allows us to apply the theory of $\mathcal{A}$-tracial algebras to random matrices. This study is deepened with the use of $\mathcal{A}$-finite dimensional cumulants which are related to some dualities as the Schur-Weyl's duality. This gives a unified and simple framework in order to understand families of random matrices which are invariant by conjugation in law by any group whose associated tensor category is spanned by partitions, this includes for example the unitary groups  or the symmetric groups. Among the various by-products, we prove that unitary invariance and convergence in distribution implies convergence in $\mathcal{P}$-distribution. Besides, a new notion of strong asymptotic invariance and independence are shown to imply $\mathcal{A}$-freeness. Finally, we prove general theorems about convergence of matrix-valued additive and multiplicative L\'{e}vy processes which are invariant in law by conjugation by the symmetric group. Using these results, a unified point of view on the study of matricial L\'{e}vy processes is given. 

\keywords{notions of freeness \and random matrices \and cumulants \and Schur-Weyl-Jones duality  \and  partitions algebras \and easy orthogonal groups \and L\'{e}vy processes}
\end{abstract}

\maketitle

\tableofcontents{}

\section{Introduction}

This is the second article of a series of three in which we generalize the notions of independence and freeness in order to define a notion of $\mathcal{A}$-freeness in the setting of $\mathcal{A}$-tracial algebras. We also generalize the notion of cumulants for random matrices of fixed size using some dualities. This setting unifies classical and free probabilities and allows us to study random matrices which are not asymptotically invariant in law by conjugation by the unitary group but by smaller groups of the unitary group. In the first article \cite{Gab1}, the reader will find the combinatorial tools needed; in this article, he will find the study of $\mathcal{A}$-tracial algebras and applications to random matrices; the article \cite{Gab3} uses the previous result and focuses on the study of general random walks on the symmetric group and the construction of the $\mathfrak{S}(\infty)$-master field.

\subsection{Small reminder about random matrices}

\label{sec:review}
Random matrices are random variables which take values in a set of matrices $\mathcal{M}_{n,m}(\mathbb{C})$. We will only consider square random matrices: $M$ is of size $N$ if $M$ is a square matrix of size $N \times N$ and we will write $M \in \mathcal{M}_{N}(\mathbb{C})$. Actually any random matrix $M$ of size $N$ in this article is supposed to be in ${L}^{\infty^-}(\Omega) \otimes \mathcal{M}_{N}(\mathbb{C})$ where $(\Omega, \mathcal{A}, \mathbb{P})$ is fixed: this means that for any $i$, $j$ in $\{1,...,N\}$ and any positive integer $k$, $\mathbb{E}[\mid\! M_{i,j}\!\mid^{k}]<\infty$.

From now on, when a letter indexed by $N$ represents a matrix, this matrix is of size $N$. Let $(M_N)_{N \in \mathbb{N}}$ be a sequence of random matrices. Following our convention, for any $N$, $M_N \in \mathcal{M}_N(\mathbb{C})$. In the usual method of moments, one is interested in the convergence of $\frac{1}{N}Tr(\left(M_N\right)^{k})$ or the convergence of the mean moments $\mathbb{E}\left[\frac{1}{N}Tr\left(\left(M_N\right)^{k}\right)\right]$. This is justified by the fact that a random matrix $M_N$ of size $N$ has $N$ random eigenvalues: $\lambda_{1}(M_N),..., \lambda_{N}(M_N)$ and for any integer $k \in \mathbb{N}$: 
\begin{align*}
\frac{1}{N} Tr\left(\left(M_N\right)^{k}\right) = \frac{1}{N} \sum_{i=1}^{N} \lambda_{i}(M_N)^{k}. 
\end{align*}
Let the empirical eigenvalues distribution of $M_N$ be $\eta_{M_N} = \frac{1}{N}\sum_{i=1}^{N}\delta_{\lambda_{i}(M_N)}$. For any integer $k$:
\begin{align*}
\frac{1}{N}Tr\left(\left(M_N\right)^{k}\right) = \int_{\mathbb{C}} z^{k} \eta_{M_N}(dz). 
\end{align*}

Let us suppose, just until next theorem, that for any integer $N$, $M_N$ is symmetric or Hermitian. Then $\eta_{M_N}$ is a measure supported by the real line. Using the Carleman's continuity theorem, Theorem $2.2.9$ in \cite{Tao}, one can use the convergence of the moments or the mean moments to prove that, in probability or in expectation, the random measures $(\eta_{M_N})_{N \in \mathbb{N}}$ converge when $N$ goes to infinity. Likewise, we can apply similar arguments for unitary or orthogonal matrices. 

\begin{theorem} 
\label{convergencemean}
Let $(M_N)_{N \in \mathbb{N}}$ be a sequence of random matrices such that for any positive integer $k$, 
\begin{align*}
\mathbb{E}\left[\frac{1}{N} Tr\left(\left(M_N\right)^{k}\right)\right] 
\end{align*}
converges when $N$ goes to infinity. 
\begin{enumerate}
\item If for any integer $N$, $M_N$ is a unitary or orthogonal matrix then there exists $\mu$ a probability measure on the circle such that the mean empirical eigenvalues distribution $\mathbb{E}[\eta_{M_N}]$ of $M_N$ converges to $\mu$ as $N$ tends to infinity. 
\item If for any integer $N$, $M_N$ is symmetric or Hermitian (resp. skew-symmetric or skew-hermitian), under a condition of uniform subgaussianity on $(\mathbb{E}[\eta_{M_N}])_{N \in \mathbb{N}}$, the measure $\mathbb{E}[\eta_{M_N}]$ converges to a probability measure, named $\mu$, supported by the real line (resp. the imaginary line) as $N$ goes to infinity. 
\end{enumerate}
Besides, for any integer $k$: 
\begin{align*}
\int_{\mathbb{C}} z^{k} \mu(dz) = \lim_{N \to \infty}  \mathbb{E}\left[\frac{1}{N} Tr\left(\left(M_N\right)^{k}\right)\right].
\end{align*}
If the moments also converge in probability and not only in expectation, the convergence of the empirical eigenvalues distributions holds in probability. 
\end{theorem}

Studying the asymptotic of mean moments of random matrices, Voiculescu discovered the property of asymptotic freeness of unitary invariant random matrices~\cite{Voicu}. 
\begin{theorem}
\label{th:voicu}
Let $(M_N)_{N \in \mathbb{N}}$ and $(L_N)_{N \in \mathbb{N}}$ be two sequences of random matrices. Let us suppose that for any $N$, $M_N$ and $L_N$ are independent, and that $L_N$ is invariant by conjugation by the unitary group. Then $(M_N)_{N \in \mathbb{N}}$ and $(L_N)_{N \in \mathbb{N}}$ are asymptotically free. This means that for any polynomials $P_1(X)$, $Q_1(X)$, ..., $P_k(X)$, $Q_k(X)$, if the limit of the first moment of $P_1(M_N)$, $Q_1(L_N)$, ..., $P_k(M_N)$, $Q_k(L_N)$, is equal to zero then:
\begin{align*}
\lim_{N \to \infty} \mathbb{E}\left[\frac{1}{N} {Tr}\left(P_1(M_N)Q_1(L_N) ... P_k(M_N)\right)\right]&=0,\\
 \lim_{N \to \infty} \mathbb{E}\left[\frac{1}{N} {Tr}\left(P_1(M_N)Q_1(L_N) ... P_k(M_N)Q_k (L_N)\right)\right] &= 0.
\end{align*}
\end{theorem}

This leads to the definition of freeness in the abstract setting of algebras. Let $A$ be an algebra, $\phi$ be a linear form on $A$. Two sub-algebras $A_1, A_2$  of $A$ are free if for any $a_1, b_1, ..., a_k, b_k $ which are alternatively in $A_1$ and $A_2$ such that $\phi(a_i) = 0 = \phi(b_i)$ for any $i$, then $\phi(a_1b_1 ... a_kb_k ) = \phi(a_1b_1 ... a_k) =0$. 

One goal of this article is to generalize this theorem and the notion of freeness for random matrices which are not invariant by conjugation by the unitary group but {\em are invariant by subgroups of the unitary group. }
\subsection{Layout of the article}

\subsubsection{The $\mathcal{A}$-tracial algebras}
In Section \ref{sec:Atracialalge}, we extend the structure of algebra on $A$ and define the notion of $\mathcal{P}$-tracial algebra. This is an algebra endowed with some multi-linear observables, called moments, which are indexed by partitions in $\mathcal{P} := \cup_{k \in \mathbb{N}} \mathcal{P}_k$ and are in some sense tracial and compatible with the multiplication on $A$. We generalize this by defining a notion of $\mathcal{A}$-tracial algebra for any $\mathcal{A}\in \{ \mathfrak{S}, \mathcal{B}, \mathcal{B}s, \mathcal{H}, \mathcal{H}\}$. Each letter represents a subset of the sets of partitions $\cup_{k \geq 0}\mathcal{P}_k$.

\begin{result}
A notion of $\mathcal{A}$-freeness is defined for $\mathcal{A}$-tracial algebras and two characterizations of this notion are given, one uses a notion of $\mathcal{A}$-cumulants and another uses a notion of $\mathcal{A}$-exclusive-moments.  (Theorem \ref{theorem:autreformulation})
\end{result}

\begin{result}
The law of the sum or the product of $\mathcal{A}$-free elements are computed explicitely (Theorems \ref{th:calculloi} and \ref{th:Rtransfoprod})
\end{result}

Using these results about $\mathcal{A}$-freeness, we can prove a central limit theorem. 

\begin{result}
A central limit theorem, given in Theorem \ref{th:tcl}, holds for $\mathcal{A}$-tracial algebras: the limit is called a $\mathcal{A}$-Gaussian.
\end{result}

Since we can choose which set of partitions $\mathcal{A}$ we consider, the question of restriction and extension of structure arises naturally. 

\begin{result}
There exists a natural notion of restriction and extension of structure which allows to change the set $\mathcal{A}$ to another set $\mathcal{A}'$. This allows us to define a natural notion of $\mathcal{G}(\mathcal{A})$-invariance, where $\mathcal{G}(\mathcal{A})$ is a family of groups associated with the set $\mathcal{A}$. All these notions behave well with the notion of $\mathcal{A}$-freeness.  (Section \ref{sec:restiexten})
\end{result}
 
In Theorem \ref{th:critere} ,we also study the different links between all the different notions of freeness: Voiculescu freeness and $\mathcal{A}$-freeness when $\mathcal{A} \in \{ \mathfrak{S}, \mathcal{B}, \mathcal{B}s, \mathcal{H}, \mathcal{H}, \mathcal{P}\}$. 

\begin{result}
The $\mathcal{B}$-freeness or $\mathcal{B}s$-freeness implies the $\mathfrak{S}$-freeness, which,  under some hypothesis, implies Voiculescu freeness. 

Yet, $\mathcal{P}$-freeness or $\mathcal{H}$-freeness does not imply in general the $\mathcal{B}$-, $\mathcal{B}s$- or $\mathfrak{S}$-freeness. In fact, there exists a simple criterion which allows us to show that some $\mathcal{P}$-free subalgebras are not $\mathfrak{S}$-free.
\end{result}

We also introduce the notion of {\em transpose operation} for $\mathcal{A}$-tracial algebras, which allows us to prove the following result, given in Theorem \ref{th:transpofree}. 

\begin{result}
Any unitary-invariant family, which satisfies a certain hypothesis of factorization, is free (in the sense of Voiculescu) from it's transpose family. 
\end{result}

The last main result of Section \ref{sec:Atracialalge}, given in Section \ref{sec:Classicetfreecumu}, is the following fact. 

\begin{result}
One can retrieve: 
\begin{enumerate}
\item classical cumulants by considering the $\mathcal{A}$-cumulants of special elements, called classical elements, in a $\mathcal{A}$-tracial algebra, 
\item free cumulants by considering the $\mathcal{A}$-cumulants of special elements, called deterministic $U$-invariant elements, in a $\mathcal{A}$-tracial algebra. 
\end{enumerate} 
\end{result}

\subsubsection{Application to random matrices}
In Section \ref{sec:basicmatrices}, we apply the setting of $\mathcal{A}$-tracial algebras to the study of the asymptotics of random matrices. We will not state all the results since they are transcriptions of the results dealing with $\mathcal{A}$-tracial algebras, we will give only a summary. 

\begin{result}
There exists a notion of {\em $\mathcal{A}$-moments} for random matrices. This allows us to define the {\em convergence in $\mathcal{A}$-distribution} which generalizes the notion of convergence in distribution for random matrices. A new notion of {\em asymptotic $\mathcal{A}$-factorization property} allows us to study the convergence in probability of the $\mathcal{A}$-moments. Notions of {\em asymptotic $\mathcal{A}$-cumulants} and {\em asymptotic $\mathcal{A}$-exclusive moments} are defined and are used in order to define a notion of {\em asymptotic $\mathcal{A}$-freeness} for random matrices. The $\mathcal{A}$-law of the sum and product of random matrices, which are asymptotic $\mathcal{A}$-free, is computed. A notion of {\em asymptotic $\mathcal{G}(\mathcal{A})$-invariance} is defined. The links between the different notions of asymptotic freeness are studied. We retrieve the result that unitary invariant random matrices are, under some factorization property, asymptotically free in the sense of Voiculescu.
\end{result}

\subsubsection{The dualities and finite-dimensional world}
In Section \ref{duali}, we use the dualities between the set of partitions $\mathcal{A}$ and the family of groups $\mathcal{G}(\mathcal{A})$, like the Schur-Weyl duality, in order to define finite-dimensional observables which approximate the asymptotic observables (Theorem \ref{th:lienlimite}).

\begin{result}
There exists a notion of {\em finite dimensional $\mathcal{A}$-cumulants} and {\em $\mathcal{A}$-exclusive moments} for random matrices of fixed size. The random matrices converge in $\mathcal{A}$-distribution if and only if the finite dimensional  $\mathcal{A}$-cumulant or $\mathcal{A}$-exclusive moments converge. If so, both converge to their asymptotic counterpart. 
\end{result} 

Using these results, we can prove different properties of sequences of random matrices which are invariant in law by conjugation by $\mathcal{G}(A)(N)$ (Theorems \ref{th:invarianceasymp}, \ref{th:invarianceasymp2}). 

\begin{result}
If a sequence of random matrices, which are invariant in law by conjugation by the family of groups $\mathcal{G}(A)$, converges in $\mathcal{A}$-distribution, it converges in $\mathcal{P}$-distribution. 
Besides, when the matrices are unitary invariant, we have simple formulas which link the limits of the $\mathcal{P}$-moments, the $\mathcal{P}$-exclusive moments and the free cumulants. 
\end{result}

These results allow us to give the asymptotic of the moments of the entries of $\mathcal{G}(\mathcal{A})$-invariant matrices (Theorem \ref{th:entrees})

\begin{result}
The first order asymptotic of the moments of the entries of $\mathcal{G}(\mathcal{A})$-invariant random matrices which converge in $\mathcal{A}$-distribution is known. 
\end{result}

Finally, we give a link between the notion of classical cumulants and finite-dimensional cumulants: this gives a Schur-Weyl duality interpretation of classical cumulants (Theorem \ref{cumulantmagique}).

\begin{result}
Classical cumulants can be obtained by computing the finite-dimensional cumulants of matrices with diagonal independent identically distributed entries. 
\end{result}

\subsubsection{Independence and invariance}
It remains to exhibit families of random matrices which are $\mathcal{A}$-invariant: this is done in Section \ref{sec:GAinvetindep}. 

\begin{result}
Let us consider two families of random matrices which are independent, converge in $\mathcal{A}$-distribution and such that one family is $\mathcal{G}(\mathcal{A})$-invariant. These two families are asymptotically $\mathcal{A}$-free (Theorem \ref{Lemain}). 
\end{result}

This theorem only applies when the second family is $\mathcal{G}(\mathcal{A})$-invariant: this is not the case for many families such as general Wigner matrices. This is why we introduce the notion of asymptotic strong $\mathcal{G}(\mathcal{A})$-invariance. 

\begin{result}
There exists a notion of asymptotic strong $\mathcal{G}(\mathcal{A})$-invariance such that if we consider two families of random matrices which are independent, converge in $\mathcal{P}$-distribution and such that one family is asymptotically strongly $\mathcal{G}(\mathcal{A})$-invariant, then the two families are asymptotically $\mathcal{A}$-free (Theorem \ref{Main1}).
\end{result} 

\subsubsection{L\'{e}vy processes}
In Section \ref{sec:Levy}, we study, with the tools developed in the article, sequences of matricial L\'{e}vy processes: one of the main result is Theorem \ref{convergencegenerale}.

\begin{result}
A sequence of $\mathcal{G}(\mathcal{A})$-invariant matricial additive (or multiplicative) L\'{e}vy processes converges in $\mathcal{P}$-distribution if and only if its generator converges in $\mathcal{A}$-distribution. There exists a simple criterion in order to know if the convergence holds in probability.
\end{result}

This result is generalized for the $*$-$\mathcal{P}$-distribution in Section \ref{sec:*convergence}. Using this general result, we retrieve the convergence in distribution, and even in $\mathcal{P}$-distribution of Hermitian and unitary Brownian motions. This leads us to a matricial Wick formula (Theorem \ref{Wickmat}).

\begin{result}
There exists a matricial Wick formula. For a Gaussian vector of symmetric of Hermitian random matrices, $\mathbb{E}[M_1\otimes ... \otimes M_k]$ can be written as a sum over special pairings.  
\end{result}

The results on Brownian motions allow us to develop an intuition which is used in order to prove the following result given in Theorem \ref{th:approximationgaussienne}. 

\begin{result}
The $\mathcal{P}$-distribution of any $\mathcal{P}$-Gaussian element can be approximated by sequences of random matrices which are explicitly given.
\end{result}

As a consequence of our general theorem about convergence in $\mathcal{P}$-distribution of matricial L\'{e}vy processes, we show the following result in Section \ref{sec:infinitelydivisible}.

\begin{result}
The proof of the approximations of free multiplicative and additive L\'{e}vy processes by matricial Hermitial and unitary L\'{e}vy processes given in  \cite{Flor}, \cite{Caba} and \cite{Cebron2} can be handled in a similar combinatorial way which also implies the convergence in probability without using any concentration of measure arguments. 
The same proofs allow us to generalize these results in order to have approximations by matricial symmetric and orthogonal L\'{e}vy processes. 
\end{result}
 
 \subsubsection{Algebraic fluctuations}
The notions of asymptotic $\mathcal{A}$-observables and $\mathcal{A}$-cumulants defined in this the previous sections are generalized in order to study more precisely the asymptotics of the $\mathcal{A}$-moments and $\mathcal{A}$-finite dimensional cumulants. We generalize the previous sections in this setting. 

\begin{result}
What is done for the first order asymptotic can be done with the higher order of fluctuations. In particular, a notion of $\mathcal{A}$-freeness of higher order is defined and it is shown that random matrices which are  $\mathcal{G}(\mathcal{A})$-invariant are asymptotically $\mathcal{A}$-free up to higher orders. 
\end{result}
 
As noticed by C. Male, some of the results can be seen as extensions in the setting of $\mathcal{A}$-tracial algebras of results in \cite{Camille}. The differences between the two independent approaches, of this article and \cite{Camille}, lie in the combinatorial objets studied (partitions here and graphs in \cite{Camille}), the use of dualities in the present article, and the observables which are mostly considered (mostly cumulants in this article and mostly connected exclusive moments in \cite{Camille}).  In the work in progress \cite{GabCebronGuill}, the link between the theory of $\mathcal{P}$-tracial algebras and the theory of traffics in \cite{Camille} is inverstigated. 

\subsection{Basic definitions on partitions}
\label{sec:basic}
Once the layout of the article given, we briefly review the basic definitions and results, proved in \cite{Gab1}, on partitions that we will need later. Let $k$ be a positive integer, $\mathcal{P}_k$ is the set of partitions of $\{1,...,k, 1',...,k'\}$. There exists a graphical notation for partitions as illustrated in Figure \ref{fig:1}. Let $p \in \mathcal{P}_k$, let us consider $k$ vertices in a top row, labelled from $1$ to $k$ from left to right and $k$ vertices in a bottom row, labelled from $1'$ to $k'$ from left to right. Any edge between two vertices means that the labels of the two vertices are in the same block of the partition $p$. Using this graphical point of view, the set of permutations of $k$ elements, namely $\mathfrak{S}_k$, is a subset of $\mathcal{P}_k$: if $ \sigma$ is a permutation, we associate the partition $\{\{i,\sigma(i)'\} | i \in \{1,...,k\}\}$. 

\begin{figure}
\centering
\resizebox{0.75\textwidth}{!}{
  \includegraphics[width=6pt]{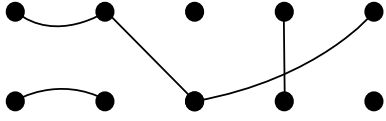}
}
\caption{The partition $\left\{ \{1',2' \}, \{1,2,3',5 \}, \{3 \}, \{4',4\},\{5' \} \right\}$.}
\label{fig:1}       
\end{figure}

 Let $p$ and $p'$ be two partitions in $\mathcal{P}_k$. The set $\mathcal{P}_k$ has some interesting structures we are going to explain: 
\begin{enumerate}
\item{\bf Transposition: } the partition $^{t}p$ is obtained by flipping along an horizontal line a diagram representing $p$. 
\item {\bf Order: } $p'$ is coarser that $p$, denoted $p \trianglelefteq p'$, if any block of $p$ is included in a block of $p'$. 
\item {\bf Supremum: } $p\vee p'$ is obtained by putting a diagram representing $p'$ over one representing $p$. 
\item {\bf Multiplication: } $p \circ p'$ is obtained by putting a diagram representing $p'$ above one representing $p$, identifying the lower vertices of $p'$ with the upper vertices of $p$, erasing the vertices in the middle row, keeping the edges obtained by concatenation of edges passing through the deleted vertices. It has this nice property: if $p \circ p' \in \mathfrak{S}_k$ then $p$ and $p'$ are in $\mathfrak{S}_k$. Doing so, we remove a certain number of connected components, number which is denoted by $\kappa(p,p')$.
\item {\bf A family of multiplications :} $\mathbb{C}[\mathcal{P}_k(N)]$ is the algebra in which the multiplication is given by $pp'=N^{\kappa(p,p')} p \circ p'.$
\item {\bf Neutral element: } the partition $\mathrm{id}_k = \{\{i,i'\}| i \in \{1,...,k\}\}$ is a neutral element for the multiplication $\circ$. Often we will denote it by $\mathrm{id}$ when there can not be any confusion. 
\item {\bf Height function: }  ${\sf nc}(p)$ is the number of blocks of $p$. 
\item {\bf Cycle: } a cycle is a block of $p\vee id$: ${\sf nc}(p\vee id)$ is the number cycles of $p$. A partition $p$ is {\em irreducible} if ${\sf nc}(p \vee id) = 1$. A partition which is, up to a permutation of the columns, of the form ${\mathrm{id}_k} \otimes p'$ with $p'$ an irreducible partition, is {\em weakly irreducible}. 
\item {\bf Representation: } there exists an important representation of $\mathbb{C}[\mathcal{P}_k(N)]$ on $(\mathbb{C}^{N})^{\otimes k}$. Let $(e_i)_{i=1}^{N}$ be the canonical basis of $\mathbb{C}^{N}$ and let $E_{i}^{j}$ be the matrix which sends $e_j$ on $e_i$ and any other element of the canonical basis on $0$. For any $I\!=\!(i_1,...,i_k,i_{1'},...,i_{k'})$ in $\{1,...,N\}^{2k}$, we define ${\sf Ker}(I)$ the partition such that two elements $a$ and $b$ of $\{1,...,k,1',...,k'\}$ is in a block of ${\sf Ker}(I)$ if and only if $i_a = i_b$. We define: 
\begin{align*}
\rho_N(p) = \sum_{I = (i_1,...,i_k,i_{1'},...,i_{k'}) \in \{1,...,N\}^{2k}| p \trianglelefteq {\sf Ker}(I) } E_{i_{1'}}^{i_1} \otimes ... \otimes E_{i_{k'}}^{i_k}. 
\end{align*}
The application $\rho_N$ is a representation of $\mathbb{C}[\mathcal{P}_k(N)]$.
\item {\bf Tensor product: } the partition $p \otimes p' \in \mathcal{P}_{2k}$ is obtained by putting a diagram representing $p'$ on the right of a diagram representing $p$. It satisfies the identity $\rho_{N}(p \otimes p') = \rho_N(p) \otimes \rho_N(p')$. 
\item {\bf Extraction: } the extraction of $p$ to a symmetric subset $I$ of $\{1,...,k,1',...,k'\}$ is obtained by erasing the vertices which are not in $I$ and relabelling the remaining vertices. It is denoted by $p_{I}$ (definition set before Definition $3.2$ of \cite{Gab1}). 
\item {\bf Set of factorizations: } for any set $X$, $\mathcal{E}(X)$ is the notation for the set of sets of $X$. For any $I \in \mathcal{E}(X)$,  let $I^{c}$ be the complement of $I$ in $X$. Tthe set of factorizations of $p$, namely $\mathfrak{F}_{2}(p)$, is the set of $(p_1,p_2,I) \in \mathcal{P} \times \mathcal{P} \times \mathcal{E}(\{1,...,k,1',...,k'\})$ such that $p_{ I} = p_1$, $p_{ I^{c}} = p_2$  and ${\sf nc}(p_1)+{\sf nc}(p_2) = {\sf nc}(p)$.
\item {\bf Left- and right- part: } if $k' \leq k$, the left-part of $p$, denoted by $p^{l}_{k'}$ is the extraction of $p$ to $\{1,...,k'\}$ and the right part of $p$, denoted by $p^{r}_{k'}$ is the extraction of $p$ to $\{k'+1,...,k\}$. 
\end{enumerate}

Besides, let us remark that the permutations in $\mathfrak{S}_k$ can be seen as a subset of $\mathcal{P}_k$. In particular, we denote by $(i,j)$ the transposition which changes $i$ with $j$:
\begin{align*}
(i,j) = \{\{i,j'\},\{i',j\}\} \cup \{ \{u,u' \}, u \in \{1,...,k\}\setminus \{i,j\}\}. 
\end{align*} 
We will also use the Weyl contraction $[i,j]$ defined as: 
\begin{align*}
[i,j] = \{\{i,j\},\{i',j'\}\} \cup \{ \{u,u' \}, u \in \{1,...,k\}\setminus \{i,j\}\}. 
\end{align*}

In \cite{Gab1}, we defined a distance (Definition $2.2$ of \cite{Gab1}) and a geodesic order on $\mathcal{P}_k$ (Definition $2.4$ of \cite{Gab1}). For any partitions $p$ and $p'$ in $\mathcal{P}_k$, the distance between $p$ and $p'$~is: 
\begin{align*}
d(p,p') = \frac{{\sf nc}(p) + {\sf nc}(p')}{2} - {\sf nc}(p\vee p'). 
\end{align*}
We wrote that $p \leq p'$ if $d(\mathrm{id}_k,p)+d(p,p') = d(\mathrm{id}_k, p')$: this defines an order on $\mathcal{P}_k$. The set $\{p| p\leq p'\}$ is denoted by $[\mathrm{id}_k, p']$. We denoted also by ${\sf d}(p',p)$ the defect of $p'$ not being in a geodesic between $\mathrm{id}_k$ and $p$: 
\begin{align*}
{\sf df}(p',p) = d(\mathrm{id}_k, p)-d(\mathrm{id}_k, p') -d(p',p).
\end{align*}

The order $\leq$ allowed us to define, in Definition 4.2 of \cite{Gab1}, a notion of $\mathcal{R}$-transform on $(\bigoplus_{k=0}^{\infty}\mathbb{C}[\mathcal{P}_k])^{*}$. This distance and the order was studied deeply in \cite{Gab1}: for example these notions allowed us to define a notion of Kreweras complement of $p$ in $p_0$, denoted by ${\sf K}_{p_0}(p)$ for any partition $p$ and $p_0$ in $\mathcal{P}_k$ (Definition 3.14 of \cite{Gab1}).

\section{$\mathcal{A}$-tracial algebras and $\mathcal{A}$-freeness}
\label{sec:Atracialalge}

\subsection{$\mathcal{P}$-tracial algebras}
Let us define the {\em insertion} of a partition $p \in \mathcal{P}_k$ in $\mathcal{P}_l$ where $l \geq k$. Let us consider $(i_1,...,i_k)$ a $k$-uple of distinct and increasing elements of $\{1,...,l\}$. The insertion ${\sf Ins}^{l}_{(i_1,...,i_k)}(p)$ is the partition: 
\begin{align*}
\left\{ \{ i_{j} | j \in \{1,...,k\} \cap b \} \cup \{ i'_{j} | j' \in \{1',...,k'\} \cap b\}  | b \in p \right\} \cup \{\{i,i'\} | i \notin \{i_1,...,i_k\}\}.
\end{align*}
\begin{definition}
A $\mathcal{P}$-tracial algebra $(A,(m_p)_{p \in \mathcal{P}})$ is the data of an algebra $A$ and a family $(m_p)_{p \in \mathcal{P}}$ such that for any $k > 0$ and any $p \in \mathcal{P}_k$, $m_p : A^{k} \to \mathbb{C}$ is a $k$-linear form and:
\begin{enumerate}
\item for any $k > 0$, any $p \in \mathcal{P}_k$, any permutation $\sigma \in \mathfrak{S}_k$, for any $a_1,..., a_k \in A$:
\begin{align*}
m_{p}(a_{1},...,a_{k}) = m_{\sigma \circ p \circ \sigma^{-1}}\left(a_{\sigma^{-1}(1)}, ..., a_{\sigma^{-1}(k)}\right), 
\end{align*}
\item for any $k > 0$, any $p \in \mathcal{P}_k$, any $(a_i^{(1)})_{i=1}^{n_1}$, ..., $(a_i^{(k)})_{i=1}^{n_k}$ where $a_{i}^{(j)} \in A$, if we denote ${\bold n} = \sum_{i=1}^{k} n_i$, ${\bold i} = (1,1+n_1,...,1+n_1+...n_{k-1})$ and $\sigma = (1,...,n_1)(n_1+1,...,n_1+n_2) ...(n_1+...+n_{k-1}+1, {\bold n})$:  
\begin{align*}
m_{p}\left(\prod_{i=1}^{n_1} a_{i}^{(1)}, ..., \prod_{i=1}^{n_k} a_{i}^{(k)} \right) = m_{ {\sf Ins}^{{\bold n}}_{\bold i}(p) \circ \sigma } (a^{(1)}_1,...,a^{(1)}_{n_1},...,a^{(k)}_1,...,a^{(k)}_{n_k}).
\end{align*}
The family $(m_p)_{p \in \mathcal{P}}$ is called the $\mathcal{P}$-moments. 
\end{enumerate}
\end{definition}

In particular, $m_{id_1}(.)$ define a tracial linear form on $A$ since, for any $a$ and $b$, by applying the axioms 2., 1. and again 2.: 
\begin{align*}
m_{id_1}(ab) = m_{(1,2)}(a,b) = m_{(1,2)}(b,a) = m_{id_{1}}(ba).
\end{align*}
We will always use the convention $m_{\emptyset} = 1$.

\begin{definition}
If for any partitions $p_1 \in \mathcal{P}_{k_1}$ and $p_2 \in \mathcal{P}_{k_2}$, any $(a_1,...,a_{k_1+k_2})$ in $A^{k_1+k_2}$, $m_{p_1\otimes p_2} (a_1,...,a_{k_1+k_2}) = m_{p_1}(a_1,...,a_{k_1}) m_{p_2} (a_{k_1+1}, ..., a_{k_1 + k_2})$, we say that $A$ is a deterministic $\mathcal{P}$-tracial algebra. If $A_1$ is a family of elements of $A$, we say that $A_1$ is deterministic if the algebra generated by $A_1$, namely $<A_1>$, is deterministic. 
\end{definition}

\begin{remark}
\label{rq:facto}
Actually, using the axioms of $\mathcal{P}$-tracial algebras, a family $A_1$ of elements of $A$ is deterministic if and only if the factorization property of the moments holds when one replaces $A$ by $A_1$. 
\end{remark}

\begin{exemple} \label{ex:Ptracial}Let us give some important examples of $\mathcal{P}$-tracial algebras. 
\begin{description}
\item[\textbf{1-Matrices:} ] Let $A = \mathcal{M}_N(\mathbb{C})$ and for any $k \geq 0$, any $p \in \mathcal{P}_k$ and any $(M_1,...,M_k) \in A^{k}$, we consider: 
\begin{align}
\label{eq:normalized}
m_{p}(M_1,...,M_k) = \frac{1}{N^{{\sf nc}(p \vee \mathrm{id}_k)}}{\sf Tr}^{k}\left[(M_1\otimes... \otimes M_k)\rho_{N}(^{t}p)\right]. 
\end{align} 
It is a deterministic $\mathcal{P}$-tracial algebra. 
\item[\textbf{2-Random variables I:} ] Let $A = L^{\infty^{-}}(\Omega, \mathcal{A}, \mathbb{P})$ be the algebra of random variables whose finite moments are all finite. We consider for any $k \geq 0$, any $p \in \mathcal{P}_k$ and any $(X_1,...,X_k) \in A^{k}$: 
\begin{align*}
m_{p}[X_1,...,X_k] = \mathbb{E}\left[\prod_{i=1}^{k} X_i\right].
\end{align*}
It is not a deterministic $\mathcal{P}$-tracial algebra. 
\item[\textbf{3-Random variables II:} ] Let $A = L^{\infty^{-}}(\Omega, \mathcal{A}, \mathbb{P})$ and for any $k \geq 0$, any $p \in \mathcal{P}_k$ and any $(X_1,...,X_k) \in A^{k}$, we consider: 
\begin{align*}
m_{p}[X_1,...,X_k] =\prod_{b \in p \vee \mathrm{id}_k} \mathbb{E}\left[\prod_{i\in b \cap \{1,...,k\}} X_{i}\right].
\end{align*}
It is a deterministic $\mathcal{P}$-tracial algebra. 
\item[\textbf{3-Random matrices:} ]Let $A = L^{\infty^{-}}(\Omega, \mathcal{A},\mathbb{P}) \otimes \mathcal{M}_N(\mathbb{C})$ and for any $k \geq 0$, any $p \in \mathcal{P}_k$ and any $(M_1,...,M_k) \in A^{k}$, we consider: 
\begin{align}
\label{eq:normalized2}
\mathbb{E}m_{p}(M_1,...,M_k) =  \frac{1}{N^{{\sf nc}(p \vee \mathrm{id}_k)}}\mathbb{E}\left[{\sf Tr}^{k}\left[(M_1\otimes... \otimes M_k)\rho_{N}(^{t}p)\right]\right]. 
\end{align} 
It is not a deterministic $\mathcal{P}$-tracial algebra. 
\item[\textbf{4-$\mathcal{P}$-tracial distribution algebra: }] Let $(B,(m^{B}_{p})_{p \in \mathcal{P}})$ be a $\mathcal{P}$-tracial algebra, let $n \geq 0$  and let $(b_1,...,b_n) \in B^{n}$. We can associate another $\mathcal{P}$-tracial algebra given by $A = \mathbb{C}\{X_1,...,X_n\}$ (the non-commutative polynomial algebra) and, for any $k \geq 0$, any $p \in \mathcal{P}_k$ and any $(P_1,...,P_k) \in A^{k}$: 
\begin{align*}
m^{(b_1,...,b_n)}_{p}(P_1,...,P_k) = m^{B}_{p}(P_1(b_1,...,b_n), ..., P_k(b_1,...,b_n)). 
\end{align*}
The $\mathcal{P}$-tracial algebra $\left(\mathbb{C}\{X_1,...,X_n\}, (m^{(b_1,...,b_n)}_{p})_{p \in \mathcal{P}}\right)$ is called the $\mathcal{P}$-distribu\-tion of $(b_1,...,b_n)$. It is a deterministic $\mathcal{P}$-tracial algebra if and only if the algebra generated by $(b_1,...,b_n)$ endowed with $(m^{B}_{p})_{p \in \mathcal{P}}$ is a deterministic $\mathcal{P}$-tracial algebra. 
\end{description}
\end{exemple}

This last example allows us to define a notion of convergence in $\mathcal{P}$-distribution: we will use the same notations in the following definition. 
 
\begin{definition}
\label{def:conv}
For any $ N \geq 0$, let $\left(A_N, (m_{p}^{N})_{p \in \mathcal{P}}\right)$ be a  $\mathcal{P}$-tracial algebra and let $(a^{N}_i)_{i \in I}$ be elements of $A_N$. The sequence $(a^{N}_i)_{i \in I}$ converges in $\mathcal{P}$-distribution if for any $n \geq 0$, $(i_1,...,i_n) \in I^n$ and any $p \in \mathcal{P}$, $m_{p}^{(a_{i_1}^{N}, ..., a_{i_n}^{N})}$ converges as $N$ goes to infinity. 
\end{definition}

Let $(A, (m_{p})_{p \in \mathcal{P}})$ be a $\mathcal{P}$-tracial algebra. We will define two useful triangular transformations of the linear forms. Let us recall some notions from \cite{Gab1}. 

\begin{definition}
Let $k \geq 0$, $p$ and $p'$ in $\mathcal{P}_k$, the partition $p'$ is {\em coarser-compatible} with $p$ if $p'$ is coarser than $p$ and ${\sf nc}(p' \vee \mathrm{id}_k)={\sf nc}(p \vee \mathrm{id}_k)$. If so, we denote $p' \dashv p$. 

The partition $p'$ if {\em finer-compatible} with $p$ if $p'$ is finer than $p$ and ${\sf nc}(p') - {\sf nc}(p' \vee \mathrm{id}_k) ={\sf nc}(p) - {\sf nc}(p \vee \mathrm{id}_k) $. If so, we denote $p' \sqsupset p$. 
\end{definition}

The relations $\dashv$ and $\sqsupset$ are orders on $\mathcal{P}$ and any strictly decreasing chain is finite. 

\begin{definition}
The  $\mathcal{P}$-exclusive moments $(m_{p^{c}})_{p \in \mathcal{P}}$ are the unique multi-linear forms such that, for any $k \geq 0$, any $p \in \mathcal{P}_k$, any $(a_1,...,a_k) \in A^{k}$:
\begin{align*}
m_{p}(a_1,...,a_k) = \sum_{p' \in \mathcal{P}_k | p' \dashv p} m_{p'^{c}}(a_1,...,a_k). 
\end{align*}
The $\mathcal{P}$-cumulant linear forms $(\kappa_{p})_{p \in \mathcal{P}}$ are the unique multi-linear forms such that,  for any $k \geq 0$, any $p \in \mathcal{P}_k$, any $(a_1,...,a_k) \in \mathcal{A}^{k}$: 
\begin{align*}
m_{p}(a_1,...,a_k) = \sum_{p' \in \mathcal{P}_k | p' \leq p} \kappa_{p'}(a_1,...,a_k). 
\end{align*}
\end{definition}

In \cite{Gab1}, Equation (22), we proved the following result. 

\begin{proposition}
\label{prop:lienmomexclucumu}
For any $k \geq 0$, any $p \in \mathcal{P}_k$, any $(a_1,...,a_k) \in A^{k}$, 
\begin{align*}
m_{p^{c}} (a_1,...,a_k) = \sum_{ p' \in \mathcal{P}_k | p' \sqsupset p } \kappa_{p'}(a_1,...,a_k). 
\end{align*}
\end{proposition}

Using the cumulants, we can give a new characterization of the fact that a $\mathcal{P}$-tracial algebra is deterministic or not. The proof is similar to the proof of Proposition $4.5$ of \cite{Gab1}, thus we will not provide a proof for it. 
\begin{lemma}
\label{lemme:detercumu}
The $\mathcal{P}$-tracial algebra $A$ is deterministic if and only if for any $k_1$, $k_2 \geq 0$, any partitions $p_1 \in \mathcal{P}_{k_1}$ and $p_2 \in \mathcal{P}_{k_2}$, any $(a_1,...,a_{k_1+k_2}) \in A^{k_1+k_2}$, $$\kappa_{p_1\otimes p_2} (a_1,...,a_{k_1+k_2}) = \kappa_{p_1}(a_1,...,a_{k_1}) \kappa_{p_2} (a_{k_1+1}, ..., a_{k_1 + k_2}).$$ 
\end{lemma}

\begin{remark}
\label{rq:facto2}
Again, using the axioms of $\mathcal{P}$-tracial algebras, a family $A_1$ of elements of $A$ is deterministic if and only if the factorization property of the cumulants holds when one replaces $A$ by $A_1$.
\end{remark}

\subsection{The $\mathcal{R}$-transform} 
If $a \in A$, we have seen that we can define its $\mathcal{P}$-tracial distribution algebra as a $\mathcal{P}$-tracial algebra built over $\mathbb{C}[X]$. Because of the first axiom of $\mathcal{P}$-tracial algebras, we can also encode this distribution as a linear form in $\left(\bigoplus_{k=0}^{\infty} \mathbb{C}[\mathcal{P}_k / \mathfrak{S}_k]\right)^{*}$. The linear form which sends $p$ on $m_p(a)$ will be called the $\mathcal{P}$-distribution of $a$ and it is denoted by $\mathcal{M}(a)$. Using the second axiom of $\mathcal{P}$-tracial algebras, we see that the $\mathcal{P}$-distribution of $a$ is enough in order to recover the $\mathcal{P}$-tracial distribution algebra of $a$. 

\begin{definition}
The $\mathcal{R}$-transform of $a$ is the linear form in $\left(\bigoplus_{k=0}^{\infty} \mathbb{C}[\mathcal{P}_k / \mathfrak{S}_k]\right)^{*}$ which sends $p$ on $\kappa_{p}(a)$. 
\end{definition}

Recall the notion of $\mathcal{R}$-tranform defined on $\left(\bigoplus_{k=0}^{\infty} \mathbb{C}[\mathcal{P}_k / \mathfrak{S}_k]\right)^{*}$ in Definition $4.2$ of \cite{Gab1}. By definition, the $\mathcal{P}$-distribution and the $\mathcal{R}$-transform of $a$ are linked by the equality $\mathcal{R}(a) = \mathcal{R}[\mathcal{M}(a)]$. Besides, using Lemma \ref{lemme:detercumu}, we see that $a \in A$ is deterministic if and only if $\mathcal{R}(a)$ is a character.

\subsection{$\mathcal{P}$-freeness}
We define the notion of $\mathcal{P}$-freeness which will play the role of the usual freeness but for $\mathcal{P}$-tracial algebras. Let $A_1$ and $A_2$ be two sub-algebras of $A$.

\begin{definition}
Let $(a_1,...,a_k)$ be a $k$-tuple of elements of  $A_1\cup A_2$. Let $p$ be a partition in $\mathcal{P}_k$. We say that $p$ is compatible with $(a_1,...,a_k)$ if for any $i,j \in \{1,...,k\}$, if $a_i \in A_1$ and $a_j \in A_2$, then $i$ and $j$ are in two different blocks of $p$. 
\end{definition}

Let us remark that for a partition, being compatible with a $k$-tuple actually depends on the choice of $A_1$ and $A_2$: these choices will always be explicit. Let us remark that for any element $a \in A$, if $A_1 = <a> = A_2$, no partition $p \in \mathcal{P}_k$ except $\mathrm{id}_k$ is compatible with $(a,...,a)$.

The notion of $\mathcal{P}$-freeness is given in terms of vanishing mixed cumulants. 

\begin{definition}
The two sub-algebras $A_1$ and $A_2$ are $\mathcal{P}$-free if the two following conditions hold: 
\begin{description}
\item[$\bullet$ \textbf{compatibility condition:}] for any integer $k$, any $k$-uple $(a_1,...,a_{k}) \in (A_1\cup A_2)^{k}$ and any partition $p \in \mathcal{P}_k$, if $p$ is not compatible with $(a_1,...,a_{k})$: 
\begin{align*}
\kappa_{p}(a_1,...,a_{k}) = 0, 
\end{align*}
\item[$\bullet$ \textbf{factorization property:}] for any $k_1$, $k_2 \geq 0$,  any $p_1 \in \mathcal{P}_{k_1}$, any $p_2 \in \mathcal{P}_{k_2}$, any $(a_1,...,a_{k_1}) \in A_{1}^{k_1}$ and any $(a_{k_1+1},...,a_{k_1+k_2}) \in A_2^{k_2}$, 
\begin{align*}
\kappa_{p_1\otimes p_2}(a_1,...,a_{k}) = \kappa_{p_1}(a_1,...,a_{k_1})\kappa_{p_2}(a_{k_1+1},...,a_{k_1+k_2}).
\end{align*} 
\end{description}
Let $A_1$ and $A_2$ be two families of $A$, they are $\mathcal{P}$-free if the algebras they generate are $\mathcal{P}$-free.
\end{definition}

If $A_1$ and $A_2$ are $\mathcal{P}$-free, using the restrictions of $(m_{p})_{p \in \mathcal{P}}$ on $A_1$ and $A_2$, we can compute the restriction of $(m_{p})_{p \in \mathcal{P}}$ on the algebra $<A_1,A_2>$. Indeed, the second axiom (compatibility with the product) in the definition of a $\mathcal{P}$-tracial algebra, shows that we only need to compute the restriction of $(m_{p})_{p \in \mathcal{P}}$ on  $A_1 \cup A_2$. But the data of $(m_p)_{p \in \mathcal{P}}$ is equivalent to the data of $(\kappa_{p})_{p \in \mathcal{P}}$. The $\mathcal{P}$-freeness allows us to compute the restriction of $(\kappa_{p})_{p \in \mathcal{P}}$ on $A_1\cup A_2$ knowing its restriction on $A_1$ and $A_2$. 

\begin{remark}
Again, using the axioms of $\mathcal{P}$-tracial algebras, if $A_1$ and $A_2$ are two families of $A$, they are $\mathcal{P}$-free if and only if the compatibility condition and the factorization property hold for $A_1$ and $A_2$: we do not need to consider the algebras $<A_1>$ and $<A_2>$.
\end{remark}

Using Lemma \ref{lemme:detercumu}, we can already state the following lemma whose proof is straightfoward. 
\begin{lemma}
\label{lem:deterministic}
If $A_1$ and $A_2$ are two deterministic $\mathcal{P}$-free sub-algebras of $A$, then the algebra $<\!A_1, A_2\!>$ is a deterministic sub-algebra. 
\end{lemma}

We can give another characterization of $\mathcal{P}$-freeness using the exclusive moments. Recall the notion of left- and right-part of a partition ((13) in Section \ref{sec:basic}).

\begin{theorem}
\label{theorem:autreformulation}
Let $A_1$, $A_2$ be two sub-algebras of the $\mathcal{P}$-tracial algebra $(A,(m_p)_{p \in \mathcal{P}})$. They are $\mathcal{P}$-free if and only if for any $k_1, k_2 >0$,  any $p \in \mathcal{P}_{k_1+k_2}$, any $(a_i)_{i=1}^{k_1} \in A_1$, any $(a_i)_{i=k_1+1}^{k_1+k_2}\in A_2$, 
\begin{align}
\label{eq:freeexclusive}
m_{p^{c}}\left[(a_i)_{i=1}^{k_1+k_2}\right] = \delta_{p_{k_1}^{l} \otimes p_{k_1}^{r} \sqsupset p} m_{(p_{k_1}^{l})^{c}} \left[(a_i)_{i=1}^{k_1}\right] m_{(p_{k_1}^{r})^{c}} \left[(a_{i})_{i=k_1+1}^{k_1+k_2}\right], 
\end{align}
where we recall that:
\begin{align*}
p_{k_1}^{l} &= \left\{ b \cap \{1,...,k_1,1',...,k_1'\} | b \in p \right\},\\
p_{k_1}^{r}&= \left\{ b \cap \{k_1+1,...,k_1+k_2,k_1',...,k_1'+k_2'\} | b \in p \right\}. 
\end{align*}
\end{theorem}

\begin{proof}
Let us remark that, for any $(a_i)_{i=1}^{k_1} \in A_1$, any $(a_i)_{i=k_1+1}^{k_1+k_2}\in A_2$, the r.h.s. of Equation (\ref{eq:freeexclusive}) is equal to: 
\begin{align*}
\delta_{p_{k_1}^{l} \otimes p_{k_1}^{r} \sqsupset p} \sum_{p_1' \sqsupset p_{k_1}^{l}, p_2' \sqsupset p_{k_1}^{r}} \kappa_{p'_1}  \left[(a_i)_{i=1}^{k_1}\right]  \kappa_{p'_2} \left[(a_i)_{i=k_1+1}^{k_1+k_2}\right]. 
\end{align*}
Using Proposition $3.2$ of \cite{Gab1}, $p_1' \sqsupset p_{k_1}^{l}, p_2' \sqsupset p_{k_1}^{r}$ and $p_{k_1}^{l} \otimes p_{k_1}^{r} \sqsupset p$ if and only if $p_1' \otimes p_2' \sqsupset p$: the r.h.s. of (\ref{eq:freeexclusive}) is then equal to: 
\begin{align*}
\sum_{p_1' \otimes p_2' \sqsupset p} \kappa_{p'_1}  \left[(a_i)_{i=1}^{k_1}\right]  \kappa_{p'_2} \left[(a_i)_{i=k_1+1}^{k_1+k_2}\right]. 
\end{align*}
On the other way, the l.h.s. of (\ref{eq:freeexclusive}) is equal to: 
\begin{align*}
\sum_{p' \sqsupset p} \kappa_{p'}\left[(a_{i})_{i=1}^{k_1+k_2}\right].
\end{align*}
Thus, we must prove that $A_1$ and $A_2$ are $\mathcal{P}$-free if and only if for any $k_1$, $k_2>0$, any $p \in \mathcal{P}_{k_1+k_2}$, any  $(a_i)_{i=1}^{k_1} \in A_1$, any $(a_i)_{i=k_1+1}^{k_1+k_2}\in A_2$, 
\begin{align*}
\sum_{p_1' \otimes p_2' \sqsupset p} \kappa_{p'_1}  \left[(a_i)_{i=1}^{k_1}\right]  \kappa_{p'_2} \left[(a_i)_{i=k_1+1}^{k_1+k_2}\right] = \sum_{p' \sqsupset p} \kappa_{p'}\left[(a_{i})_{i=1}^{k_1+k_2}\right].
\end{align*}
This equivalence is quite straightfoward. 
\qed \end{proof}

\begin{remark}
Again, using the axioms of $\mathcal{P}$-tracial algebras, the same theorem holds when $A_1$ and $A_2$ are only two families of $A$. 
\end{remark}

\subsection{Sum and product}
Let $A_1$, $A_2$ be two  $\mathcal{P}$-free sub-algebras of the $\mathcal{P}$-tracial algebra $(A,(m_p)_{p \in \mathcal{P}})$. Recall the notion of factorizations ((12) in Section \ref{sec:basic}) and Kreweras complement (Definition 3.14 in \cite{Gab1}). The following theorem generalizes Theorem 5.2.2 of \cite{Speic}.

\begin{theorem}
\label{th:calculloi}
Let $k \geq 0$, let $(a_i)_{i=1}^{n} \in A_1^{k}$  and $(b_i)_{i=1}^{n} \in A_2^{k}$, let $p \in \mathcal{P}_k$: 
\begin{align*}
\kappa_{p}(a_1+b_1,...,a_k+b_k) &= \sum_{(p_1,p_2,I) \in \mathfrak{F}_{2}(p)} \kappa_{p_1}((a_{i})_{i \in I}) \kappa_{p_2} ((b_i)_{i \in \{1,...,k\}\setminus I}), \\
\kappa_{p}(a_1b_1,...,a_kb_k)&= \sum_{ p_1\prec p, p_2 \in{\sf K}_p(p_1)} \kappa_{p_1}((a_i)_{i=1}^{k})\kappa_{p_2}((b_i)_{i=1}^{k}), \\
m_{p}(a_1b_1,...,a_kb_k) &= \sum_{p_1 \leq p} \kappa_{p_1}((b_i)_{i=1}^{k}) m_{{ }^{t}p_1\circ p} ((b_i)_{i=1}^{k}). 
\end{align*}  
\end{theorem}

\begin{proof}
The first equality is a consequence of the multi-linearity property of the $\mathcal{P}$-cumulant form and the $\mathcal{P}$-freeness of  $(a_i)_{i=1}^{n} \in A_1^{k}$  and $(b_i)_{i=1}^{n} \in A_2^{k}$. 

 Using the definition of the $\mathcal{P}$-cumulant forms, the second equality is equivalent to the fact that for any $p \in \mathcal{P}_k$: 
 \begin{align}
 \label{eq:aprouver}
 m_{p}(a_1b_1,...,a_kb_k) = \sum_{p' \leq p} \sum_{p_1, p_2 | p_1\circ p_2 = p', p_2 \in {\sf K}_{p_1\circ p_2}(p_1)} \kappa_{p_1} ((a_i)_{i=1}^{k})\kappa_{p_2}((b_i)_{i=1}^{k})
 \end{align}
By Theorem $3.5$ of \cite{Gab1}, $p_1\circ p_2 \leq p$ and $p_2 \in {\sf K}_{p_1\circ p_2}(p_1)$ if and only if $p_1\otimes p_2 \leq (p \otimes \mathrm{id}_{k}) \tau$ where $\tau = (1,k+1)(2,k+2)...(k,2k) \in \mathfrak{S}_{2k}$. Thus, the r.h.s. of Equation (\ref{eq:aprouver}) is equal to:
\begin{align*}
\sum_{p_1\otimes p_2 \leq (p \otimes \mathrm{id}_{k}) \tau}  \kappa_{p_1} ((a_i)_{i=1}^{k})\kappa_{p_2}((b_i)_{i=1}^{k}). 
\end{align*}
Using the $\mathcal{P}$-freeness hypothesis, for any $p_1$ and $p_2$ in $\mathcal{P}_k$,   $\kappa_{p_1} ((a_i)_{i=1}^{k})\kappa_{p_2}((b_i)_{i=1}^{k})$ is equal to $\kappa_{p_1\otimes p_2}(a_1,...,a_k,b_1,...,b_k)$. Besides, if $p' \in \mathcal{P}_k$ is not compatible with $(a_1,...,a_k,b_1,...,b_k)$, which means that there exists no partitions $p_1$ and $p_2$ such that $p_1 \otimes p_2 = p'$, then $\kappa_{p'}((a_1,...,a_k,b_1,...,b_k)) = 0$. Thus, the r.h.s. of Equation (\ref{eq:aprouver}) is equal to: 
\begin{align*}
\sum_{p' \leq (p \otimes \mathrm{id}_{k}) \tau} \kappa_{p'}(a_1,...,a_k,b_1,...,b_k)\! &=\! m_{(p \otimes \mathrm{id}_{k}) \tau} (a_1,...,a_k,b_1,...,b_k)\! \\
&= m_{{\sf Ins}_{(1,3,...,2k-1)}^{2k}(p)}(a_1,b_1,...,a_k,b_k)\\
&=\! m_{p}(a_1b_1,...,a_kb_k), 
\end{align*}
where, for the last two equalities, we applied the axioms of $\mathcal{P}$-tracial algebras. 

It remains to prove the last equality which is a consequence of the last equality and the equivalence of the first and second assertions in Theorem $3.5$ of \cite{Gab1}. Indeed, 
\begin{align*}
m_{p}(a_1,...,a_k,b_1,...,b_k) &= \sum_{p' \leq p}\kappa_{p'}(a_1,...,a_k,b_1,...,b_k) \\
&=\sum_{p' \leq p, p_1\prec p', p_2 \in {\sf K}_{p'}(p_1)} \kappa_{p_1}(a_1,...,a_k) \kappa_{p_2}(b_1,...,b_k) \\
&= \sum_{p_1\leq p,\ p_2 \leq ^{t}p_1 \circ p} \kappa_{p_1}(a_1,...,a_k) \kappa_{p_2}(b_1,...,b_k) \\
&=  \sum_{p_1 \leq p} \kappa_{p_1}(a_1,...,a_k) m_{{ }^{t}p_1 \circ p} (b_1,...,b_k).
\end{align*}
This ends the proof. 
\qed \end{proof}

Recall Section $4.2$ of \cite{Gab1} where we defined the two convolutions $\boxplus$ and $\boxtimes$ on the set of linear forms $\left(\bigoplus_{k=0}^{\infty} \mathbb{C}[\mathcal{P}_k / \mathfrak{S}_k]\right)^{*}$. 

\begin{theorem}
\label{th:Rtransfoprod}
Let $a$ and $b$ be two $\mathcal{P}$-free elements of $A$: 
\begin{align*}
\mathcal{R}[a+b] = \mathcal{R}[a] \boxplus \mathcal{R}[b] \text{\ \ \  and \ \ \ }
\mathcal{R}[ab] = \mathcal{R}[a] \boxtimes  \mathcal{R}[b]. 
\end{align*}
\end{theorem}

\subsection{Semi-groups}
Let $(A, (m_p)_{p \in \mathcal{P}})$ be a $\mathcal{P}$-tracial algebra and $(a_t)_{t \geq 0}$ be a family of elements of $A$ such that $a_0$ is the unit of $A$.  

\begin{definition}
The process $(a_t)_{t \geq 0}$ is an additive (or $\boxplus$-) $\mathcal{P}$-L\'{e}vy process if: 
\begin{enumerate}
\item the $\mathcal{P}$-distribution of $a_t$ is continuous in $t$, 
\item for any $s \geq t \geq 0$, the $\mathcal{P}$-distribution of $a_s - a_t $ only depends on $s-t$, 
\item for any $t \geq 0$, $(a_s-a_t)_{s \geq t}$ is $\mathcal{P}$-free from $(a_u)_{t \geq u}$. 
\end{enumerate}
The process $(a_t)_{t \geq 0}$ is a multiplicative (or $\boxtimes$-) $\mathcal{P}$-L\'{e}vy process if for any $t > 0$, $a_t$ is invertible in $A$ and the three conditions hold when one replaces $a_s-a_t$ by $a_sa_t^{-1}$. 
\end{definition}
From now on, $\boxdot$ stands either for $\boxplus$ or $\boxtimes$. Let $(a_t)_{t \geq 0}$ be a $\boxdot$-$\mathcal{P}$-L\'{e}vy process. The following lemma is a consequence of the definitions and Theorem \ref{th:Rtransfoprod}. 

\begin{lemma}
\label{lemme:semi}
The family $(\mathcal{R}[a_t])_{t \geq 0}$ is a continuous semi-group for the $\boxdot$-convolution. 
\end{lemma}

\begin{definition}
The infinitesimal $\boxdot$-transform of $(a_t)_{t \geq 0}$, denoted by $\mathfrak{r}[(a_t)_{t \geq 0}]$, is the element of $\left(\bigoplus_{k=0}^{\infty} \mathbb{C}[\mathcal{P}_k / \mathfrak{S}_k]\right)^{*}$ defined by $\frac{d}{dt}_{|t = 0} \mathcal{R}[a_t]$. 
\end{definition}

By definition, for any $t \geq 0$, $\mathcal{R}[a_t] = e^{\boxdot t \mathfrak{r}[(a_s)_{s \geq 0}] }$. Recall Theorem $4.1$ of \cite{Gab1}: it implies, with Lemmas \ref{lemme:detercumu} and \ref{lem:deterministic}, the following theorem. 

\begin{theorem}
\label{th:deterministicsemigroup}
If $\mathfrak{r}[(a_t)_{t \geq 0}]$ is a $\boxdot$-infinitesimal character, then the algebra $<(a_t)_{t \geq 0}>$ is deterministic. 
\end{theorem}

\subsection{Central limit theorem}
Let $a$ be an element of $A$. For any $i \geq 0$, we denote by $\mathcal{R}_i(a)$ the restriction of $\mathcal{R}(a)$ to $\mathbb{C}[\mathcal{P}_i/\mathfrak{S}_i]$, seen as an element of $(\bigoplus_{k=0}^{\infty} \mathbb{C}[\mathcal{P}_k/\mathfrak{S}_k])^{*}$. This means that for any partition $p \in \mathcal{P}_k$ with $k \neq i$, $\mathcal{R}_{i}(a)(p) = 0$ and for any $p \in \mathcal{P}_i$, $\mathcal{R}_{i}(a)(p) = \mathcal{R}(a)(p)$.

\begin{definition}
\label{def:Pgaussian}
The element $a$ is a centered $\mathcal{P}$-Gaussian if $\mathcal{R}(a) = e^{\boxplus \mathcal{R}_2(a)}$. 
\end{definition}

In particular, for any $k \geq 3$, any irreducible $p \in \mathcal{P}_{k}$, $(\mathcal{R}(a))(p) = 0$.

\begin{theorem}
\label{th:tcl}
Let $(a_n)_{n=1}^{\infty}$ be a sequence  in $A$ of $\mathcal{P}$-free elements which have the same $\mathcal{P}$-distribution and such that $\mathcal{R}_1(a_1)=0$. Then $\frac{1}{\sqrt{n}} \sum_{i=1}^{n} a_i$ converges in $\mathcal{P}$-distribution to a centered $\mathcal{P}$-Gaussian element whose $\mathcal{R}$-transform is $e^{\boxplus \mathcal{R}_2(a_1)}$. 
\end{theorem}

\begin{proof}
By multi-linearity, for any integers $n, k$ and any $p \in \mathcal{P}_k$: 
\begin{align*}
\left(\mathcal{R}\left[\frac{1}{\sqrt{n}} \sum_{i=1}^{n} a_i\right]\right)(p) = \frac{1}{n^{\frac{k}{2}}} \sum_{(i_{j})_{j=1}^{k} \in \{1,...,n\}^{k}} \kappa_{p}(a_{i_1}, ...,a_{i_k}).
\end{align*}
Using the $\mathcal{P}$-freeness of $(a_n)_{n=1}^{\infty}$, we can only keep the $k$-tuples $(i_{j})_{j=1}^{k}$ such that $p$ is compatible with $(a_{i_1}, ...,a_{i_k})$. Using also the second axiom in the definition of freeness and the fact that every $a_i$ has the same $\mathcal{P}$-distribution: 
\begin{align*}
\left(\mathcal{R}\left[\frac{1}{\sqrt{n}} \sum_{i=1}^{n} a_i\right]\right)(p) =\frac{1}{n^{\frac{k}{2}}} \sum_{f : p\vee \mathrm{id}_k \to \{1,...,n\} }\prod_{i=1}^{n} \kappa_{p_{| \underset{{b \in f^{-1}(i)}}{\cup} b }}(a_1,...,a_1), 
\end{align*}
where $p \vee {\sf id}_{k}$ is the set composed by the cycles of $p$, $p_{| \underset{{b \in f^{-1}(i)}}{\cup} b }$ is the extraction of $p$ to $ \underset{{b \in f^{-1}(i)}}{\cup} b $ and where we took the convention $(\mathcal{R}(a))(\emptyset) = 1$.

For any function $f : p \vee {\sf id}_{k} \to  \{1,...,n\}$, we define ${\sf Ker}(f)$ as the partition of $p \vee {\sf id}_{k} $ equal to $\{ f^{-1}(i) | i\in \{1,...,n\}\}$. Let us denote by $\mathcal{P}(p \vee {\sf id}_{k})$ the set of partitions of $p \vee {\sf id}_{k}$. For any partition $\pi \in\mathcal{P}(p \vee {\sf id}_{k})$, there exist approximatively $n^{{\sf nc}(\pi)}$ functions $f$ such that ${\sf Ker}(f) = \pi$: 

\begin{align}
\label{eq:tcl}
\left(\mathcal{R}\left[\frac{1}{\sqrt{n}} \sum_{i=1}^{n} a_i\right]\right)(p) \simeq \frac{1}{n^{\frac{k}{2}}} \sum_{\pi \in \mathcal{P}(p \vee {\sf id}_{k})} n^{{\sf nc}(\pi)} \prod_{b \in \pi} \kappa_{p_{| b}} (a_1,...,a_1).
\end{align}

Since $\mathcal{R}_1(a_1) = 0$, for any $\pi \in \mathcal{P}(p \vee {\sf id}_{k})$ such that ${\sf nc}(\pi) > \frac{k}{2}$, we have that $\prod_{b \in \pi} \kappa_{p_{| b}} (a_1,...,a_1) =~0$. Indeed, if ${\sf nc}(\pi) > \frac{k}{2}$, one block of $\pi$ must be a singleton which contrains a cycle of size $1$ of $p$ and we recall that $\mathcal{R}_{1}(a_1)=0$. Thus we can impose in the r.h.s. of Equation (\ref{eq:tcl}) the condition ${\sf nc}(\pi) \leq \frac{k}{2}$: the l.h.s. of Equation \ref{eq:tcl} converges when $n$ goes to infinity. 

If ${\sf nc}(\pi)<\frac{k}{2}$, the term associated with $\pi$ disappears when $n$ goes to infinity and it only remains partitions $\pi \in  \mathcal{P}(p \vee {\sf id}_{k})$ such that ${\sf nc}(\pi) = \frac{k}{2}$ and such that none of the block of $\pi$ is a singleton which contains a cycle of lenght $1$ of $p$. In order that the limit of the l.h.s. of (\ref{eq:tcl})  converges to a non zero value, $p$ must have some cycles of size $2$ and an even number of cycles of size $1$: the blocks of $\pi$ contain either two cycles of size $1$ or a cycle of size $2$ of $p$. This implies, by sending $n$ to infinity in $(\ref{eq:tcl})$ that: 
\begin{align*}
\lim_{n \to \infty }\left(\mathcal{R}\left[\frac{1}{\sqrt{n}} \sum_{i=1}^{n} a_i\right]\right)(p) = e^{\boxplus \mathcal{R}_2(a_1)}(p), 
\end{align*}
which is the equality we wanted to prove. 
\qed \end{proof}

This theorem applied to the $\mathcal{P}$-tracial algebras given in ``Random variables I and II'' of Example \ref{ex:Ptracial} implies the usual central limit theorem for random variables with bounded moments. Besides, applied to the forthcoming $\mathfrak{S}$-tracial algebra defined in Example \ref{ex:stracialalgebra}, this theorem implies the usual free central limit theorem. Indeed, this theorem apply in the setting of $\mathfrak{S}$-tracial algebras and we will see that the notion of $\mathfrak{S}$-freeness in deterministic  $\mathfrak{S}$-tracial algebras is equivalent to the notion of freeness of Voiculescu. For any set  of partitions $\mathcal{A}$ in $\mathcal{P}$, $\mathcal{B}$, $\mathcal{S}$, $\mathcal{H}$ or $\mathcal{B}s$, let us define the notion of $\mathcal{A}$-tracial algebra. 

\subsection{$\mathcal{A}$-tracial algebras}
For the sake of simplicity, we explained first the notion of $\mathcal{P}$-tracial algebras without explaining that one can define for any $\mathcal{A} \in \{ \mathcal{P}$, $\mathcal{B}$, $\mathcal{S}$, $\mathcal{H}, \mathcal{B}s\}$ a notion of $\mathcal{A}$-tracial algebra. In this section, we explain how to do so and how to generalize the notions we defined for $\mathcal{P}$-tracial algebras. This leads us to study the notion of $\mathcal{G}(\mathcal{A})$-invariant family and the links between the different notions of $\mathcal{A}$-freeness.

\subsubsection{Basic definitions}
Let  us choose a set of partitions $\mathcal{A}$ in $\{ \mathcal{P}$, $\mathcal{B}$, $\mathcal{S}$, $\mathcal{H}, \mathcal{B}s\}$ where these last sets are described in the up-coming Table $1$ in Section \ref{duali}. From now on, the letter $\mathcal{A}$ will only be used for such set of partitions. The group $\mathcal{G}(\mathcal{A})$ is the group which corresponds to $\mathcal{A}$ via the duality given in Table $1$. The notions of (deterministic or not) $\mathcal{A}$-tracial algebra, convergence in $\mathcal{A}$-distribution, $\mathcal{A}$-cumulant linear forms, $\mathcal{R}_{\mathcal{A}}$-transform, $\mathcal{A}$-freeness and $\boxdot$-$\mathcal{A}$-L\'evy processes are defined with a straightforward generalization of the case $\mathcal{A}=\mathcal{P}$. The Lemmas \ref{lemme:detercumu}, \ref{lem:deterministic}, \ref{lemme:semi} and the Theorems \ref{th:calculloi}, \ref{th:Rtransfoprod}, \ref{th:deterministicsemigroup} and \ref{th:tcl} are still valid in their straightforward modified formulation. In order to define the $\mathcal{A}$-exclusive moments, we will have to be careful about how we define them: this will be explained later. In a $\mathcal{A}$-tracial algebra, the $\mathcal{A}$-cumulant linear forms will be denoted by $(\kappa_{p}^{\mathcal{A}})_{p \in \mathcal{A}}$.

\begin{exemple}
In a $\mathfrak{S}$-tracial algebra $(A,(m_\sigma)_{\sigma \in \mathfrak{S}})$, the $\mathfrak{S}$-cumulant linear forms $\kappa^{\mathfrak{S}}_{p}$ are the unique multi-linear forms such that, for any $k \geq 0$, any $p \in \mathcal{P}_k$, any $(a_1,...,a_k) \in A_k$, $m_{p} (a_1,...,a_k) = \sum_{p' \in \mathfrak{S}_k | p' \leq p} \kappa_{p'}^{\mathfrak{S}}(a_1,...,a_k)$. 
\end{exemple}

Let us give a natural example of $\mathfrak{S}$-tracial algebra.

\begin{exemple}
\label{ex:stracialalgebra}
Any algebra endowed with a tracial form, namely $(A,\phi)$, can be endowed with a structure of $\mathfrak{S}$-tracial algebra: for any $k\geq 0$, any $\sigma \in \mathfrak{S}_k$ and any $(a_1,...,a_k) \in A^{k}$, we consider: 
\begin{align*}
m_{\sigma}(a_1,...,a_k) = \prod_{c \in \sigma\vee \mathrm{id}_k} \phi\left[\prod_{i \in c \cap \{1,...,k\}} a_{i} \right], 
\end{align*}
where the product is taken according to the order of the cycle of $\sigma$ associated with $c$. This is a deterministic $\mathfrak{S}$-tracial algebra called the $\mathfrak{S}$-tracial algebra associated with $(A, \phi)$.
\end{exemple}

Let $(A,(m_{\sigma})_{\sigma \in \mathfrak{S}})$ be a {\em deterministic} $\mathfrak{S}$-tracial algebra. It is actually the $\mathfrak{S}$-tracial algebra associated with $(A, m_{{\mathrm{id}}_{1}})$. Besides, the notion of $\mathfrak{S}$-freeness for $(A,(m_{\sigma})_{\sigma \in \mathfrak{S}})$ is equivalent to the notion of Voiculescu's freeness for $(A, m_{{\mathrm{id}}_{1}})$: this is a consequence of the bijection between $[\mathrm{id}_k, (1,...,k)] \cap \mathfrak{S}_k$ and ${\sf NC}_{k}$, and the fact that the freeness in free probability theory can be stated as a property of vanishing cumulants (Speicher's condition, Theorem 11.16 in \cite{Speic}).

\subsubsection{Natural restriction, extension of structure, $\mathcal{G}(\mathcal{A})$-invariance and $\mathcal{A}$-exclusive moments}
\label{sec:restiexten}
 Let $\mathcal{A}_1$ and $\mathcal{A}_2$ be sets of partitions in $\{ \mathcal{P}, \mathcal{B}, \mathfrak{S}, \mathcal{H}, \mathcal{B}s\}$ such that $\mathcal{A}_2 \subset \mathcal{A}_1$. There exists a natural application which allows us to forget some of the structure. Indeed, any $\mathcal{A}_{1}$-tracial algebra can be seen as a $\mathcal{A}_2$-tracial algebra by forgetting the $m_p$ such that $p \notin \mathcal{A}_2$. Thus, on any $\mathcal{A}_1$-tracial algebra, we can consider the $\mathcal{A}_2$-cumulant linear forms, denoted by $(\kappa_p^{\mathcal{A}_2})_{p \in \mathcal{A}_2}$. This implies that we can talk about $\mathcal{A}_2$-freeness for elements which are in a $\mathcal{A}_1$-tracial algebra. In Section \ref{sec:libertes}, we will study the links between these notions of freeness. 
 
\begin{remark}
Recall the definitions in Section 4.4.1 of \cite{Gab1} and recall the notion of cumulants defined in the same paper. Let $(A,(m_p))_{p \in \mathcal{A}_1}$ be a $\mathcal{A}_1$-tracial algebra. The notions of $\mathcal{A}_1$- and $\mathcal{A}_2$- cumulants are linked by the following equality: 
\begin{align*}
\mathcal{C}^{\kappa}_{\mathcal{A}_2}\bigg(\sum_{p \in \mathcal{A}_1}\kappa^{\mathcal{A}_1}_p (a_1,...,a_k) p^{*}\bigg) = \sum_{p \in \mathcal{A}_2}\kappa^{\mathcal{A}_2}_p (a_1,...,a_k) p^{*}, 
\end{align*}
where $p^{*}$ is the dual form of $p$. Thus, using Proposition 5.2 of \cite{Gab1}, one can compute the $\mathcal{A}_2$-cumulants forms on $(a_1,...,a_k)$ by considering the limits, as $N$ goes to infinity, of the cumulants of $\int_{\mathcal{G}(\mathcal{A}_2)(N)}g^{\otimes k} \rho_N(E_N) (g^{*})^{\otimes k} dg,$ where: $$E_N \!=\! \sum\limits_{p' \in (\mathcal{A}_1)_k}\!\! \kappa^{\mathcal{A}_1}_{p'} (a_1,...,a_k) N^{{\sf nc}(p' \vee \mathrm{id}_k) - {\sf nc}(p')}p'.$$ 
\end{remark}

 One can also define an extension of structure. Let $(A,(m_{p})_{p \in \mathcal{A}_{2}})$ be a $\mathcal{A}_2$-tracial algebra.

\begin{definition}
 The natural extension of $(A,(m_{p})_{p \in \mathcal{A}_2})$ as a $\mathcal{A}_1$-tracial algebra is given by the fact that for any $k \geq 0$, $p \in \mathcal{A}_1$ and any $(a_1,...,a_k) \in A^{k}$, 
\begin{align*}
m_{p} (a_1,...,a_k)= \sum_{ p' \in \mathcal{A}_2 | p' \leq p} \kappa^{\mathcal{A}_2}_{p'}(a_1,...,a_k). 
\end{align*}
\end{definition}

\begin{exemple}
The natural extension of a $\mathfrak{S}$-, respectively $\mathcal{B}$-, tracial algebra to a $\mathcal{P}$-tracial algebra has a simpler form. Indeed, using Lemma 3.2 and Definition 3.6 of \cite{Gab1} and the notations in that article, for any $p \in \mathcal{P}_k$, $m_{p} = \sum_{p' \in \overline{\mathfrak{S}_k}  | p' \dashv p} \kappa^{\mathfrak{S}}_{{\sf Mb}(p')}$, respectively $m_{p} = \sum_{p' \in \overline{\mathcal{B}_k} | p' \dashv p}\kappa^{\mathcal{B}}_{{\sf Mb}(p')}$. 
\end{exemple}

The definition of natural extension is set such that for any $k \geq 0$, any $p \in \mathcal{A}_1$ of length $k$ and any $(a_1,...,a_k)$, $\kappa^{\mathcal{A}_1}_p(a_1,...,a_k)$ computed using the natural extension structure is equal to $\delta_{p \in \mathcal{A}_2} \kappa^{\mathcal{A}_2}_p(a_1,...,a_k)$. Thus, we already see that two families $A_1$ and $A_2$ of $(A,(m_p)_{p \in \mathcal{A}_2})$ are $\mathcal{A}_2$-free if and only if they are $\mathcal{A}_1$-free for the natural extension of $A$ as a $\mathcal{A}_1$-tracial algebra. This natural extension leads us to the following definition of $\mathcal{G}(\mathcal{A})$-invariance.

\begin{definition}
\label{def:GAinva}
Let $(A,(m_p)_{p \in \mathcal{A}_1})$ be a $\mathcal{A}_1$-tracial algebra and $A_1$ be a family of elements of $A$. The family $A_1$ is $\mathcal{G}(\mathcal{A}_2)$-invariant if for any $k \geq 0$, any $p \in \mathcal{A}_1$ of length $k$ and any $(a_1,...,a_k) \in A_{1}^{k}$, $\kappa^{\mathcal{A}_1}_p(a_1,...,a_k) = \delta_{p \in \mathcal{A}_2} \kappa^{\mathcal{A}_2}_p(a_1,...,a_k)$, where $(\kappa^{\mathcal{A}_2}_p)_{p \in \mathcal{A}_2}$ are the $\mathcal{A}_2$-cumulants forms of   $(A,(m_p)_{p \in \mathcal{A}_2})$.
\end{definition}

With this definition, we see that any $\mathcal{A}_2$-tracial algebra is $\mathcal{G}(\mathcal{A}_2)$-invariant in its natural extension as a $\mathcal{A}_1$-tracial algebra. 

Let $(A,(m_p)_{p \in \mathcal{A}_1})$ be a $\mathcal{A}_1$-tracial algebra. Using the axioms of $\mathcal{A}_1$-tracial algebras, one can see that a family $A_1$ of elements of $A$ is $\mathcal{G}(\mathcal{A}_2)$-invariant if and only if $<A_1>$ is $\mathcal{G}(\mathcal{A}_2)$-invariant. Besides, if $\mathcal{A}_1 = \mathcal{P}$ and $\mathcal{A}_2$ is $\mathfrak{S}$ or $\mathcal{B}$, using a slight generalization of the second assertion in Lemma $4.2$ of \cite{Gab1}, we have the characterization of $\mathcal{G}(\mathcal{A}_2)$ invariance. Recall Definition 3.6 in \cite{Gab1}.

\begin{theorem} 
\label{th:invariancecaract}
Let us suppose that $\mathcal{A}_1 = \mathcal{P}$ and $\mathcal{A}_2$ is $\mathfrak{S}$, respectively $\mathcal{B}$, then a family $A_1$ of elements of $A$ is $U$-invariant, resp. $O$-invariant, if  and only if for any $k \geq 0$, any $p \in \mathcal{P}_k$ and any $(a_1,...,a_k) \in A_{1}^{k}$, $m_{p^{c}}(a_1,...,a_k)$ is equal to $\delta_{p \in \overline{\mathfrak{S}_k}} m_{({\sf Mb}(p))^{c}}(a_1,...,a_k)$, respectively $\delta_{p \in \overline{\mathcal{B}_k}} m_{({\sf Mb}(p))^{c}}(a_1,...,a_k)$. 
\end{theorem}

Using the notion of natural extension, we can now define the notion of $\mathcal{A}_1$-exclusive moments. Recall that $(A,(m_p)_{p \in \mathcal{A}_1})$ is a $\mathcal{A}_1$-tracial algebra. 

\begin{definition}
The $\mathcal{A}_1$-exclusive moments $(m^{\mathcal{A}_1}_{p^{c}})_{p \in \mathcal{P}}$ are the $\mathcal{P}$-exclusive moments of the natural extension of $(A,(m_p)_{p \in \mathcal{A}_1})$ as a $\mathcal{P}$-tracial algebra. 
\end{definition}
In particular, using Proposition \ref{prop:lienmomexclucumu}, the $\mathcal{A}_1$-exclusive moments are characterized by the fact that for any $k \geq 0$, any $p \in \mathcal{P}_k$ and any $(a_1,...,a_k) \in A^{k}$, 
\begin{align*}
m^{\mathcal{A}_1}_{p^{c}}(a_1,...,a_k) = \sum_{p' \in (\mathcal{A}_1)_k | p' \sqsupset p } \kappa_{p'}^{\mathcal{A}_1} (a_1,...,a_k).
\end{align*}
With this definition, one can see that for any $k \geq 0$, any $p \in \mathcal{A}_1$ of length $k$ and any $(a_1,...,a_k) \in A^k$, 
\begin{align*}
m_{p} (a_1,...,a_k) = \sum_{p' \in \mathcal{P}_k | p' \dashv p } m_{p^{c}}^{\mathcal{A}_1}(a_1,...,a_k).
\end{align*}

With this definition of $\mathcal{A}$-exclusive moments, Theorem \ref{theorem:autreformulation} now becomes valid in the general setting of $\mathcal{A}$-tracial algebras.

\subsubsection{The different notions of $\mathcal{A}$-freeness}
 \label{sec:libertes}
Recall that $\mathcal{A}_1$ and $\mathcal{A}_2$ are two sets of partitions in $\{ \mathcal{P}, \mathcal{B}, \mathfrak{S}, \mathcal{H}, \mathcal{B}s\}$ such that $\mathcal{A}_2 \subset \mathcal{A}_1$. Let $(A,(m_p)_{p \in \mathcal{A}_1})$ be a $\mathcal{A}_1$-tracial algebra. 

\begin{theorem}
\label{th:invarianceetliberte}
Let $A_1$ and $A_2$ be families of $A$ such that $A_1\cup A_2$ is $\mathcal{G}(\mathcal{A}_2)$-invariant. The two families $A_1$ and $A_2$ are $\mathcal{A}_2$-free if and only if they are $\mathcal{A}_1$-free. 

Besides, if only $A_1$ is $\mathcal{G}(\mathcal{A}_2)$-invariant and if $A_1$ and $A_2$ are $\mathcal{A}_1$-free then they are $\mathcal{A}_2$-free. 
\end{theorem}

\begin{proof}
Let us prove the first assertion. If $A_1\cup A_2$ is $\mathcal{G}(\mathcal{A}_2)$-invariant, then $<A_1,A_2>$ is  $\mathcal{G}(\mathcal{A}_2)$-invariant. This implies that $(<A_1,A_2>,(m_{p})_{p \in \mathcal{A}_1})$ is the natural extension of $(<A_1,A_2>,(m_{p})_{p \in \mathcal{A}_2})$ as a $\mathcal{A}_1$-tracial algebra. Using our discussion about natural extension, $A_1$ and $A_2$ are $\mathcal{A}_2$-free in $(<A_1,A_2>,(m_{p})_{p \in \mathcal{A}_2})$ if and only if they are $\mathcal{A}_1$-free in $(<A_1,A_2>,(m_{p})_{p \in \mathcal{A}_1})$. This proves the first assertion. 

Let us prove the second assertion: let us suppose only that $A_1$ is $\mathcal{G}(\mathcal{A}_2)$-invariant and that $A_1$ and $A_2$ are $\mathcal{A}_1$-free. Let $k$, $k'$ be two integers, let $(a_1,...,a_k) \in A_1^{k}$ and $(b_1,...,b_{k'}) \in A_2^{k'}$ and let $p \in (\mathcal{A}_2)_{k+k'}$. The easiest way to compute the cumulant $\kappa_{p}^{\mathcal{A}_2}(a_1,...,a_k,b_1,...,b_{k'})$ is to use Proposition $5.2$ of \cite{Gab1} which asserts that it is the limit of the $p$-cumulant of $E^{\mathcal{G}(\mathcal{A}_2)}_N = \int_{\mathcal{G}(\mathcal{A}_2)(N)} g^{\otimes k+k'} \rho_{N}(E_N) (g^{*})^{\otimes k+k'} dg$, where   
\begin{align*}
E_N = \sum_{p' \in (\mathcal{A}_1)_{k+k'}} \kappa_{p'}^{\mathcal{A}_1}(a_1,...,a_k,b_1,...,b_{k'}) N^{{\sf nc}(p' \vee \mathrm{id}_{k+k'}) - {\sf nc}(p')} p'. 
\end{align*}
Using the $\mathcal{G}(\mathcal{A}_2)$-invariance of $A_1$ and the $\mathcal{A}_1$-freeness of $A_1$ and $A_2$, for any integer $N$, $E_N$ is equal to: 
\begin{align*}
 \sum_{p_1 \in (\mathcal{A}_2)_k, p_2 \in (\mathcal{A}_1)_{k'} }\kappa^{\mathcal{A}_2}_{p_1}(a_1,...,a_k) \kappa_{p_2}^{\mathcal{A}_1}&(b_1,...,b_{k'}) \\ &N^{{\sf nc}(p_1\vee \mathrm{id}_k) - {\sf nc}(p_1) + {\sf nc}(p_2\vee \mathrm{id}_{k'})- {\sf nc}(p_2)}p_1\!\otimes\! p_2, 
\end{align*}
and thus, using the duality between $\mathcal{A}_2$ and $\mathcal{G}(\mathcal{A}_2)$, $E_N^{\mathcal{G}(\mathcal{A}_2)}$ is equal to: 
\begin{align*}
 \sum_{p_1 \in (\mathcal{A}_2)_k} \kappa_{p_1}^{\mathcal{A}_2} (a_1,...,a_k)  N^{{\sf nc}(p_1\vee \mathrm{id}_k) - {\sf nc}(p_1)} p_1 \otimes \int_{\mathcal{G}(\mathcal{A}_2)(N)} g^{\otimes k'} \rho_N(F_N)(g^{*})^{\otimes k'} dg 
\end{align*}
with $F_N = \sum_{p_2 \in (\mathcal{A}_1)_{k'}} \kappa_{p_2}^{\mathcal{A}_1} (b_1,...,b_{k'}) N^{{\sf nc}(p_2\vee \mathrm{id}_{k'}) - {\sf nc}(p_2)} p_2$. Again, using Proposition $5.2$ of \cite{Gab1}, the limit of the $p'$-cumulant of $\int_{\mathcal{G}(\mathcal{A}_2)(N)} g^{\otimes k'} \rho_N(F_N)(g^{*})^{\otimes k'} dg $ is equal to $\kappa_{p'}^{\mathcal{A}_2}(b_1,...,b_{k'})$. Thus, 
\begin{align*}
\kappa_{p}^{\mathcal{A}_2}(a_1,...,a_k,b_1,...,b_{k'}) \!=\! \delta_{\exists p_1\in (\mathcal{A}_2)_k, p_2 \in (\mathcal{A}_2)_{k'} | p=p_1\otimes p_2} \kappa_{p_1}^{\mathcal{A}_2}(a_1,...,a_k) \kappa_{p_2}^{\mathcal{A}_2}(b_1,...,b_{k'}).
\end{align*}
This proves that $A_1$ and $A_2$ are $\mathcal{A}_2$-free. 
 \end{proof}
As noticed by C. Male, this last following theorem is a generalization of the ``rigidity of freeness" theorem of \cite{Camille}.

For sake of simplicity, from now on let us suppose that $\mathcal{A}_1 = \mathcal{P}$, generalizations to other cases are straightfoward. We mostly focused on $\mathcal{A}_2 \in \{ \mathcal{P}, \mathcal{B}, \mathfrak{S}\}$ and a little on $\mathcal{A}_2 = \mathcal{B}s$. Recall that ${\sf 0}_2$ is the partition $\{\{1,2,1',2'\}\}$.

\begin{theorem}
\label{th:critere}
Let $A_1$ and $A_2$ be two families of $A$. We have the following equivalence or implications: 
\begin{enumerate}
\item if $A_1\cup A_2$ is deterministic, the $\mathfrak{S}$-freeness of $A_1$ and $A_2$ is equivalent to the freeness of $A_1$ and $A_2$ in Voiculescu sense for the linear form $m_{\mathrm{id}_1}$, 
\item if $A_1$ and $A_2$ are $\mathcal{B}$- or $\mathcal{B}s$-free then they are $\mathfrak{S}$-free, 
\end{enumerate}
and the following negative assertions hold: 
\begin{enumerate}
\item if $A_1$ and $A_2$ are $\mathcal{P}$-free and if there exists $(a_1,a_2) \in A_1^{2}$ and $(b_1,b_2) \in A_2^{2}$ such that: 
\begin{align*}
\kappa^{\mathcal{P}}_{{\sf 0}_2}(a_1,a_2) \kappa^{\mathcal{P}}_{{\sf 0}_{2}}(b_1,b_2) \neq 0, 
\end{align*}
then $A_1$ and $A_2$ are not $\mathfrak{S}$-free. In particular, $\mathcal{P}$-freeness does not imply the $\mathfrak{S}$-freeness.
\item the $\mathcal{P}$-freeness of $A_1$ and $A_2$ does not generally imply the $\mathcal{B}$-freeness of $A_1$ and~$A_2$, 
\item the $\mathfrak{S}$-freeness of $A_1$ and $A_2$ does not generally imply the $\mathcal{B}$-freeness of $A_1$ and~$A_2$.
\end{enumerate}
\end{theorem}

A summary of these assertions can be found in Diagram 1. Let us remark that the first negative assertion is inspired by a result of C. Male, namely the first point of Corollary $3.5$ of \cite{Camille}. The proof that we give here differs from the one of C. Male as it is based on the cumulants we introduced in this article.

\begin{proof}[Proof of Theorem \ref{th:critere}]
The first assertion was already explained in the previous sections. Let us prove the second assertion. At first, let us remark that for any integer $k$, any permutation $\sigma \in \mathfrak{S}_k$, there exists no element $p \in \mathcal{B}s_k \cup \mathcal{B}_k$ such that $p \leq \sigma$. In Lemma $3.1$ of \cite{Gab1}, we proved this assertion when one replaces $\mathcal{B}s_k \cup \mathcal{B}_k$ by $\mathcal{B}_k$: this proof is easily generalized to $\mathcal{B}s_k \cup \mathcal{B}_k$. A consequence is that for any $k \geq 0$, $\sigma \in \mathfrak{S}_k$, $\kappa_{\sigma}^{\mathcal{B}s} = \kappa_{\sigma}^{\mathcal{B}}= \kappa_{\sigma}^{\mathfrak{S}}$. Then the vanishing of either the mixed  $\mathcal{B}s$-cumulants or the mixed  $\mathcal{B}$-cumulants implies the vanishing of the mixed $\mathfrak{S}$-cumulants. 

Now, let us consider the negative assertions. For the first one, let $(a_1,a_2) \in A_1^{2}$ and $(b_1,b_2) \in A_2^{2}$ such that $\kappa_{{\sf 0}_2}(a_1,a_2) \kappa_{{\sf 0}_{2}}(b_1,b_2) \neq 0.$ Using the axiom of $\mathcal{P}$-tracial algebras, Theorem \ref{th:calculloi} and the fact that $[\mathrm{id}_k, (1,2)] = \{\mathrm{id}_k, {\sf 0}_{2}, (1,2)\}$, $m_{(1,2,3,4)}(a_1,b_1,a_2,b_2)$ is equal to $m_{(1,2)}(a_1b_1,a_2b_2)$ which is equal to $\kappa_{\mathrm{id}_2}^{\mathcal{P}}(a_1,a_2) m_{(1,2)}(b_1,b_2) + \kappa_{{\sf 0}_2}^{\mathcal{P}}(a_1,a_2)$ $m_{{\sf 0}_2}(b_1,b_2) + \kappa_{(1,2)}^{\mathcal{P}}(a_1,a_2) m_{{\mathrm{id}}_{2}}(b_1,b_2)$ or: 
\begin{align*}
&m_{\mathrm{id}_2}(a_1,a_2) m_{(1,2)}(b_1,b_2) + m_{(1,2)}(a_1,a_2) m_{\mathrm{id}_2}(b_1,b_2) - m_{\mathrm{id}_2}(a_1,a_2) m_{\mathrm{id}_2} (b_1,b_2) \\ 
& \ \ \ \ \ \ \ \ \ \ \ \ \ \ \ \ \ \ \ \ \ \ \ \ \ \ \ \ \ \ \ \ \ \ \ \ \ \ \ \ \ \ \ \ \ \ \ \ \ \ \ \ \ \ \ \ \ \ \ \ \ \ \ \ \ \ \ \ \ \ \ \ \ \ \ \ \ \ \ \ \ \ \ \ \ + \kappa_{{\sf 0}_2}^{\mathcal{P}}(a_1,a_2)\kappa_{{\sf 0}_2}^{\mathcal{P}}(b_1,b_2).
\end{align*}
If $(a_1,a_2)$ and $(b_1,b_2)$ are $\mathfrak{S}$-free, using similar arguments, $m_{(1,2,3,4)}(a_1,b_1,a_2,b_2)$ is equal to $m_{\mathrm{id}_2}\!(a_1,a_2) m_{(1,2)}\!(b_1,b_2) \!+\! m_{(1,2)}\!(a_1,a_2) m_{\mathrm{id}_2}\!(b_1,b_2) \!-\! m_{\mathrm{id}_2}\!(a_1,a_2) m_{\mathrm{id}_2}\!(b_1,b_2)$.
Since $\kappa_{{\sf 0}_2}^{\mathcal{P}}(a_1,a_2)\kappa_{{\sf 0}_2}^{\mathcal{P}}(b_1,b_2) \neq 0$, $(a_1,a_2)$ and $(b_1,b_2)$ are not $\mathfrak{S}$-free and thus $A_1$ and $A_2$ are not $\mathfrak{S}$-free. 

For the second negative assertion, let us remark that if $\mathcal{P}$-freeness was implying $\mathcal{B}$-freeness, because of the third implication we proved, $\mathcal{P}$-freeness would imply $\mathfrak{S}$-freeness: we just proved that this was not true. 

For the third implication, let us suppose that $\mathfrak{S}$-freeness implies $\mathcal{B}$-freeness. In next section, we will see that if $a \in A$ is $U$-invariant, $a$ and $^{t}a$ are $\mathfrak{S}$-free. Thus, $a$ and $^{t}a$ would be $\mathcal{B}$-free: this would imply that $a$ is $\mathfrak{S}$-free with itself: this is not generally the case. 
\qed \end{proof}

\begin{remark}
Since ${\sf 0}_{2} \in \mathcal{H}_2$, the first negative assertion in Theorem \ref{th:critere} holds if $A_1$ and $A_2$ are two $\mathcal{H}$-free sub-algebras of $A$: $\mathcal{H}$-freeness does not imply the $\mathfrak{S}$-freeness. 
\end{remark}

\hspace{-10pt}\label{diagram}
\begin{align*}
 \xymatrix{
 & \textbf{Voiculescu's freeness}\ar@/^{1pc}/[dd] \\
 & \text{ }\\
 &\mathfrak{S}-\textbf{freeness} \ar@/^{-1pc}/[uu]^{\text{(a) under deterministic asumption }}\ar@/^1.5pc/[dd]^{\text{(b) under } U- \text{invariance}}  && \mathcal{B}s-{\textbf{freeness}}\ar@/^{-1.5pc}/[ll]_{}\\
 & \text{ }\\
 & \mathcal{B}-\textbf{freeness} \ar@/^1.5pc/[uu]^{\text{(d)}}\ar@/^1.5pc/[dd]^{\text{(c) under } O- \text{invariance}}\ar@/^{-1.5pc}/[uu] \\
& \text{ }\\
& \mathcal{P}-\textbf{freeness}\ar@/^{-1.5pc}/[uu]
}
\end{align*}
\begin{center}
Diagram $1$. The different notions of $Voiculescu$-, $\mathfrak{S}$-, $\mathcal{B}$-, $\mathcal{B}s$- and $\mathcal{P}$-freeness.
\end{center}

\subsubsection{Voiculescu freeness and the transpose operation}

In this section, we show that the notion of transpose operation allows us to construct families which are free in the sense of Voiculescu. This gives a generalization, in an algebraic setting, of a result  of Mingo and Popa in \cite{Mingo} which asserts that unitary invariant random matrices are free with their transpose. It illustrates also the use of the $\mathcal{B}$-cumulants in order to prove some interesting results about freeness. Let $k_1<k$ be integers and let $p$ be in $\mathcal{P}_k$.

\begin{definition}
\label{def:Sk}
The partition ${\sf S}_{k_1}(p)$ is the partition obtained by permuting $i$ with $i'$ for any $i \in \{ k_1+1,...,k\}$ in the definition of $p$. 
\end{definition}

For any $\mathcal{A}$ except $\mathfrak{S}$, for any integers $k_1<k$, $\mathcal{A}_k$ is stable by the operation ${\sf S}_{k_1}$. Until the end of this section, let us consider such $\mathcal{A}$.

\begin{definition}
A transpose operation on a $\mathcal{A}$-tracial algebra $(A,(m_p)_{p \in \mathcal{A}})$ is a linear operation $^{t} : a \mapsto a^{t}$ such that for any integers $k_1<k$, any $p \in \mathcal{A}_k$, any $(a_1,...,a_k) \in A^{k}$, $m_{p} (a_1,...,a_{k_1}, { }^{t}a_{k_1+1}, ..., { }^{t}a_k) = m_{{\sf S}_{k_1}(p)} (a_1,...,a_{k})$. 
\end{definition}

For example, the matrix transposition is a transposition operation on the $\mathcal{P}$-tracial algebras of matrices or random matrices defined in Example \ref{eq:normalized}. For any family of elements $A'$, we denote by $^{t}A'$ the set $\{^{t}a | a \in A'\}$. In order to prove next theorem, we need the following definition. Recall the notion of left- and right- parts of a partition defined in Definition 3.12 of \cite{Gab1}.
\begin{definition}
\label{def:cut}
Let $k, k'$ be two integers and $\sigma \in \mathfrak{S}_{k+k'}$. The left-part of $\sigma$, $\sigma_{k}^{l}$, is composed of blocs of size $2$ except possibly two blocs of size $1$. Let us glue the blocs of size $1$ together and we get a permutation that we call ${\sf C}_{k}^{l}(\sigma) \in \mathfrak{S}_k$. By doing the same with the right part, we get another permutation ${\sf C}_{k}^{r}(\sigma) \in\mathfrak{S}_{k'}$.
\end{definition}

\begin{theorem}
\label{th:transpofree}
Let $(A,(m_p)_{p \in \mathcal{A}})$ be a $\mathcal{A}$-tracial algebra endowed with a transpose operation denoted by $\text{ }^{t}$. For any $A_1, A_2 \subset A$ such that $A_1 \cup A_2$ is $U$-invariant and deterministic, $A_1$ and $^{t}\!A_{2}$ are free in the sense of Voiculescu.
\end{theorem}

\begin{proof}
Let $A_1, A_2 \subset A$ such that $A_1\cup A_2$ is $U$-invariant. Let $k$, $k'$ be two integers, $(a_1,...,a_k)$ be in $A_1^{k}$, $(b_1,...,b_{k'})$ be in $A_2^{k'}$ and $\sigma \in \mathfrak{S}_{k+k'}$. Using the definitions of a transpose operation and $\mathfrak{S}$-cumulants: 
\begin{align*}
m_{\sigma}(a_1,...,a_k, ^{t}\!b_1,..., ^{t}\!b_{k'}) &= m_{{\sf S}_{k}(\sigma)}(a_1,...,a_k,b_1,...,b_{k'})\\
&=\sum_{p \in \mathcal{B}_{k+k'} | p \leq {\sf S}_{k}(\sigma)} \kappa^{\mathcal{B}}_{p} (a_1,...,a_k,b_1,...,b_{k'}).
\end{align*}
Since $A_1 \cup A_2$ is $U-$invariant, for any $p \in \mathcal{A}$, when restricted to $A_1 \cup A_2$, $\kappa^{\mathcal{A}}_p$ is equal to $\delta_{p \in \mathfrak{S}} \kappa^{\mathfrak{S}}_p$. Actually this implies that for any $p \in \mathcal{B}$, when restricted to $A_1 \cup A_2$,  $\kappa^{\mathcal{B}}_{p} $ is equal to $\delta_{ p \in \mathfrak{S}} \kappa^{\mathfrak{S}}_p$. Thus, 
\begin{align*}
m_{\sigma}(a_1,...,a_k, ^{t}\!b_1,..., ^{t}\!b_{k'}) =\sum_{\sigma' \in \mathfrak{S}_{k+k'} | \sigma' \leq {\sf S}_{k}(\sigma)} \kappa^{\mathfrak{S}}_{\sigma'} (a_1,...,a_k,b_1,...,b_{k'}).
\end{align*}
Recall Definition \ref{def:cut}. Using Lemma \ref{lem:geodesic}, $\{ \sigma' \in \mathfrak{S}_{k+k'} | \sigma' \leq  {\sf S}_{k}(\sigma) \} $ is equal to $\{\sigma_1' \otimes \sigma_2' | \sigma_1' \in \mathfrak{S}_{k}, \sigma_2' \in \mathfrak{S}_{k'}, \sigma_1' \leq {\sf C}_k^{l}(\sigma), \sigma_2' \leq \text{ }^{t}{\sf C}_k^{r}(\sigma)\}$. Thus, using the fact that $A_1 \cup A_2$ is deterministic and Lemma \ref{lemme:detercumu}:
\begin{align*}
m_{\sigma}(a_1,...,a_k, ^{t}\!b_1,..., ^{t}\!b_{k'}) &=\sum_{\sigma'_1, \sigma'_2 | \sigma'_1 \leq  {\sf C}_k^{l}(\sigma), \sigma_2' \leq\text{ }^{t}{\sf C}_k^{r}(\sigma) } \kappa^{\mathfrak{S}}_{\sigma'_1 \otimes \sigma'_2} (a_1,...,a_k,b_1,...,b_{k'}) \\
&= \sum_{\sigma'_1, \sigma'_2 | \sigma'_1 \leq  {\sf C}_k^{l}(\sigma), \sigma_2' \leq\text{ }^{t}{\sf C}_k^{r}(\sigma) } \kappa^{\mathfrak{S}}_{\sigma'_1} (a_1,...,a_k)  \kappa^{\mathfrak{S}}_{\sigma'_2} (b_1,...,b_{k'}) \\
&= \sum_{\sigma'_1, \sigma'_2 | \sigma'_1 \leq  {\sf C}_k^{l}(\sigma), \sigma_2' \leq{\sf C}_k^{r}(\sigma) } \kappa^{\mathfrak{S}}_{\sigma'_1} (a_1,...,a_k)  \kappa^{\mathfrak{S}}_{\sigma'_2} (\!\text{ } ^{t}b_1,..., ^{t}\!b_{k'}).
\end{align*}
Using Lemma \ref{lem:geo2},  $\{\sigma'_1 \otimes \sigma'_2 | \sigma'_1 \leq  {\sf C}_k^{l}(\sigma), \sigma'_2 \leq  {\sf C}_k^{r}(\sigma) \} = \{ \sigma_1' \otimes \sigma_2' \leq \sigma | \sigma_1' \in \mathfrak{S}_{k}, \sigma_2' \in \mathfrak{S}_{k'}\}$. Thus: 
\begin{align*}
m_{\sigma}(a_1,...,a_k, ^{t}\!b_1,..., ^{t}\!b_{k'}) = \sum_{\sigma_1' \in \mathfrak{S}_k, \sigma_2' \in \mathfrak{S}_{k'} | \sigma_1' \otimes \sigma_2' \leq \sigma } \kappa^{\mathfrak{S}}_{\sigma'_1} (a_1,...,a_k)  \kappa^{\mathfrak{S}}_{\sigma'_2} (\!\text{ } ^{t}b_1,..., ^{t}\!b_{k'}).
\end{align*}
Since this equality holds for any $\sigma \in \mathfrak{S}_{k+k'}$, any choice of $(a_1,...,a_k)$ and $(b_1,..., b_{k'})$, by definition of the $\mathfrak{S}$-cumulants, we get that $A_1$ and $^{t}A_2$ are $\mathfrak{S}$-free: since $A_1 \cup A_2$ is deterministic, we already saw that it is equivalent to say that $A_1$ and $A_2$ are free in the sense of Voiculescu.
\qed \end{proof}

In the end of this section, we prove the lemmas needed in order to complete the proof of Theorem \ref{th:transpofree}. Recall that $(1,...,l)$ is the $l$-cycle which sends $i$ on $i+1$ modulo $l$ and $(l,...,1)$ is the one which sends $i$ on $i-1$: $(l,...,1)$ is equal to $^{t}(1,...l)$. Let $l \leq k$ be two integers and $\sigma$ be in $\mathfrak{S}_{k}$.

\begin{lemma}
\label{lem:geo2}
Let $\sigma_1$ be in $\mathfrak{S}_{l}$ and $\sigma_2$ be in $\mathfrak{S}_{k-l}$, then $\sigma_1 \otimes \sigma_2 \leq \sigma$ if and only if $\sigma_1 \leq {\sf C}_k^{l}(\sigma)$ and $\sigma_2 \leq {\sf C}_k^{r}(\sigma).$
\end{lemma}

\begin{proof}
The lemma is a consequence of the special case when $\sigma = (1,...,k)$. Let us suppose that $\sigma= (1,...,k)$, then we have to prove that $\sigma_1 \otimes \sigma_2 \leq (1,...,k)$ if and only if $\sigma_1 \leq (1,...,l)$ and $\sigma_2 \leq (1,...,k-l)$. Recall the notion of defect defined in Equation~(1) of \cite{Gab1}: we need to prove that ${\sf df}(\sigma_1 \otimes \sigma_2, (1,...,k)) = {\sf df}(\sigma_1,(1,...,l)) + {\sf df}(\sigma_2, (1,...,k-l))$. After a simple calculation, this equality is equivalent to ${\sf nc}(\sigma_1 \vee (1,...,l) ) + {\sf nc}(\sigma_2 \vee (1,...,k-l)) = {\sf nc}(\sigma_1\otimes \sigma_2 \vee (1,...,k)) + 1.$ Recall Equation $(14)$ of \cite{Gab1} which asserts that for any $p, p' \in \mathcal{P}_k$, ${\sf Tr}(\rho_N(p \text{ }^{t}\!p') ) = N^{{\sf nc}(p \vee p')}$. This implies that for any permutations $u, v, w$, ${\sf nc}(u \vee vw) = {\sf nc}((u \text{ }^{t}w) \vee v)$. Thus: 
\begin{align*}
{\sf nc}(\sigma_1\otimes \sigma_2 \vee (1,...,k))& = {\sf nc}\left[\sigma_1\otimes \sigma_2 \vee \left((1,l+1)((1,...,l)\otimes (1,...,k-l)) \right)\right]\\
&=  {\sf nc}\left[(\sigma_1\otimes \sigma_2)((l,...,1) \otimes (k-l,...,1)) \vee (1,l+1)\right]
\end{align*}
but the transposition $(1,l+1)$ glue two cycles of $(\sigma_1\otimes \sigma_2)((l,...,1) \otimes (k-l,...,1))$:
\begin{align*}
{\sf nc}(\sigma_1\otimes \sigma_2 \vee (1,...,k))& = {\sf nc}((\sigma_1\otimes \sigma_2)((l,...,1) \otimes (k-l,...,1)) ) - 1 \\
&=  {\sf nc}\left[(\sigma_1\otimes \sigma_2) \vee ((1,...,l) \otimes (1,...,k-l))\right] -1\\
& = {\sf nc}(\sigma_1\vee (1,...,l)) + {\sf nc}(\sigma_2\vee (1,...,k-l))-1
\end{align*}
which was the equality to prove. 
\qed \end{proof}

\begin{lemma}
\label{lem:geodesic}
 The permutation $\sigma$ satisfies $\sigma \leq {\sf S}_l((1,...,k))$ if and only if there exists $\sigma_1 \leq (1,...,l)$ and $\sigma_2 \leq (k-l,...,1)$ such that $\sigma =\sigma_1\otimes \sigma_2$.  
\end{lemma}

\begin{proof}
For any integers $l\leq k$, any $\sigma$ and $\sigma'$ in $\mathfrak{S}_k$, ${\sf df}(\sigma', \sigma) = {\sf df}({\sf S}_l(\sigma'), {\sf S}_{l}(\sigma))$. Thus, $\sigma' \leq \sigma$ if and only if ${\sf S}_{l}(\sigma') \leq{\sf S}_{l}(\sigma)$. 

Let us suppose that $\sigma \leq {\sf S}_l((1,...,k))$, then ${\sf S}_{l}(\sigma) \leq (1,...,k)$. Lemma $3.1$ of \cite{Gab1} implies that ${\sf S}_{l}(\sigma) \in \mathfrak{S}_k$: there exists $\sigma_1 \in \mathfrak{S}_k$ and $\sigma_2 \in \mathfrak{S}_{k-l}$ such that $\sigma =\sigma_1\otimes \sigma_2$. Thus ${\sf S}_{l}(\sigma_1 \otimes \sigma_2) = \sigma_1 \otimes\text{ }^{t}\sigma_2 \leq (1,..., k)$. Using Lemma \ref{lem:geo2}, $\sigma_1 \leq (1,...,l)$ and $\sigma_2 \leq (k-l,...,1)$. The other implication is now straigthforward. 
\qed \end{proof}

\subsection{Classical and free cumulants}
\label{sec:Classicetfreecumu}
The notion of $\mathcal{P}$-cumulants generalizes the notions of classical and free cumulants. 
\subsubsection{Classical cumulants as $\mathcal{P}$-cumulants}
\label{sec:classical}
For any integer $k$, recall that ${\sf 0}_k$ is the partition $\{\{1,...,k,1',...,k'\}\}$. Let $(A,(m_p)_{p \in \mathcal{P}})$ be a $\mathcal{P}$-tracial algebra.

\begin{definition}
A family $A_1$ of elements of $A$ is classical if it is deterministic and for any integer $k$, any irreducible $p \in \mathcal{P}_k$ and any $(a_1,...,a_k) \in A^{k}$, $m_{p} (a_1,...,a_k) = m_{{\sf 0}_k}(a_1,...,a_k)$. 
\end{definition}

\begin{theorem}
\label{th:classicalcumulant}
Let $A_1$ be a classical family of elements of $A$. For any integer $k$, any $p \in \mathcal{P}_k$ and any $(a_1,...,a_k) \in A_1^{k}$, 
\begin{align*}
\kappa_{p}(a_1,...,a_k) =  \prod_{b \in p \vee {\mathrm{id}_k}}\delta_{p_{| b} = {\sf 0}_{\frac{\#b}{2}}} {\sf cum}_{\frac{\#b}{2}}\left((a_i)_{i \in b \cap \{1,...,k\}}\right), 
\end{align*}
where for any $n$, any $(a_1,...,a_n)$, ${\sf cum}_{\frac{\#b}{2}}\left((a_i)_{i \in b \cap \{1,...,k\}}\right)$ is the classical cumulant defined recursively by the fact that for any integer $n$, any $(a_1,...,a_n) \in A^{n}$: 
\begin{align*}
m_{{\sf 0}_{n}} (a_1,...,a_n) = \sum_{p \in {\sf P}_{n}} \prod_{b \in p} {\sf cum}_{\#b} ((a_i)_{i \in b}), 
\end{align*}
with ${\sf P}_n$ the set of partitions of $\{1,...,n\}$. 
\end{theorem}

\begin{proof}
Using the characterization of $\mathcal{P}$-cumulants, we need to prove that for any integer $k$, any $p \in \mathcal{P}_k$ and any $(a_1,...,a_k) \in A_1^{k}$: 
\begin{align*}
m_{p}\left( a_1,...,a_k \right) = \sum_{p' \leq p} \prod_{b \in p' \vee {\mathrm{id}_k}}\delta_{p'_{| b} = {\sf 0}_{\frac{\#b}{2}}} {\sf cum}_{\frac{\#b}{2}}\left((a_i)_{i \in b \cap \{1,...,k\}}\right). 
\end{align*}
Since $A_1$ is deterministic, it is enough to prove this equality when $p$ is irreducible. Recall Definition 3.4 of \cite{Gab1} where $\mathcal{D}_k$ was defined as the set of partitions in $\mathcal{P}_k$  which are coarser than ${\mathrm{id}_k}$. If $p \in \mathcal{P}_k$ is irreducible, $\{p'  \in \mathcal{D}_k | p'\leq p\} = \{ p' \in \mathcal{D}_k | p' \leq {\sf 0}_k \}$. Since $A_1$ is classical, this shows that it remains to prove: 
\begin{align}
\label{eq:amontrer}
m_{{\sf 0}_k}\left( a_1,...,a_k \right) = \sum_{p' \leq {\sf 0}_k} \prod_{b \in p' \vee {\mathrm{id}_k}}\delta_{p'_{| b} = {\sf 0}_{\frac{\#b}{2}}} {\sf cum}_{\frac{\#b}{2}}\left((a_i)_{i \in b \cap \{1,...,k\}}\right). 
\end{align}
Let $p' \in \mathcal{P}_k$ such that $p' \leq  {\sf 0}_k$. Recall the notions of admissible gluings and admissible splittings in Definition 2.5 and 2.6 in \cite{Gab1}. Since ${\sf 0}_k$ does not have any admissible gluing, using Theorem 2.3 of \cite{Gab1}, $p'$ must be a admissible splitting of ${\sf 0}_k$: the set $\{p' \in \mathcal{P}_k | p' \leq {\sf 0}_k\}$ is thus equal to $\mathcal{D}_k$. By Theorem $3.2$ of \cite{Gab1}, $(\mathcal{D}_k, \leq)$ is isomorphic to $({\sf P}_k, \trianglelefteq)$ where $\trianglelefteq$ is the finer-order. Using this isomorphism, the r.h.s. of Equation (\ref{eq:amontrer}) is equal to $\sum_{p \in {\sf P}_{k}} \prod_{b \in p} {\sf cum}_{\#b}\left((a_i)_{i \in b}\right)$ which, by definition, this is equal to $m_{{\sf 0}_k}(a_1,...,a_k)$: Equation (\ref{eq:amontrer}) is valid.
\qed \end{proof}

In particular, if we consider the $\mathcal{P}$-tracial algebra given by the example ``Random variables II" in Example \ref{ex:Ptracial}, for any integer $k$, any irreducible $p \in \mathcal{P}_k$, $m_{p} = m_{{\sf 0}_k}$. Thus, for any $p \in \mathcal{P}_{k}$ and any random variables $(X_1,...,X_k)$, 
\begin{align*}
\kappa_{p}(X_1,...,X_k) = \prod_{b \in p \vee {\mathrm{id}_k}}\delta_{p_{| b} = {\sf 0}_{\frac{\#b}{2}}} {\sf cum}_{\frac{\#b}{2}}\left((X_i)_{i \in b \cap \{1,...,k\}}\right). 
\end{align*} 
where ${\sf cum}$ are the classical probabilistic cumulants.

\subsubsection{Free cumulants as $\mathcal{P}$-cumulants}
Let $(A,(m_\sigma)_{\sigma \in \mathfrak{S}})$ be a {\em deterministic} $\mathfrak{S}$-tracial algebra, then using the bijection between $[\mathrm{id}_k, (1,...,k)] \cap \mathfrak{S}_k$ and the non-crossing partitions ${\sf NC}_{k}$ and the fact that the restriction of the order $\leq$ on  $[\mathrm{id}_k, (1,...,k)] \cap \mathfrak{S}_k$ is the geodesic order, one can see that $\kappa_{\sigma}^{\mathfrak{S}}$ are the usual free cumulants. A generalization of this fact is the following theorem whose proof is straightfoward. 

\begin{theorem}
Let $(A,(m_p)_{p \in \mathcal{A}})$ be a $\mathcal{A}$-tracial algebra, and let $(a_1,...,a_k)$ be a deterministic $U$-invariant $k$-tuple of elements of $A$. For any $p \in \mathcal{A}_k$, 
\begin{align*}
\kappa_{p}(a_1,...,a_k) =\delta_{p \in  \mathfrak{S}_k} \prod_{c \text{ cycle of } p} {\sf cum}_{\#c}((a_{i})_{i \in c}),
\end{align*}
where the $(a_i)_{i \in c}$ are ordered according to the cycle $c$ and where for any $n$, any $(a_1,...,a_n)$, ${\sf cum}_{\#c}((a_{i})_{i \in c})$ is the free cumulant defined recursively by the fact that for any integer $n$, any $(a_1,...,a_n) \in A^{n}$:
\begin{align*}
m_{(1,...,n)}(a_1,...,a_n) = \sum_{p \in {\sf NC}_n} \prod_{\{i_1<...<i_l\} \in p} {\sf cum}_{l} (a_{i_1},...,a_{i_l}).
\end{align*}
\end{theorem}

\section{Random matrices and $\mathcal{A}$-distribution}
\label{sec:basicmatrices}
\subsection{Basic definitions}
In the next subsections, we will study more the $\mathcal{P}$-tracial algebra of random matrices of size $N$ when $N$ goes to infinity. 

The mean $p$-moment of a $k$-tuple of matrices of size $N$, namely $(M_1,...,M_k)$, is defined in Equation (\ref{eq:normalized2}), and the $p$-moment of $(M_1,...,M_k)$, is defined in Equation (\ref{eq:normalized}). These are generalization of the usual moments described in Section \ref{sec:review}. Indeed, as explained in \cite{Levymaster}, for any $\sigma \in \mathfrak{S}_k,$
\begin{align}
\label{eq:lienmomentsigma}
 m_{\sigma}(M_1,...,M_k) = \prod_{c \in \sigma \vee {\mathrm{id}_k}} \frac{1}{N}{\sf Tr}\left(\prod_{i \in c \cap \{1,...,k\}} M_{i}\right), 
\end{align}
where the product is taken according to the order of the cycle of $\sigma$ associated with $c$. In general, if $b \in \mathcal{B}_k$, one can show that $m_b(M_1,...,M_k)$ is a product of normalized traces of products in the $M_i$ and $^{t}M_i$, $i \in \{1,...,k\}$. 

Recall that in Definition \ref{def:conv}, we defined the notion of convergence in $\mathcal{P}$-distribution whose definition could be extended for any set of partitions $\mathcal{A} \in \{ \mathcal{P}, \mathcal{B}, \mathfrak{S}, \mathcal{H}, \mathcal{B}s\}$. Let us consider such set $\mathcal{A}$ and for any integer $N$, let $(M^{N}_i)_{i \in \mathcal{I}}$ be a family of random matrices of size $N$. We will study the convergence of these family as $N$ goes to infinity, yet, since the matrices are random, we can either:
\begin{itemize}
\item consider the {\em convergence in $\mathcal{A}$-distribution} by applying the definition of convergence in $\mathcal{A}$-distribution to $(L^{\infty^{-}}\otimes \mathcal{M}_N(\mathbb{C}), (\mathbb{E}m_p)_{p \in \mathcal{A}})$,
\item or consider the {\em convergence in probability in $\mathcal{A}$-distribution} by applying the definition of convergence in $\mathcal{A}$-distribution to $((\mathcal{M}_N(\mathbb{C}), (m_p)_{p \in \mathcal{A}}))_{N \in \mathbb{N}}$.
\end{itemize} 

Using the link between the $\mathfrak{S}$-moments and the usual moments given by Equation (\ref{eq:lienmomentsigma}), if $(M^{N}_i)_{i \in \mathcal{I}}$ converges in $\mathfrak{S}$-distribution, it converges in non-commutative distribution. Yet, the convergence in $\mathfrak{S}$-distribution is slightly more demanding than the usual convergence of moments. More generally, one can see that $(M^{N}_i)_{i \in \mathcal{I}}$ converges in $\mathcal{B}$-distribution if and only if $(M^{N}_i)_{i \in \mathcal{I}} \cup (^{t}M^{N}_i)_{i \in \mathcal{I}}$ converges in $\mathfrak{S}$-distribution. Thus the convergence in  $\mathcal{B}$-distribution implies the convergence in non-commutative distribution of the family and its transpose. At last, the convergence in $\mathcal{P}$-distribution implies the convergence in distribution of traffics (see \cite{Camille} and the correspondence explained in \cite{GabCebronGuill}). 

Using this discussion, one can see that if for any integer $N$, $(M^{N}_i)_{i \in \mathcal{I}}$ is a family of symmetric or skew-symmetric matrices, it converges in $\mathfrak{S}$-moments if and only if it converges in $\mathcal{B}$-moments. The same holds if $\# \mathcal{I} = 1$ and the matrices are orthogonal. 

\begin{remark}
\label{rq:convsigma}
A less trivial result is that if for any integer $N$, $M_N$ is a permutation matrix, then $M_N$ converges in $\mathcal{P}$-distribution if and only if it converges in $\mathfrak{S}$-distribution. This is due to the fact that for any integers $k$ and $N$, any $p \in \mathcal{P}_k$ and any permutation $S \in \mathfrak{S}(N)$, there exists and integer $l\leq k$ and a permutation  $\sigma \in \mathfrak{S}_l$ such that $m_{p}(S,...,S) = m_{\sigma}(S,...,S)$. 
\end{remark}

From now on, if $\mathcal{O}$ is an observable, for which $\mathcal{O}\left(\left(M^{N}_{i_1}, ..., M^{N}_{i_k}\right)\right)$ converges, we set: 
\begin{align*}
\mathcal{O}(M_{i_1}, ..., M_{i_k}) = \lim_{N \to \infty}\mathcal{O}(M^{N}_{i_1}, ..., M^{N}_{i_k}). 
\end{align*}
Besides, if $\mathcal{O}$ needs $k$ arguments and if $M$ is a matrix we denote by $\mathcal{O}(M)$ the value of $\mathcal{O}(M,...,M)$ where we wrote $k$ times $M$.

\subsection{Asymptotic $\mathcal{A}$-factorization and convergence in probability}
Let us suppose that $(M^{N}_i)_{i \in \mathcal{I}}$ converges in $\mathcal{A}$-distribution. By definition, it defines a limiting $\mathcal{A}$-tracial distribution algebra, denoted by $(\mathbb{C}\{X_i, i \in \mathcal{I}\}, (\mathbb{E}m_{p}^{(M_i)_{i \in \mathcal{I}}})_{p \in \mathcal{A}})$. 

\begin{definition}
The family $(M^{N}_i)_{i \in \mathcal{I}}$ satisfies the asymptotic $\mathcal{A}$-factorization property if $(\mathbb{C}\{X_i, i \in \mathcal{I}\}, (\mathbb{E}m_{p}^{(M_i)_{i \in \mathcal{I}}})_{p \in \mathcal{A}})$ is a deterministic $\mathcal{A}$-tracial algebra. 
\end{definition}

This is equivalent to say that for any $k$, $l$, $p_1\in (\mathcal{A}_k$, $p_2\in \mathcal{A}_l$, $(i_1,...,i_{k+l}) \in \mathcal{I}^{k+l}$, $\mathbb{E}m_{p_1\otimes p_2}(M_{i_1}, ..., M_{i_k}, ..., M_{i_{k+l}}) = \mathbb{E}m_{p_1}(M_{i_1}, ..., M_{i_k})\mathbb{E}m_{p_2}(M_{i_{k+1}},...,  M_{i_{k+l}}).$

\begin{theorem}
\label{theoremconvprob}
For any integer $N$, let $(M^{N}_i)_{i \in \mathcal{I}}$ be a family of {\em real} random matrices which converges in $\mathcal{A}$-distribution. If $(M^{N}_i)_{i \in \mathcal{I}}$ satisfies the asymptotic $\mathcal{A}$-factoriza\-tion property then it converges in probability in $\mathcal{A}$-distribution and for any integer $k$, any $p \in \mathcal{A}_k$, any $i_1,...,i_k \in \mathcal{I}$, $m_{p}(M_{i_1}, ..., M_{i_k}) = \mathbb{E}m_{p}(M_{i_1}, ..., M_{i_k})$.  

If for any integer $N$, $(M^{N}_i)_{i \in \mathcal{I}}$ is a family of complex random matrices, the same result holds if we suppose that $(M^{N}_i)_{i \in \mathcal{I}}$ is stable by the conjugate or adjoint operations.
\end{theorem}

\begin{proof}
We will only prove the first part of the theorem, the second can be proved using similar arguments. For any integer $N$, let $(M^{N}_i)_{i \in \mathcal{I}}$ be a family of {\em real} random matrices which converges in $\mathcal{A}$-distribution. Let us suppose that $(M^{N}_i)_{i \in \mathcal{I}}$ satisfies the asymptotic $\mathcal{A}$-factorization property. Let $k, N \in \mathbb{N}$, $p \in \mathcal{A}_k$ and  $(i_1,...,i_k) \in I^{k}$. The variance $\mathbb{V}{\sf ar}\left[ m_{p}\left(M_{i_1}^{N}\otimes ...\otimes M_{i_k}^{N}\right)\right]$ is equal to: 
\begin{align*}
\mathbb{E}\left[m_{p \otimes p}  \left(M_{i_1}^{N}, ..., M_{i_k}^{N} , M_{i_1}^{N}, ..., M_{i_k}^{N}\right)\right] - \mathbb{E}\left[m_{p}  \left(M_{i_1}^{N}, ..., M_{i_k}^{N}\right)\right]^{2}, 
\end{align*}
which, by definition, is equal to: 
\begin{align*}
\mathbb{E}m_{p \otimes p} \left(M_{i_1}^{N}, ..., M_{i_k}^{N}, M_{i_1}^{N}, ..., M_{i_k}^{N}\right)  -  \left[\mathbb{E}m_{p}  \left(M_{i_1}^{N}, ..., M_{i_k}^{N}\right)\right]^{2}.
\end{align*}
Thus the variance has a limit which is given by: 
\begin{align*}
\mathbb{E}m_{p \otimes p}  \left(M_{i_1}, ..., M_{i_k}, M_{i_1}, ..., M_{i_k}\right)  -  \left[\mathbb{E}m_{p}  \left(M_{i_1}, ..., M_{i_k}\right)\right]^{2}, 
\end{align*}
which is equal to zero as $(M^{N}_i)_{i \in \mathcal{I}}$ satisfies the asymptotic $A$-factorization property.
\qed \end{proof}

When one considers symmetric, hermitian, orthogonal or unitary matrices, the asymptotic $\mathcal{A}$-factorization property of $(M^{N}_i)_{i \in \mathcal{I}}$ (under some stability by conjugate or adjoint operations) implies the convergence of probability of the empirical eigenvalues distributions. Let us remark that Theorem \ref{theoremconvprob2} shows that, in some cases, one can go further than the convergence in probability in order to get almost sure convergence.

\subsection{Consequences of the theory of $\mathcal{A}$-tracial algebras}
The definitions and results that were obtained in Section \ref{sec:Atracialalge} are applied in this sections in order to get results about asymptotics of random matrices. Recall that for any $N$, $(M_i^{N})_{i \in \mathcal{I}}$ is a family of random matrices of size $N$. Let us suppose that $(M_i^{N})_{i \in \mathcal{I}}$ converges in $\mathcal{A}$-distribution, recall that $(\mathbb{C}\{X_i, i \in \mathcal{I}\}, (\mathbb{E}m_{p}^{(M_i)_{i \in \mathcal{I}}})_{p \in \mathcal{A}})$ is the  limiting $\mathcal{A}$-tracial distribution algebra.

\begin{definition}
For any integer $k$, any $p$ in $\mathcal{P}_k$ and any $(i_1,...,i_k) \in \mathcal{I}^{k}$, the {\em asymptotic $p$-$\mathcal{A}$-exclusive moment} of $(M_{i_1}^{N}, ..., M_{i_k}^{N})$ is $\mathbb{E}m_{p^{c}}^{\mathcal{A}}(M_{i_1},...,M_{i_k}) = m_{p^{c}}^{\mathcal{A}}(X_{i_1}, ..., X_{i_k})$ and if $p \in \mathcal{A}_k$, the {\em asymptotic $p$-$\mathcal{A}$-cumulant} of $(M_{i_1}^{N}, ..., M_{i_k}^{N})$ is $\mathbb{E}\kappa_{p}^{\mathcal{A}}(M_{i_1}, ..., M_{i_k}) = \kappa^{\mathcal{A}}_p(X_{i_1}, ..., X_{i_k})$, where in each equalities the r.h.s. uses the definition of exclusive moment and cumulant forms in the limiting $\mathcal{A}$-tracial distribution algebra. 

Let $i$ be in $\mathcal{I}$. The $\mathcal{R}_{\mathcal{A}}$-transform of $(M_i^{N})_{N \in \mathbb{N}}$ is $\mathcal{R}_{\mathcal{A}}(M_i) = \mathcal{R}_{\mathcal{A}}(X_i)$, where again, the latter uses the definition of $\mathcal{R}_{\mathcal{A}}$-transform in the limiting $\mathcal{A}$-tracial distribution algebra. 
\end{definition}

\begin{notation}
If we do not specify $\mathcal{A}$, we suppose implicitely that $\mathcal{A}=\mathcal{P}$: for example, $m_{p^{c}}$ is the asymptotic $p$-$\mathcal{P}$-exclusive moment and we will call it simply the asymptotic $p$-exclusive moment. 
\end{notation}

By definition, the asymptotic $\mathcal{A}$-cumulants are uniquely characterized by the fact that for any $k \geq 0$, any $p$ in $\mathcal{A}_k$ and any $(i_1,...,i_k) \in \mathcal{I}^{k}$, $$\mathbb{E}m_{p}(M_{i_1}, ..., M_{i_k}) = \sum_{p' \in \mathcal{A}_k| p' \leq p} \mathbb{E}\kappa_{p'}^{\mathcal{A}}(M_{i_1}, ..., M_{i_k}).$$
Besides, for any $i \in \mathcal{I}$, any $p \in \mathcal{A}$, $\mathcal{R}_{\mathcal{A}}[M_i](p) = \kappa^{\mathcal{A}}_{p}(M_{i}, ..., M_{i})$. In the case where $\mathcal{A} = \mathfrak{S}$, we recover the usual notion of $\mathcal{R}$-transform (Theorem 4.4 in \cite{Gab1}).

Also by definition, the asymptotic $\mathcal{A}$-exclusive moment of $(M_{i_1}^{N}, ..., M_{i_k}^{N})$ are characterized by the fact that for any $k \geq 0$, any $p \in \mathcal{P}_k$ and any $(i_1,...,i_k) \in \mathcal{I}^{k}$,
\begin{align*}
\mathbb{E}m_{p^{c}}^{\mathcal{A}}(M_{i_1},...,M_{i_k}) = \sum_{p' \in \mathcal{A}_k | p' \sqsupset p} \mathbb{E}\kappa_{p'}^{\mathcal{A}}(M_{i_1},...,M_{i_k}).
\end{align*}
Besides, for any integer $k\geq 0$, any $p \in \mathcal{A}_k$, and any $(i_1,...,i_k) \in \mathcal{I}^{k}$,
\begin{align*}
\mathbb{E}m_{p}(M_{i_1},...,M_{i_k}) = \sum_{p' \in \mathcal{P}_k | p' \dashv p} \mathbb{E}m_{p'^{c}}^{\mathcal{A}}(M_{i_1},...,M_{i_k}).
\end{align*}
Let us remark that the last equality only holds for $p \in \mathcal{A}_k$, even if the family $(M^{N}_{i})_{i \in \mathcal{I}}$ happens to converge in $\mathcal{P}$-distribution. This is due to the fact that asymptotic $\mathcal{A}$-exclusive moments are defined using the notion of natural extension of $\mathcal{A}$-tracial algebras as $\mathcal{P}$-tracial algebras. 

Using Lemma 3.2 of \cite{Gab1}, the exclusive moments have a simpler expression in some special cases. Recall the definitions of ${\sf Mb}(p)$, $\overline{\mathfrak{S}_k}$ and $\overline{\mathcal{B}_k}$ in Definition 3.6 of~\cite{Gab1}. 
\begin{theorem}
\label{th:simplification}
Let $k$ be a positive integer. For any $p \in \mathcal{B}_k$, any $(i_1,...,i_k) \in \mathcal{I}^{k}$: 
\begin{align*}
\mathbb{E}m_{p^{c}}^{\mathcal{A}}(M_{i_1},...,M_{i_k}) = \delta_{p \in \mathcal{A}_k} \mathbb{E}\kappa_{p}^{\mathcal{A}}(M_{i_1},...,M_{i_k}). 
\end{align*} 
Besides, if $\mathcal{A} \in \{ \mathfrak{S}, \mathcal{B}\}$, for any $p \in \mathcal{P}_k$, any $(i_1,...,i_k) \in \mathcal{I}^{k}$: 
\begin{align*}
\mathbb{E}m_{p^{c}}^{\mathcal{A}}(M_{i_1},...,M_{i_k}) = \delta_{p \in \overline{\mathcal{A}_k}} \mathbb{E}\kappa_{{\sf Mb}(p)}^{\mathcal{A}}(M_{i_1},...,M_{i_k}). 
\end{align*} 
\end{theorem}

For any integer $N$, let $(M_i^{N})_{i \in \mathcal{I}}$ and $(L_j^{N})_{j \in \mathcal{J}}$ be two families of random matrices. We will always suppose in this article  that $\mathcal{I} \cap \mathcal{J} = \emptyset$ for the sake of simplicity and clarity. Let us suppose that the family $(M_i^{N})_{i \in \mathcal{I}} \cup(L_j^{N})_{j \in \mathcal{J}}$ converges in $\mathcal{A}$-distribution: it defines a limiting $\mathcal{A}$-tracial distribution algebra, denoted by $(\mathbb{C}\{X_i, X_j, i \in \mathcal{I}, j \in \mathcal{J}\}, (\mathbb{E}m_{p}^{(M_i)_{i \in \mathcal{I}} \cup (L_j)_{j \in \mathcal{J}}})_{p \in \mathcal{A}})$.

\begin{definition}
The families $(M_i^{N})_{i \in \mathcal{I}}$ and $(L_j^{N})_{j \in \mathcal{J}}$  are asymptotically $\mathcal{A}$-free if and only if the families $(X_i)_{i \in \mathcal{I}}$ and $(X_j)_{j \in \mathcal{J}}$ are $\mathcal{A}$-free in the limiting $\mathcal{A}$-tracial distribution algebra.
\end{definition}

An other formulation is to say that  $(M_i^{N})_{i \in \mathcal{I}}$ and $(L_j^{N})_{j \in \mathcal{J}}$  are asymptotically $\mathcal{A}$-free if and only if the mixed asymptotic $\mathcal{A}$-cumulants vanish and the compatible asymptotic $\mathcal{A}$-cumulants factorize. The following result is a consequence of Lemma~\ref{lem:deterministic}.  

\begin{proposition}
Let us suppose that $(M_i^{N})_{i \in \mathcal{I}}$ and $(L_j^{N})_{j \in \mathcal{J}}$ are asymptotically $\mathcal{A}$-free and satisfy the asymptotic $\mathcal{A}$-factorization property then the asymptotic $\mathcal{A}$-factorization property holds for $(M_i^{N})_{i \in \mathcal{I}} \cup (L_j^{N})_{j \in \mathcal{J}}$. 
\end{proposition}

Using the version of Theorem \ref{theorem:autreformulation} for $\mathcal{A}$-tracial algebras, we have another formulation of asymptotic $\mathcal{A}$-freeness. 

\begin{theorem}
\label{th:freeasymptoautre}
The families $(M_i^{N})_{i \in \mathcal{I}}$ and $(L_j^{N})_{j \in \mathcal{J}}$  are asymptotically $\mathcal{A}$-free if and only for any $k_1, k_2 >0$,  any $p \in \mathcal{P}_{k_1+k_2}$, any $(i_1,...,i_{k_1}) \in \mathcal{I}^{k_1}$, any $(j_1,...,j_{k_2}) \in \mathcal{I}^{k_2}$, 
\begin{align*}
\mathbb{E}m^{\mathcal{A}}_{p^{c}}\!\left[M_{i_1},...,M_{i_{k_1}}, L_{i_1}, ..., L_{i_{k_2}}\right] \!\!=\! \delta_{p_{k_1}^{l} \otimes p_{k_1}^{r} \sqsupset p} \mathbb{E}m^{\mathcal{A}}_{(p_{k_1}^{l})^{c}} \!\left[(M_i)_{i=1}^{k_1}\right] \!\mathbb{E}m^{\mathcal{A}}_{(p_{k_1}^{r})^{c}} \!\left[(L_{j})_{j=1}^{k_2}\right]\!.
\end{align*}
\end{theorem}

From now on, let us suppose that  $(M_i^{N})_{i \in \mathcal{I}}$ and $(L_j^{N})_{j \in \mathcal{J}}$ are asymptotically $\mathcal{A}$-free. Then the $\mathcal{A}$-distribution of  $(M_i^{N})_{i \in \mathcal{I}} \cup(L_j^{N})_{j \in \mathcal{J}}$ only depends on the $\mathcal{A}$-distribution of  $(M_i^{N})_{i \in \mathcal{I}}$ and the $\mathcal{A}$-distribution of $(L_j^{N})_{j \in \mathcal{J}}$. As a consequence of Theorems \ref{th:calculloi} and \ref{th:Rtransfoprod}, we get the following theorem. Recall the notion of factorizations ((12) in Section \ref{sec:basic}).

\begin{theorem}
\label{sommeetproduit}
Let $k$ be an integer, $p$ be in $\mathcal{A}_k$, $(i_1,...,i_k)$ be in $\mathcal{I}^{k}$ and $(j_1,...,j_k)$ be in $\mathcal{J}^{k}$: 
\begin{align*}
\mathbb{E}\kappa_{p}^{\mathcal{A}}[M_{i_1}+L_{j_1}, ...,M_{i_k}+L_{j_k}] &= \sum_{(p_1,p_2, I) \in \mathfrak{F}_2(p)}\mathbb{E}\kappa_{p_1}^{\mathcal{A}}[(M_{i_1},...,M_{i_k})] \mathbb{E}\kappa_{p_2}^{\mathcal{A}}[(L_{j_1},...,L_{j_k})],\\
\mathbb{E}\kappa_{p}^{\mathcal{A}}[M_{i_1}L_{j_1}, ...,M_{i_k}L_{j_k}] &= \!\!\!\!\!\!\!\!\!\!\!\!\!\!\!\!\!\!\!\!\!\!\!\!\!\!\sum_{p_1, p_2 \in \mathcal{A}_k | p_1 \prec p, p_2 \in {\sf K}_{p}(p_1)} \!\!\!\!\!\!\!\mathbb{E}\kappa_{p_1}^{\mathcal{A}}[(M_{i_1},...,M_{i_k})]\mathbb{E} \kappa_{p_2}^{\mathcal{A}}[(L_{j_1},...,L_{j_k})], \\
\mathbb{E}m_{p}[M_{i_1}L_{j_1}, ...,M_{i_k}L_{j_k}] &= \sum_{p_1 \in \mathcal{A}_k | p_1 \leq p} \mathbb{E}\kappa_{p_1}^{\mathcal{A}}[(M_{i_1},...,M_{i_k})] \mathbb{E}m_{^{t}p_1\circ p}[(L_{j_1},...,L_{j_k})]. 
\end{align*}
Thus, for any $i$ in $\mathcal{I}$ and $j$ in $\mathcal{J}$, 
\begin{align*}
\mathcal{R}_{\mathcal{A}}[M_i+L_j] =\mathcal{R}_{\mathcal{A}}[M_i] \boxplus \mathcal{R}_{\mathcal{A}}[L_j] \ \ \ \text{ and }\ \ \
\mathcal{R}_{\mathcal{A}}[M_iL_j] =\mathcal{R}_{\mathcal{A}}[M_i] \boxtimes \mathcal{R}_{\mathcal{A}}[L_j]. 
\end{align*}
\end{theorem}

An asymptotic $\mathcal{P}$-central limit theorem for random matrices can be also deduced from Theorem \ref{th:tcl}. Let $\mathcal{A}_1 \subset \mathcal{A}$ be a subset of partitions in $\{\mathcal{P}, \mathcal{H}, \mathcal{B}s, \mathcal{B}, \mathfrak{S}\}$. The notion of $\mathcal{G}(\mathcal{A}_1)$-$[\mathcal{A}]$-asymptotic invariance is deduced from Definition~\ref{def:GAinva}. 

\begin{definition}
The family $(M^{N}_i)_{i \in \mathcal{I}}$ is {\em asymptotically $\mathcal{G}(\mathcal{A}_1)$-$[\mathcal{A}]$-invariant} if $(X_i)_{i \in \mathcal{I}}$ is $\mathcal{G}(\mathcal{A}_1)$-invariant in the limiting $\mathcal{A}$-tracial distribution algebra of  $(M^{N}_i)_{i \in \mathcal{I}}$. 
If $\mathcal{A} = \mathcal{P}$ we simply will talk about {\em asymptotic $\mathcal{G}(\mathcal{A}_1)$-invariance. }
\end{definition}
This means that for any integer $k$, any $p \in \mathcal{A}$ of length $k$ and any $(i_1,...,i_k) \in \mathcal{I}^{k}$, $\mathbb{E}\kappa_{p}^{\mathcal{A}}(M_{i_1}, ..., M_{i_k}) = \delta_{p \in \mathcal{A}_1} \mathbb{E}\kappa_{p}^{\mathcal{A}_1}(M_{i_1}, ..., M_{i_k})$. Let us remark that this notion depends also on the chosen $\mathcal{A}$. In particular any family $(M^{N}_i)_{i \in \mathcal{I}}$ which converges in $\mathcal{A}$-distribution is asymptotically $\mathcal{G}(\mathcal{A})$-$[\mathcal{A}]$-invariant. The following theorem is a consequence of Theorem \ref{th:invariancecaract}.

\begin{theorem} 
Let us suppose that $(M^{N}_i)_{i \in \mathcal{I}}$ converges in $\mathcal{P}$-distribution. It is asymptotically $U$-invariant, resp. $O$-invariant, if and any if  for any $k \geq 0$, any $p \in \mathcal{P}_k$ and any $(i_1,...,i_k) \in \mathcal{I}^{k}$, $\mathbb{E}m_{p^{c}}(M_{i_1},...,M_{i_k})$ is equal to $\delta_{p \in \overline{\mathfrak{S}_k}} \mathbb{E}m_{({\sf Mb}(p))^{c}}(M_{i_1},...,M_{i_k})$, respectively $\delta_{p \in \overline{\mathcal{B}_k}} \mathbb{E}m_{({\sf Mb}(p))^{c}}(M_{i_1},...,M_{i_k})$. 
\end{theorem}

\subsection{The different notions of asymptotic $\mathcal{A}$-freeness}
This section is a direct consequence of Section \ref{sec:Atracialalge}.  The following results about asymptotic $\mathcal{A}$-freeness for random matrices are consequences of Theorems \ref{th:invarianceetliberte}, \ref{th:critere} and \ref{th:transpofree}. 
Let $\mathcal{A}_2 \subset \mathcal{A}_1$ be two sets of partitions in $\{\mathcal{P}, \mathcal{H}, \mathcal{B}s, \mathcal{B}, \mathfrak{S}\}$. For any integer $N$, let $(M_i^{N})_{i \in \mathcal{I}}$ and $(L_j^{N})_{j \in \mathcal{J}}$ be two families of random matrices which converge in $\mathcal{A}_1$-distribution. 

\begin{theorem}
\label{th:asympinvetliberte}
If $(M_i^{N})_{i \in \mathcal{I}} \cup (L_j^{N})_{j \in \mathcal{J}}$ is asymptotically $\mathcal{G}(\mathcal{A}_2)$-[$\mathcal{A}_1$]-invariant, the two families  $(M_i^{N})_{i \in \mathcal{I}}$ and $(L_j^{N})_{j \in \mathcal{J}}$ are asymptotically $\mathcal{A}_2$-free if and only if they are asymptotically $\mathcal{A}_1$-free. 

Besides, if only $(L_j^{N})_{j \in \mathcal{J}}$ is asymptocally $\mathcal{G}(\mathcal{A}_2)$-[$\mathcal{A}_1$]-invariant and if the families  $(M_i^{N})_{i \in \mathcal{I}}$ and $(L_j^{N})_{j \in \mathcal{J}}$ are asymptotically $\mathcal{A}_1$-free then they are $\mathcal{A}_2$-free. 
\end{theorem}

The notions of asymptotic $\mathcal{A}$-freeness are linked even when we do not suppose any asymptotic invariance. 

\begin{theorem}
\label{th:lesdifferentesliberte}
For this theorem, let us suppose that $\mathcal{A}_1= \mathcal{P}$. Then we have the following equivalence or implications: 
\begin{enumerate}
\item if  $(M_i^{N})_{i \in \mathcal{I}}$ and $(L_j^{N})_{j \in \mathcal{J}}$ satisfy the asymptotic $\mathcal{P}$-factorization, the two families are asymptotically $\mathfrak{S}$-free if and only if they are asymptotically free in the sense of Voiculescu, 
\item if $(M_i^{N})_{i \in \mathcal{I}}$ and $(L_j^{N})_{j \in \mathcal{J}}$  are asymptotically $\mathcal{B}$- or $\mathcal{B}s$-free then they are asymptotically $\mathfrak{S}$-free. 
\end{enumerate}
and the following negative assertions hold: 
\begin{enumerate}
\item if $(M_i^{N})_{i \in \mathcal{I}}$ and $(L_j^{N})_{j \in \mathcal{J}}$ are asymptotically $\mathcal{P}$-free and if there exists $(i_1, i_2) \in \mathcal{I}$ and $(j_1, j_2) \in \mathcal{J}$ such that:
\begin{align*}
\mathbb{E}\kappa_{{\sf 0}_{2}}^{\mathcal{P}}(M_{i_1}, M_{i_2}) \mathbb{E}\kappa_{{\sf 0}_{2}}^{\mathcal{P}}(L_{j_1}, L_{j_2}) \neq 0, 
\end{align*}
then $(M_i^{N})_{i \in \mathcal{I}}$ and $(L_j^{N})_{j \in \mathcal{J}}$ are not asymptotically $\mathfrak{S}$-free. In particular, asymptotic $\mathcal{P}$-freeness does not imply the asymptotic $\mathfrak{S}$-freeness. 
\item the asymptotic $\mathcal{P}$-freeness of  $(M_i^{N})_{i \in \mathcal{I}}$ and $(L_j^{N})_{j \in \mathcal{J}}$ does not generally imply the asymptotic $\mathcal{B}$-freeness of the two families, 
\item the asymptotic $\mathfrak{S}$-freeness of  $(M_i^{N})_{i \in \mathcal{I}}$ and $(L_j^{N})_{j \in \mathcal{J}}$ does not generally imply the asymptotic $\mathcal{B}$-freeness of the two families. 
\end{enumerate}
\end{theorem}

At last, we can prove the asymptotic freeness of a unitary invariant matrix and its transpose. 

\begin{theorem}
For this theorem, let us suppose that $\mathcal{B} \subset \mathcal{A}_1$. If $(M_i^{N})_{i \in \mathcal{I}}$ and $(L_j^{N})_{j \in \mathcal{J}}$ satisfy the asymptotic $\mathfrak{S}$-factorization property and $(M_i^{N})_{i \in \mathcal{I}} \cup (L_j^{N})_{j \in \mathcal{J}}$ is asymptotically $U$-[$\mathcal{A}_1$]-invariant, then $(M_i^{N})_{i \in \mathcal{I}}$ and $(^{t}L_j^{N})_{j \in \mathcal{J}}$ are asymptotically free in the sense of Voiculescu. 
\end{theorem}

Let us remark that, looking at the proof of Theorem \ref{th:transpofree}, we could change the condition of asymptotic $\mathcal{A}_1$-factorization property by a condition of asymptotic $\mathfrak{S}$-factorization property.

\section{Dualities and the finite-dimensional world}
\label{duali}

\subsection{Dualities}
\label{sectionschur}

Let $k$ be an integer, $p$ be in $\mathcal{A}_k$ and $L$ be a matrix  of size $N$ in $\mathcal{G}(\mathcal{A})(N)$, it is not difficult to see that $M^{\otimes k} \rho_N(p) M^{\otimes k} = \rho_N(p)$. Thus, using the tracial property, for any $k$-tuple of matrices of size $N$, $(M_1, ..., M_k)$, $\mathbb{E}m_{p}\left(M_1,...,M_k\right)$ is equal to:
\begin{align*}
\frac{1}{N^{{\sf nc}(p\vee \mathrm{id}_k)}} {\sf Tr}\!\!\left(\left[ \int_{\mathcal{G}(\mathcal{A})(N)} M^{\otimes k} \mathbb{E}\left[M_1\otimes ... \otimes M_k\right] (M^{-1})^{\otimes k}dM\right] \rho_{N}(^{t}p)\right). 
\end{align*}

The endomorphism $\int_{\mathcal{G}(\mathcal{A})(N)} M^{\otimes k} \mathbb{E}\left[M_1\otimes ... \otimes M_k\right] (M^{-1})^{\otimes k}dM$ commutes with $M^{\otimes k }$ for any $M\in \mathcal{G}(\mathcal{A})(N)$. This motivates the study of endomorphisms of $\left(\mathbb{C}^{N}\right)^{\otimes k}$ which commute with $M^{\otimes k}$, with $M\in \mathcal{G}(\mathcal{A})(N)$. We introduce for this the Schur-Weyl-Jones's duality and some similar statements. The dualities are summarized in Table $1$. The Schur-Weyl duality was previously used in the setting of random matrices, for example in \cite{Collins2003}, \cite{ColSni}, \cite{Levy2} and \cite{Levymaster}. Besides, this section and next section can be partly seen as a generalization of \cite{Casalis}. One can find more informations about these dualities in \cite{Good}, and for the last two dualities one can look at \cite{Jo}, \cite{MA1} and \cite{Halv}. The groups $\mathcal{G}(\mathcal{A})(N)$ are called, if we do not take into account the unitary groups, easy orthogonal groups. For a definition of the easy orthogonal groups and the proofs of the duality theorems one can have a look at \cite{Banica}. 

\begin{center}
Table $1$. Dualities between groups and partitions. 
\end{center}
\begin{center}

\begin{tabular}{|c|c||c|c|}
\hline 
Family  $\mathcal{G}(\mathcal{A})$& Description of $\mathcal{G}(\mathcal{A})(N)$ & Partitions $\mathcal{A}$ & Descriptions of $\mathcal{A}$ \\ 
\hline 
& & & \\
$U$ & $MM^{*} = {\mathrm{Id}}_N$ & $\mathfrak{S}$ & permutations \\ 
& & & \\
\hline 
& & & \\
$O$ & $M\ \! ^{t}\!M = {\mathrm{Id}}_N$, $M \in \mathcal{M}_N(\mathbb{R})$ & $\mathcal{B}$ & $\forall b \in p, \# b = 2$  \\ 
& & & \\
\hline 
& & & \\
$H$ & $M_{i,j} \in \{0,-1,+1\}$,& $\mathcal{H}$& $\forall b \in p, \# b \in  2 \mathbb{N}$\\
& $\sum_{j} \delta_{M_{i,j} \neq 0} =\sum_{i} \delta_{M_{i,j} \neq 0} = 1$ &  & \\ 
& & & \\
\hline 
& & & \\
$B$ & $M \in O(N)$ & $\mathcal{B}s$  & $\forall b \in p, \# b \leq 2$ \\
& $\sum_i M_{i,j} = \sum_{j} M_{i,j} = 1$ & &  \\ 
& & & \\
\hline 
& & & \\
$B \times \mathbb{Z}/2\mathbb{Z} $ &  & $\mathcal{B}s$ &  \\ 
& & & \\
\hline 
& & & \\
$\mathfrak{S} $&$M_{i,j} \in \{0,1\}$,& & \\
& $\sum_{j} \delta_{M_{i,j} \neq 0} =\sum_{i} \delta_{M_{i,j} \neq 0} = 1$ &$\mathcal{P}$  &  \\ 
& & & \\
\hline
& & & \\
$\mathfrak{S} \times  \mathbb{Z}/2\mathbb{Z}  $&  & $\mathcal{P}$&  \\ 
& & & \\
\hline 
\end{tabular} 
\text{ }\\
\end{center}

We will make a distinction between $\mathfrak{S}_k$ and $\mathfrak{S}(N)$. The notation $\mathfrak{S}_k$ will stand for the symmetric group, seen as a group of permutations, and $\mathfrak{S}(N)$ will be seen as the subgroup of permutation matrices in $\mathcal{G}l(N)$. Let $k$ and $N$ be two positive integers. 

\begin{definition}
For any $M \in {\mathcal{M}}_{N}(\mathbb{C})$, any $x_1,..., x_k \in \mathbb{C}^{N}$, let: 
$$ \rho^{k}(M)(x_1 \otimes...\otimes x_k) = Mx_1 \otimes ... \otimes Mx_k.$$
The application $\rho^{k}$ is the tensor action of $ {\mathcal{M}}(N)(\mathbb{C})$ on $\left(\mathbb{C}^{N}\right)^{\otimes k}$. 
\end{definition}

Let $C$ be a subalgebra of ${\sf End}\left((\mathbb{C}^{N})^{\otimes k}\right)$, let us define the commutant of $C$. 

\begin{definition}
The commutant of $C$, denoted by $C'$, is: 
\begin{align*}
C' = \left\{E \in {\sf End}\left((\mathbb{C}^{N})^{\otimes k}\right), \forall F \in C, EF=FE\right\}.
\end{align*}
\end{definition}
For any subgroup $G$ of $\mathcal{M}(N)$, we define:
$$\mathbb{C}\left[\rho^k_{G}\right]  = \mathbb{C}\left[\left\{\rho^{k}(g), g \in G\right\}\right] \subset {\sf End}\left((\mathbb{C}^{N})^{\otimes k}\right).$$
Beside, we define: 
$$\mathbb{C}\left[\rho_N^{\mathcal{A}_k}\right] = \mathbb{C}\left[\left\{\rho_{N}(p ), p \in \mathcal{A}_k\right\}\right]\subset {\sf End}\left((\mathbb{C}^{N})^{\otimes k}\right).$$

The kind of duality between the groups and the set of partitions described in Table~$1$ is explained in the following theorem.

\begin{theorem}
For any set of partitions $\mathcal{A} \in \{\mathfrak{S}, \mathcal{B}, \mathcal{H}, \mathcal{B}s, \mathcal{P}\}$ and any integers $k$ and $N$, 
\label{duality}
\begin{align*}
\left(\mathbb{C}\left[\rho^k_{\mathcal{G}(\mathcal{A})(N)}\right]\right)' = \mathbb{C}\left[\rho_N^{\mathcal{A}_k}\right] \text{ and } \left(\mathbb{C}\left[\rho_N^{\mathcal{A}_k}\right]\right)' = \mathbb{C}\left[\rho^k_{\mathcal{G}(\mathcal{A})(N)}\right]. 
\end{align*}
\end{theorem}

The duality $(\mathfrak{S}, \mathcal{P})$ is the simplest to show: it is the Jones' duality. One can look at the proof in \cite{Halv} which uses the exclusive basis defined in Definition $3.16$ of \cite{Gab1}. The dualities $(U, \mathfrak{S})$ and $(O,\mathcal{B})$ are the Schur-Weyl dualities.

\subsection{Finite-dimensional cumulants and exclusive moments}
The notion of asymptotic $\mathcal{A}$-cumulants and $\mathcal{A}$-exclusive moments are well defined for sequences $((M_i^{N})_{i \in \mathcal{I}})_{N \in \mathbb{N}}$ of families of random matrices which converge in $\mathcal{A}$-distribution. In this section, we introduce the notion of finite-dimensional cumulants and exclusive moments for families $(M_i)_{i \in \mathcal{I}}$ of random matrices of fixed size. 

Let $(M_i)_{i \in \mathcal{I}}$ be a family of random matrices of size $N$ and $(i_1,...,i_k)$ be a $k$-tuple of elements of $\mathcal{I}$. As noticed in Section \ref{sectionschur}, the endomorphism $$ E_N^{\mathcal{G}(\mathcal{A})} = \int_{\mathcal{G}(\mathcal{A})(N)} M^{\otimes k} \mathbb{E}\left[M_{i_1}\otimes ... \otimes M_{i_k}\right] (M^{-1})^{\otimes k}dM$$ commutes with $M^{\otimes k }$ for any $M\in \mathcal{G}(\mathcal{A})(N)$. Using the duality stated in Theorem~\ref{duality}, it belongs to $\mathbb{C}\left[\rho_{N}^{\mathcal{A}_k}\right]$. Let us suppose that $N\geq 2k$ so that $\rho_{N}$ is injective: the endomorphism $E_N^{\mathcal{G}(\mathcal{A})}$ can be seen as an element of $\mathbb{C}[\mathcal{A}_k]$.   Recall Definition 5.3 of \cite{Gab1} where the notion of cumulants for elements in $\mathbb{C}[\mathcal{P}_k]$ is defined. 

\begin{definition}
For any partition $p$ in $\mathcal{A}_k$, the $p$-$\mathcal{A}$-cumulant of $(M_{i_1},...,M_{i_k})$ is the cumulant of the endomorphism $E_N^{\mathcal{G}(\mathcal{A})}$: 
\begin{align*}
\mathbb{E}\kappa_{p}^{\mathcal{A}}\left( M_{i_1}, ..., M_{i_k}\right) = \kappa_{p}\left[ \int_{\mathcal{G}(\mathcal{A})(N)} M^{\otimes k} \mathbb{E}\left[M_{i_1}\otimes ... \otimes M_{i_k}\right] (M^{-1})^{\otimes k}dM\right]. 
\end{align*}
For any partition $p$ in $\mathcal{P}_k$, the mean $p$-$\mathcal{A}$-exclusive moment of $(M_{i_1},...,M_{i_k})$, denoted by $\mathbb{E}m_{p^{c}}^{\mathcal{A}} \left( M_{i_1}, ..., M_{i_k}\right)$, is the $p$-exclusive moment of the endomorphism $E_N^{\mathcal{G}(\mathcal{A})}$:
\begin{align*}
\frac{1}{N^{{\sf nc}(p\vee \mathrm{id}_k)}} {\sf Tr} \left[\left(\int_{\mathcal{G}(\mathcal{A})(N)} M^{\otimes k} \mathbb{E}\left[M_{i_1}\otimes ... \otimes M_{i_k}\right] (M^{-1})^{\otimes k}dM\right) \rho_{N}(^{t}p^{c})\right].
\end{align*}
\end{definition}

From now on, when we consider finite dimensional cumulants, we will always suppose that $N\geq 2k$ even if omit to write it. Besides, when $\mathcal{A} = \mathcal{P}$, we will omit to specify it: $\mathbb{E}\kappa_{p}$ and $\mathbb{E}m_{p^{c}}$ stand for the $p$-$\mathcal{P}$-cumulant and exclusive moment. 

Let us remark that for any $p \in \mathcal{A}_k$, $\mathbb{E}m_{p}(M_{i_1}, ..., M_{i_k})$ is the $p$-moment of $E_N^{\mathcal{G}(\mathcal{A})}$ since, by Theorem \ref{duality}, $\rho_N(p)$ commutes with the tensor action of $\mathcal{G}(\mathcal{A})(N)$. We can apply the results of \cite{Gab1}, namely  Theorems 5.1 (or 5.5) and 5.2. 

\begin{theorem}
\label{th:lienlimite}
For any integer $N$, let $(M_i^{N})_{i \in \mathcal{I}}$ be a family of random matrices of size $N$. The family $(M_i^{N})_{i \in \mathcal{I}}$  converges in $\mathcal{A}$-distribution if and only if one of the following assertions is valid: 
\begin{enumerate}
\item for any $k$, any $(i_1,...,i_k) \in \mathcal{I}^{k}$ and any $p \in \mathcal{A}_k$, $\mathbb{E}\kappa_{p}^{\mathcal{A}}(M_{i_1}^{N}, ..., M_{i_1}^{N})$ converges as $N$ goes to infinity, 
\item for any $k$, any $(i_1,...,i_k) \in \mathcal{I}^{k}$ and any $p \in \mathcal{P}_k$, $\mathbb{E}m_{p^{c}}^{\mathcal{A}}(M_{i_1}^{N}, ..., M_{i_1}^{N})$ converges as $N$ goes to infinity. 
\end{enumerate}
If it does, then for any $k$, any $(i_1,...,i_k) \in \mathcal{I}^{k}$ and any $p \in \mathcal{A}_k$,
\begin{align*}
\mathbb{E}\kappa^{\mathcal{A}}_{p}(M_{i_1}, ..., M_{i_k}) &= \lim_{N \to \infty} \mathbb{E}\kappa^{\mathcal{A}}_{p}(M^{N}_{i_1}, ..., M^{N}_{i_k}), 
\end{align*}
and for any $p \in \mathcal{P}_k$, 
\begin{align*}
\mathbb{E}m^{\mathcal{A}}_{p^{c}}(M_{i_1}, ..., M_{i_k}) &= \lim_{N \to \infty} \mathbb{E}m^{\mathcal{A}}_{p^{c}}(M^{N}_{i_1}, ..., M^{N}_{i_k}).
\end{align*} 
where in the two last equations, the l.h.s. is respectively the asymptotic $p$-$\mathcal{A}$-cumulant and the asymptotic $p$-$\mathcal{A}$-exclusive moment. 
\end{theorem}

\subsection{$\mathcal{G}(\mathcal{A})$-invariant families}
\subsubsection{Definitions and convergence}
Let $(M_i^{N})_{i \in \mathcal{I}}$ be a family of random matrices of size $N$. When for any integer $N$, $(M_i^{N})_{i \in \mathcal{I}}$ is invariant in law by conjugation by $\mathcal{G}(\mathcal{A})(N)$, the notions of finite-dimensional exclusive moments and cumulants become easier. 

\begin{definition}
\label{invariancedef}
The family $(M_i^{N})_{i \in \mathcal{I}}$ is $\mathcal{G}(\mathcal{A})$-invariant if for any matrix $M$ in the group $\mathcal{G}(\mathcal{A})(N)$, any integer $k$, any $(i_1,...,i_{k})$ in $\mathcal{I}^{k}$, $(MM_{i_1}M^{-1}, ..., MM_{i_k}M^{-1})$ has the same law as $(M_{i_1},...,M_{i_k})$.
\end{definition}

Let us suppose for this section that $(M_i^{N})_{i \in \mathcal{I}}$ is $\mathcal{G}(\mathcal{A})$-invariant. Let $(i_1,...,i_k) \in \mathcal{I}^{k}$. By definition, for any integer $N$: 
\begin{align}
\label{eq:fondamentale}
\mathbb{E}\left[M_{i_1}^{N} \otimes ... \otimes M_{i_k}^{N}\right] = \sum_{p \in\mathcal{A}_k} \mathbb{E}\kappa_{p}^{\mathcal{A}}(M_{i_1}^{N},...,M_{i_k}^{N}) \frac{\rho_{N}(p)}{N^{{\sf nc}(p) - {\sf nc}(p \vee {\mathrm{id}_k})} }. 
\end{align}
Thus, if we consider another set of partitions $\mathcal{A}'$ such that $\mathcal{A} \subset \mathcal{A}'$, for any $k$, any $p \in \mathcal{A}'_k$, any $(i_1,...,i_k) \in \mathcal{I}^{k}$ and any integer $N$: 
\begin{align}
\label{eq:finitedim}
\mathbb{E}\kappa_{p}^{\mathcal{A}'}[M^{N}_{i_1}, ...,M^{N}_{i_k}] = \delta_{p \in \mathcal{A}} \mathbb{E}\kappa_{p}^{\mathcal{A}}[M^{N}_{i_1}, ...,M^{N}_{i_k}].
\end{align}
This implies the following theorem. 

\begin{theorem}
\label{th:invarianceasymp}
For any integer $N$, let $(M_i^{N})_{i \in \mathcal{I}}$ be a family of random matrices. Let us suppose that $(M_i^{N})_{i \in \mathcal{I}}$ is $\mathcal{G}(\mathcal{A})$-invariant and converges in $\mathcal{A}$-distribution. Then it converges in $\mathcal{P}$-distribution and it is asymptotically $\mathcal{G}(\mathcal{A})$-invariant. 
Besides, if it satisfies the asymptotic $\mathcal{A}$-factorization property, it satisfies the asymptotic $\mathcal{P}$-factorization property. 
\end{theorem}

\begin{proof}
If $(M_i^{N})_{i \in \mathcal{I}}$ is $\mathcal{G}(\mathcal{A})$-invariant and converges in $\mathcal{A}$-distribution, then using Equation (\ref{eq:finitedim}) and Theorem \ref{th:lienlimite}, for any integer $k$, any $p \in \mathcal{P}$ and any $(i_1,...,i_k) \in \mathcal{I}^{k}$, $\mathbb{E}\kappa_{p}^{\mathcal{P}}[M^{N}_{i_1}, ...,M^{N}_{i_k}]$ converges. Again by Theorem \ref{th:lienlimite}, the family $(M_i^{N})_{i \in \mathcal{I}}$ converges in $\mathcal{P}$-distribution. Besides, by taking the limit of Equation (\ref{eq:finitedim}), we see that it is asymptotically $\mathcal{G}(\mathcal{A})$-invariant. The last assertion is straightforward. 
  \end{proof}

In particular, if for any integer $N$, $(M_i^{N})_{i \in \mathcal{I}}$ is invariant by conjugation in law by $U(N)$, and $(M_i^{N})_{i \in \mathcal{I}}$ converges in non-commutative distribution and satisfies the asymptotic $\mathfrak{S}$-factorization property, then $(M_i^{N})_{i \in \mathcal{I}}$ converges in $\mathcal{P}$-distribution and for any integer $k$, any $p \in \mathcal{P}_k$, any $(i_1,...,i_k) \in \mathcal{I}^{k}$, 
\begin{align*}
\lim_{N \to \infty} \mathbb{E}m_{p}(M_{i_1},...,M_{i_k}) = \sum_{\sigma \in \mathfrak{S}_k | \sigma\leq p} \prod_{(j_1,...,j_r) \text{ cycle of } \sigma} \kappa\left( M_{i_{j_1}}, ..., M_{i_{j_r}}\right), 
\end{align*}
where the $\kappa$ in the r.h.s. stands for the free cumulants in free probability.  It has to be noticed that this unitary-invariant case was also proved independently in the paper \cite{CGG} using the theory of traffics \cite{Camille}.

Let us suppose that $(M_i^{N})_{i \in \mathcal{I}}$ is $\mathcal{G}(\mathcal{A})$-invariant and converges in $\mathcal{A}$-distribution. Because of the $\mathcal{G}(\mathcal{A})$-invariance, for any $p \in \mathcal{P}_{k}$, any $(i_1,...,i_k) \in \mathcal{I}^{k}$, the exclusive moment $\mathbb{E}m_{p^{c}}^{\mathcal{A}}(M_{i_1}, ..., M_{i_k})$ is equal to $\mathbb{E}m_{p^{c}}(M_{i_1}, ..., M_{i_k})$ since this equality already holds for any finite $N$. Using Theorem \ref{th:simplification}, we get the following result. 

\begin{theorem}
\label{th:invarianceasymp2}
If for any integer $N$, $(M_i^{N})_{i \in \mathcal{I}}$ is invariant in law by conjugation by the unitary group, and $(M_i^{N})_{i \in \mathcal{I}}$ converges in non-commutative distribution and satisfies the asymptotic $\mathfrak{S}$-factorization, then for any integer $k$, any $p \in \mathcal{P}_k$, and any $(i_1,...,i_k) \in \mathcal{I}^{k}$, 
\begin{align*}
\lim_{N \to \infty} \mathbb{E}m_{p^{c}}(M_{i_1}, ..., M_{i_k}) = \delta_{p \in \overline{\mathfrak{S}_k}} \prod_{(j_1,...,j_r) \text{ cycle of } {\sf Mb}(p)} \kappa(M_{i_{j_1}}, ..., M_{i_{j_k}}), 
\end{align*}
where  the $\kappa$ in the r.h.s. stands for the free cumulants in free probability. 
\end{theorem}

\subsection{Microscopic observables}
We can easily relate the moments of the products of the entries of any $\mathfrak{S}$-invariant matrix with the finite dimensional cumulants: this allows us to study the asymptotics of the moments of the products of the entries of any $\mathfrak{S}$-invariant matrix. Let $N$ be a postive integer and $\mathcal{N}=(n_1,n_{1'}, ..., n_{k},n_{k'}) \in \{1,...,N\}^{2k}$. 
\begin{definition}
The kernel ${\sf Ker}(\mathcal{N})$ is the partition in $\mathcal{P}_k$ such that $u$ and $v$, both in $\{1,...,k,1',...,k' \}$, are in the same bloc if and only if $n_{u}= n_{v}$.  
\end{definition}

For example, if $\mathcal{N} = (1,1,2,1)$, then ${\sf Ker}(\mathcal{N}) = \{\{1,1',2'\}, \{2\}\}$. The following theorem is a generalization of Theorem 2.6 of \cite{CMSS}. Recall the notation $\trianglelefteq$ in (2) in \ref{sec:basic}.

\begin{theorem}
\label{th:entrees}
For any integer $N$, let $(M_i^{N})_{i \in \mathcal{I}}$ be a family of $\mathcal{G}(\mathcal{A})$-invariant random matrices of size $N$. For any integers $N$, $k$, any $(i_1,...,i_k) \in \mathcal{I}^{k}$, any $2k$-tuples $\mathcal{N} = (n_1,n_{1'},...,n_{k}, n_{k'})$ in $\{1,...,N\}^{2k}$: 
\begin{align*}
\mathbb{E}\left[ (M_{i_1}^{N})_{n_{1'}}^{n_1} ... (M_{i_k}^{N})_{n_{k'}}^{n_k}\right] =  \sum_{p_1 \in \mathcal{A}_k | p_1 \trianglelefteq {\sf Ker}(\mathcal{N})} \frac{1}{N^{{\sf nc}(p_1)- {\sf nc}(p_1\vee \mathrm{id}_k)}} \mathbb{E}\kappa_{p_1}^{\mathcal{A}}[ M_{i_1}^{N}, ..., M_{i_k}^{N}]. 
\end{align*}
Thus, if  $(M_i^{N})_{i \in \mathcal{I}}$ converges in $\mathcal{A}$-distribution, $\mathbb{E}\left[ (M_{i_1}^{N})_{n_{1'}}^{n_1} ... (M_{i_k}^{N})_{n_{k'}}^{n_k}\right]$ is equal to:
\begin{align*}
N^{{\sf nc}({\sf Ker}(\mathcal{N})\vee \mathrm{id}_k) - {\sf nc}({\sf Ker}(\mathcal{N}))}\bigg( \sum_{p_1 \in \mathcal{A}_k | p_1 \sqsupset p}\!\!\!\!\!\!\!\mathbb{E} \kappa^{\mathcal{A}}_{p_1} (M_{i_1}, ..., M_{i_k}) +o(1)\!\bigg).
\end{align*}
\end{theorem}

\begin{proof}
Since $(M_i^{N})_{i \in \mathcal{I}}$ is a family of $\mathcal{G}(\mathcal{A})$-invariant, it is $\mathfrak{S}$-invariant. Thus  the value of $\mathbb{E}\left[ (M_{i_1}^{N})_{n_{1'}}^{n_1} ... (M_{i_k}^{N})_{n_{k'}}^{n_k}\right]$ only depends on the kernel of $\mathcal{N}=(n_1,n_{1'},...,n_k,n_{k'})$. This implies that: 
\begin{align*}
\mathbb{E}\left[ (M_{i_1}^{N})_{n_{1'}}^{n_1} ... (M_{i_k}^{N})_{n_{k'}}^{n_k}\right] = \frac{N-{\sf nc}({\sf Ker}(\mathcal{N}))!}{N!} N^{{\sf nc}({\sf Ker}(\mathcal{N})\vee \mathrm{id}_k)} \mathbb{E}m_{{\sf Ker}(\mathcal{N})^{c}} (M_{i_1}^{N}, ..., M_{i_k}^{N}). 
\end{align*}
Recall that $ \mathbb{E}m_{{\sf Ker}(\mathcal{N})^{c}}$ is the exclusive moment of $\mathbb{E}\left[M_{i_1}^{N} \otimes ... \otimes M_{i_k}^{N}\right]$: using Equation~(\ref{eq:fondamentale}) in this article and Equation (31) of \cite{Gab1}, we get the first equation we had to prove. The second equation is a consequence of the first one and the definition of the order $\sqsupset$.
  \end{proof}

\begin{remark}
\label{remarque:microplus}
Actually, using the up-coming Section \ref{sec:algfluctu}, we can have a better understanding of the moments of the entries if the family converges up to a higher order of fluctuations. If the family converges to any order of fluctuations, $\mathbb{E}\left[ (M_{i_1}^{N})_{n_{1'}}^{n_1} ... (M_{i_k}^{N})_{n_{k'}}^{n_k}\right] $ is equivalent to: 
\begin{align*}
\frac{1}{N^{o_{p}(i_1,...,i_k)}} \sum_{p_1 \in \mathcal{A}_k, i \in \mathbb{N}| p_1 \trianglelefteq {\sf Ker}(\mathcal{N}), {\sf df}(p_1,p) + i = o_{p}(i_1,...,i_k)} \mathbb{E}\kappa^{i,\mathcal{A}}_{p_1}(M_{i_1},...,M_{i_k}),
\end{align*}
where: $$o_p(i_1,...,i_k) = min\{ i+{\sf df}(p_1,p) | i\in \mathbb{N}, p_1 \in \mathcal{A}_k, p_1 \trianglelefteq p, \mathbb{E}\kappa^{i,\mathcal{A}}_{p_1}(M_{i_1},...,M_{i_k})\neq 0\}.$$
\end{remark}

In the next remark, we apply Theorem \ref{th:entrees} to matrices which are U-invariant. 

\begin{remark}
Let us denote by $\mathcal{H}^{=}_k$ the set of partitions $p$ in $\mathcal{H}_k$ such that for any block $b$ in $p$, $\#( b \cap \{1,...,k\}) = \#( b \cap \{1',...,k'\})$. Let us suppose that $(M_i^{N})_{i \in \mathcal{I}}$  is $U$-invariant, then for any integer $N$, 
\begin{align*}
\mathbb{E}\left[ (M_{i_1}^{N})_{n_{1'}}^{n_1} ... (M_{i_k}^{N})_{n_{k'}}^{n_k}\right] =\delta_{ {\sf Ker}(\mathcal{N}) \in \mathcal{H}_k^{=}} \sum_{\sigma \in \mathfrak{S}_k | \sigma \trianglelefteq {\sf Ker}(\mathcal{N})} \frac{1}{N^{k - {\sf nc}(\sigma \vee \mathrm{id}_k)}} \mathbb{E}\kappa^{\mathfrak{S}}_{\sigma}(M_{i_1}^{N}, ..., M_{i_k}^{N}). 
\end{align*}
For example, $\mathbb{E}\big[(M_{i_1}^{N})_{1}^{1} (M_{i_1}^{N})_1^{2}\big] = 0$ for any integer $N$. Besides, if the family converges in non-commutative distribution and satisfies the asymptotic $\mathfrak{S}$-factorization property: 
\begin{align*}
\mathbb{E}\left[ (M_{i_1}^{N})_{n_{1'}}^{n_1} ... (M_{i_k}^{N})_{n_{k'}}^{n_k}\right]  =\delta_{ {\sf Ker}(\mathcal{N}) \in \overline{\mathfrak{S}_{k}}} \prod_{(j_1,...,j_r) \text{ cycle of }{\sf Mb}({\sf Ker}(\mathcal{N}))} \kappa(M_{i_{j_1}}, ...,M_{i_{j_r}} ), 
\end{align*} 
where the $\kappa$ in the r.h.s. stands for the free cumulants in free probability. 
Actually, again, if we suppose that the family of matrices converges in $\mathfrak{S}$ to higher order of fluctuations, we can have better results. For example, let us suppose that it does converges up to order $2$ of fluctuations and that $\mathbb{E}m_{\mathrm{id}_1}(M_{i_1}) =\mathbb{E}m_{\mathrm{id}_1}(M_{i_2}) = \mathbb{E}m^{1}_{id_{2}}[M_{i_1},M_{i_2}] = 0$ and $\kappa(M_{i_1},M_{i_2}) \neq 0$. Then, the asymptotic in Remark \ref{remarque:microplus} shows that: 
\begin{align*}
\mathbb{E}[(M_{i_1}^{N})_{1}^{1} (M_{i_2}^{N})_{1}^{1}] \sim \frac{1}{N} \kappa(M_{i_1},M_{i_2}).
\end{align*}
\end{remark}

\subsection{Classical cumulants as finite-dimensional cumulants}

\label{secclassic}
In Section \ref{sec:classical}, we explained that classical cumulants can be seen as special cases of $\mathcal{P}$-cumulants. We will see in this section a Schur-Weyl interpretation of classical cumulants: these cumulants can be seen as finite dimensional $\mathcal{P}$-cumulants.  Recall that $\mathcal{D}_k$ is defined as the set of partitions in $\mathcal{P}_k$  which are coarser than ${\mathrm{id}_k}$.

Let $N$ and $k$ be positive integers. Let $\mathcal{D}^N(\mathbb{C})$ be the set of matrices $M \in \mathcal{M}_N(\mathbb{C})$ which are diagonal. We will study elements of $\left(\mathcal{D}^{N}(\mathbb{C})\right)^{\otimes k}$ which commute with the action of $\mathfrak{S}(N)$ on $\left(\mathbb{C}^{N}\right)^{\otimes k}$.

\begin{lemma}
\label{simpleshurweyl}
For any integers $k$ and $N$, 
\begin{align*}
\left(\mathbb{C}\left[\rho^k_{\mathfrak{S}(N)}\right]\right)' \cap \left(\mathcal{D}^{N}(\mathbb{C})\right)^{\otimes k} \subset \rho_{N}\left[\mathbb{C}\left[\mathcal{D}_{k}\right]\right].
\end{align*}
\end{lemma}

\begin{proof}
Let $E$ be in $\left(\mathbb{C}\left[\rho^k_{\mathfrak{S}(N)}\right]\right)' \cap \left(\mathcal{D}^{N}(\mathbb{C})\right)^{\otimes k}$. The first step is to remark that for any $ p \in \mathcal{D}_k$, $p^{c}$ is in $\mathbb{C}\left[\mathcal{D}_{k}\right]$. Thus, we only need to show that $E$ can be written as a linear combination of elements of the form $\rho_{N}(p^{c})$ with $p \in \mathcal{D}_k$. We can decompose $E$ on the canonical base of ${\sf End}\left((\mathbb{C}^{N})^{\otimes k}\right)$ and since $E\in \left(\mathcal{D}^{N}(\mathbb{C})\right)^{\otimes k}$, $$E = \sum\limits_{(i_1,..., i_k) \in \{1,...,N\}^{k}}{c_{i_1,...,i_k}} E_{i_1}^{i_1} \otimes ...\otimes E_{i_k}^{i_k},$$
where $E_{i}^{j}$ is the endomorphism which sends the $j^{th}$ element of the canonical basis on the $i^{th}$. Any $(i_1,..., i_k) \in \{1,...,N\}^{k}$ defines a partition of $\{1,...,k\}$, denoted by ${\sf Ker}((i_1,...,i_k))$, which is the unique partition of $\{1,...,k\}$ such that two elements $u$ and $v$ of $\{1,...,k\}$ are in the same block if and only if $i_u=i_v$. Since $E$ commutes with the action $\rho^{k}_{\mathfrak{S}(N)}$, $c_{i_1,..., i_k}$ only depends on ${\sf Ker}(i_1,...,i_k)$. With obvious notations: 
\begin{align*}
E &=  \sum_{\pi  \in {\sf P}_k} c_{\pi} \sum_{(i_1,..., i_k)\mid {\sf Ker}((i_1,..., i_k)) = \pi}E_{i_1}^{i_1} \otimes ...\otimes E_{i_k}^{i_k} \\
&=  \sum_{\pi \in {\sf P}_k} c_{\pi} \rho_{N}\left[p_{\pi}^{c}\right], 
\end{align*}
where for any $\pi \in {\sf P}_k$, $p_{\pi}= \{ \{i, i\in b\}\cup \{ i', i \in b\} | b \in \pi \}$.
  \end{proof}

\begin{lemma}
\label{lem:injectif}
If $N \geq k$, the restriction of $\rho_N$ to $\mathbb{C}\left[\mathcal{D}_{k}\right]$ is injective. 
\end{lemma}

\begin{proof}
We will use the same notations used in the proof of the previous lemma. Let us suppose that $N \geq k$ and let $(c_{\pi})_{\pi \in {\sf P}_k} \in \mathbb{C}^{{\sf P}_k}$ such that $\sum_{\pi \in {\sf P}_k} c_{\pi} \rho_{N}^{\mathcal{P}_k}\left[p_{\pi}^{c}\right] = 0$. Let $\pi_0 \in {\sf P}_k$ and $(i_1,...,i_k)$ be in $\{1,...,N\}^{k}$ such that ${\sf Ker}((i_1,...,i_k))=\pi_0$. Such $k$-tuple exists since $N\geq k$. Then: 
$$
0=\left(\sum_{\pi \in {\sf P}_k} c_{\pi} \rho_{N}^{\mathcal{P}_k}\left[p_{\pi}^{c}\right] \right)\left(e_{i_1} \otimes ...\otimes e_{i_k}\right) = c_{\pi_0} \left(e_{i_1} \otimes ...\otimes e_{i_k}\right),
$$
and thus $c_{\pi_0} = 0$. This shows that the restriction of $\rho_N$ to $\mathbb{C}\left[\mathcal{D}_{k}\right]$ is injective. 
  \end{proof}

The Schur-Weyl duality interpretation of classical cumulants is given by the following theorem. Recall the notion of classical cumulants denoted by ${\sf cum}(X_1,...,X_l)$. 

\begin{theorem}
\label{cumulantmagique}
Let $(X_1,..., X_k)$ be a $k$-tuple of random variables in $L^{\infty^{-}}(\Omega, \mathcal{A},\mathbb{P})$. Let $l\geq k$ and $(X_1^{n},..., X_{k}^{n})_{n \in \{1,...,l\} }$ be a family of $l$ i.i.d. $k$-tuples of random variables which have the same law as $(X_1,..., X_k)$. For any $i \in \{1,...,k\}$, let us consider the diagonal matrice $M_{i} = {\sf Diag}\left[ (X_i^{n})_{n=1}^{l}\right].$ The following formula holds: 
\begin{align*}
{\sf cum}_{k}(X_1,..., X_k) = \mathbb{E}\kappa_{{\sf 0}_k}\left[M_1,..., M_k\right], 
\end{align*}
and more generally, for any $p \in \mathcal{D}_k$, $$\mathbb{E}\kappa^{p}\left[M_1,..., M_k\right] = \prod_{b \in p\vee {\mathrm{id}_k}} {\sf cum}_{\frac{\# b}{2}} \left((X_i)_{i \in b \cap \{1,...,k\}}\right)\!.$$
\end{theorem}

\begin{proof}
The endomorphism  $\mathbb{E}\left[\bigotimes_{i=1}^{k} M_{i}\right]$ is in $\left(\mathcal{D}_N(\mathbb{C})\right)^{\otimes k}$ and commutes with the action $\rho^{k}_{\mathfrak{S}(N)}$. By Lemma \ref{simpleshurweyl}, it belongs to $\rho_{N}\left[\mathbb{C}\left[\mathcal{D}_{k}\right]\right]$. Since $l\geq k$, by Lemma \ref{lem:injectif}, the number $ \mathbb{E}\kappa_{{\sf 0}_k}\left[M_1,..., M_k\right]$ is well-defined. For sake of clarity, we will make no difference between a partition $p$ and its representation $\rho_{N}(p)$ as an endomorphism. Following the calculations in the proof of Lemma  \ref{simpleshurweyl}, using the same notations and using the independence of the $(X_1^{n},..., X_{k}^{n})_{n \in \{1,...,l\} }$, one gets: 
\begin{align*}
\mathbb{E}\left[\bigotimes_{i =1}^{k} M_{i}\right] = \sum_{\pi \in {\sf P}_k}\left(\prod_{c \in \pi}\mathbb{E}\left[\prod_{i \in c} X_{i}\right]\right) p_{\pi}^{c}.
\end{align*}
By definition of finite-dimensional cumulants and Lemma \ref{simpleshurweyl}, we have also the following equality:
\begin{align*}
\mathbb{E}\left[\bigotimes_{i =1}^{k} M_{i}\right] = \sum_{\pi \in {\sf P}_k} \mathbb{E}\kappa_{p_{\pi}}\left[M_{1},..., M_k\right] p_{\pi} 
&= \sum_{\pi \in {\sf P}_k} \mathbb{E}\kappa_{p_{\pi}}\left[M_{1},..., M_k\right]\left( \sum_{p_{\pi} \trianglelefteq p'}p'^{c}\right)\\
&=   \sum_{\pi \in {\sf P}_k} \mathbb{E}\kappa_{p_{\pi}} \left[ M_{1},...,M_{k}\right] \left(\sum_{\pi \trianglelefteq \pi'}p_{\pi'}^{c}\right)\\
&= \sum_{\pi \in {\sf P}_k } \left( \sum_{\pi' \trianglelefteq \pi} \mathbb{E}\kappa_{p_{\pi'}} \left[M_{1},...,M_k\right] \right) p_{\pi}^{c}. 
\end{align*}
Using Lemma \ref{lem:injectif}, for any $\pi \in {\sf P}_{k}$, $\sum_{\pi' \trianglelefteq \pi} \mathbb{E}\kappa_{p_{\pi'}} \left[M_{1},...,M_{k}\right] =\prod_{c \in \pi}\mathbb{E}\left[\prod_{i \in c} X_{i}\right].$ This allows us to conclude the proof of the theorem.
  \end{proof}

Theorem \ref{cumulantmagique} shows that one could be able to study the probabilistic fluctuations of the observables of random matrices invariant in law by conjugation by the symmetric group in the framework we developed. Instead of studying the asymptotic of a matrix $M$, one would have to study the asymptotic of the matrix ${\sf Diag}\left[M_1,...,M_k\right]$ where $M_i$ are independent and have the same law as $M$: one can see that it leads to the study of partitionned partitions of $\{1,...,k,1',...,k'\}$.

\section{$\mathcal{G}(\mathcal{A})$-invariance, independence and $\mathcal{A}$-freeness}

\label{sec:GAinvetindep}

\subsection{Convergence in $\mathcal{A}$-distribution, $\mathcal{G}(\mathcal{A})$-invariance and independence implies $\mathcal{A}$-freeness}
This section generalizes Theorem \ref{th:voicu} which asserts that independence with an unitary or orthogonal invariance property imply Voiculescu's freeness. For any integer $N$, let $(M^{N}_{i})_{i \in \mathcal{I}}$ and $(L^{N}_j)_{j \in \mathcal{J}}$ be two families of random matrices.

\begin{theorem}
\label{Lemain}
Let us suppose that the two families $(M^{N}_{i})_{i \in \mathcal{I}}$ and $(L^{N}_j)_{j \in \mathcal{J}}$ converge in $\mathcal{A}$-distribution. Let us suppose that for any integer $N$, $(L^{N}_j)_{j\in \mathcal{J}}$ is $\mathcal{G}(\mathcal{A})$-invariant and the two families $(M_i^{N})_{i \in \mathcal{I}}$ and $(L_j^{N})_{j \in \mathcal{J}}$ are independent. The families $(M^{N}_i)_{i\in \mathcal{I}}$ and $(L_j^{N})_{j \in \mathcal{J}}$  are asymptotically $\mathcal{A}$-free.
\end{theorem}

\begin{remark}
Actually, the compatibility condition and the compatible factorization property hold for any integer $N$ big enough so that the $N$-dimensional $\mathcal{A}$-cumulants are well defined.
\end{remark}

\begin{proof}[Proof of Theorem \ref{Lemain}]
Let $k_1$ and $k_2$ be integers, let $(i_1,...,i_{k_1})$ be in $\mathcal{I}^{k_1}$ and $(j_1,...,j_{k_2})$ be in $\mathcal{J}^{k_2}$. For any integer $N$, by independence, and using the $\mathcal{G}(\mathcal{A})$-invariance of $(L^{N}_j)^{N}_{j\in \mathcal{J}}$, 
\begin{align*}
\int_{\mathcal{G}(\mathcal{A})(N)} M^{\otimes k_1+k_2} \mathbb{E}\left[M^{N}_{i_1}\otimes ...\otimes M^{N}_{i_{k_1}} \otimes L^{N}_{j_{1}} \otimes ...\otimes L^{N}_{j_{k_2}}\right] (M^{-1})^{\otimes k_1+k_2} dM
\end{align*}
is equal to: 
\begin{align*}
\left(\int_{\mathcal{G}(\mathcal{A})(N)} M^{\otimes k_1} \mathbb{E}\left[M^{N}_{i_1}\otimes ...\otimes M^{N}_{i_{k_1}} \right] (M^{-1})^{\otimes k_1} dM \right) \otimes \mathbb{E}\left[ L^{N}_{j_{1}} \otimes ...\otimes L^{N}_{j_{k_2}}\right]. 
\end{align*}
Thus, when $N$ is greater than $2k$, for any $p$ in $\mathcal{A}_{k_1+k_2}$, if there exists $(p_1,p_2) \in \mathcal{A}_{k_1} \times \mathcal{A}_{k_2}$ such that $p=p_1\otimes p_2$, then: 
\begin{align*}
\mathbb{E}\kappa_{p}^{\mathcal{A}}\left[M^{N}_{i_1}, ..., M^{N}_{i_k}, L^{N}_{j_1}, ..., L^{N}_{j_{k_2}}\right] =  \mathbb{E}\kappa_{p}^{\mathcal{A}}\left[(M^{N}_{i_n})_{n=1}^{k_1}\right] \mathbb{E}\kappa_{p}^{\mathcal{A}}\left[(L^{N}_{i_n})_{n=1}^{k_2}\right], 
\end{align*}
and if not, $\mathbb{E}\kappa_{p}^{\mathcal{A}}\left[M^{N}_{i_1}, ..., M^{N}_{i_k}, L^{N}_{j_1}, ..., L^{N}_{j_{k_2}}\right] = 0$.
In particular, the finite dimensional cumulant $\mathbb{E}\kappa_{p}^{\mathcal{A}}\left[M^{N}_{i_1}, ..., M^{N}_{i_k}, L^{N}_{j_1}, ..., L^{N}_{j_{k_2}}\right] $ converges as $N$ goes to infinity. Using Theorem \ref{th:lienlimite}, the family $(M_i^{N})_{i \in \mathcal{I}} \cup (L^{N}_j)_{j \in \mathcal{J}}$ converges in $\mathcal{A}$-distribution: $(M^{N}_i)_{i\in \mathcal{I}}$ and $(L_j^{N})_{j \in \mathcal{J}}$  are asymptotically $\mathcal{A}$-free.
  \end{proof}

\subsection{Convergence in $\mathcal{P}$-distribution, strong asymptotic $\mathcal{G}(\mathcal{A})$-invariance and independence implies $\mathcal{A}$-freeness}
\label{strongGinv}

In Theorem \ref{th:asympinvetliberte}, the asymptotic $\mathcal{G}(\mathcal{A})$-invariance of one of the two families was enough. One can wonder if it is possible to state a version of Theorem \ref{Lemain} where the condition of  $\mathcal{G}(\mathcal{A})$-invariance is replaced by the condition of asymptotically $\mathcal{G}(\mathcal{A})$-invariant.

\begin{remark} 
Let us suppose that Theorem \ref{Lemain} is true when one replaces the condition of  $\mathcal{G}(\mathcal{A})$-invariance by the condition of asymptotically $\mathcal{G}(\mathcal{A})$-invariance. Let us suppose that $\mathcal{A} = \mathcal{P}$. Since the family $(L^{N}_j)_{j \in \mathcal{J}}$ is asymptotically $\mathfrak{S}$-invariant, by hypothesis, $(M^{N}_i)_{i \in \mathcal{I}}$ and $(L^{N}_j)_{j \in \mathcal{J}}$ should be asymptotically $\mathcal{P}$-free. 
 
For any integer $N$, let $M_N$ be the diagonal matrix with $\lfloor\frac{N}{2}\rfloor$ zeros followed by $\lceil\frac{N}{2}\rceil$ ones and $L_N$ be the diagonal matrix with $\lfloor\frac{N}{2}\rfloor$ ones follows by $\lceil\frac{N}{2}\rceil$ zeros. The matrices $M_N$ and $L_N$ converge in $\mathcal{P}$-distribution and have the same $\mathcal{P}$-distribution. Besides, $\mathbb{E}m_{{\mathrm{id}}_1}(M) = \mathbb{E}m_{{\mathrm{id}}_1}(L) = \frac{1}{2}. $

 According to our previous discussion, $(M_N)_{N \in \mathbb{N}}$ and $(L_N)_{N \in \mathbb{N}}$ should be asymptotically $\mathcal{P}$-free. Recall that ${\sf 0}_{2}$ is the partition $\{\{1,2,1',2'\}\}$. We should have the following equality $\mathbb{E}m_{{\sf 0}_{2}}(M,L) = \mathbb{E}m_{{\mathrm{id}}_1}(M) \mathbb{E}m_{{\mathrm{id}}_1}(L).$ Yet the l.h.s. is equal to zero, and the right hand side is equal to $\frac{1}{4}$. Thus $(M_N)_{N \in \mathbb{N}}$ and $(L_N)_{N \in \mathbb{N}}$ can not be $\mathcal{P}$-free: Theorem \ref{Lemain} is not true when one just replaces the condition of $\mathcal{G}(\mathcal{A})$-invariance by the condition of asymptotically $\mathcal{G}(\mathcal{A})$-invariance.
\end{remark}

Actually, one can state a version of Theorem \ref{Lemain} where one replaces the notion of $\mathcal{G}(\mathcal{A})$-invariance by the condition of {\em asymptotic strong $\mathcal{G}(\mathcal{A})$-invariance} that we are going to define.

\begin{notation}
Let $(O_{t})_{t \in \mathcal{T}}$ be a family of random matrices of size $N$. Let $k$ be an integer, let $\mathbb{T} = (t_{1}, ... t_{k})$ be a $k$-tuple of elements of $\mathcal{T}$ and $\mathbb{I}=(i_1, i_{1'}, ..., i_{k}, i_{k'} )$ be a $2k$-tuple of $\{ 1,...,N\}^{2k}$. We denote $ (O_{t_1})^{i_1}_{i_{1'}} ... (O_{t_k})^{i_k}_{i_{k'}}$ by $O_{\mathbb{I}, \mathbb{T}}$. We recall that the kernel ${\sf Ker}(\mathbb{I})$ is the partition in $\mathcal{P}_k$ such that $u$ and $v$ are in the same bloc if and only if $i_{u}= i_{v}$.  
\end{notation}

\begin{definition}
The family $(L^{N}_j)_{j \in \mathcal{J}}$ is asymptotically strongly $\mathcal{G}(\mathcal{A})$-invariant if it is asymptotically $\mathcal{G}(\mathcal{A})$-invariant and for any integer $k$, any $p \in \mathcal{P}_k$ and any $k$-tuple $\mathbb{J} = (j_{1}, ... j_{k})$ of elements of $\mathcal{J}$: 
\begin{align*}
\sup_{\mathbb{I}, \mathbb{I'} \in \{ 1,...,N\}^{2k} \mid {\sf Ker}(\mathbb{I})= {\sf Ker}(\mathbb{I'}) = p} N^{{\sf nc}(p )- {\sf nc}(p \vee {\mathrm{id}_k})}| \mathbb{E}\left[ L^{N}_{\mathbb{I}, \mathbb{J}} - L^{N}_{\mathbb{I'}, \mathbb{J}}\right]| \underset{N \to \infty}{\longrightarrow} 0. 
\end{align*}
\end{definition}

\begin{theorem}
\label{Main1}
Let us suppose that the two families $(M_i^{N})_{i \in \mathcal{I}}$ and $(L_j^{N})_{j \in \mathcal{J}}$ converge in $\mathcal{P}$-distribution and that $(L^{N}_j)_{j\in \mathcal{J}}$ is asymptotically strongly $\mathcal{G}(\mathcal{A})$-invariant. If for every positive integer $N$, the two families $(M_i^{N})_{i \in \mathcal{I}}$ and $(L_j^{N})_{j \in \mathcal{J}}$ are independent then the two families $(M_i^{N})_{i \in \mathcal{I}}$ and $(L_j^{N})_{j \in \mathcal{J}}$ are asymptotically $A$-free.
\end{theorem} 

\begin{proof}
Let us suppose that the  two families $(M^{N}_i)_{i \in \mathcal{I}}$ and $(L_j^{N})_{j \in \mathcal{J}}$ satisfy the hypotheses stated in the theorem and that, for every positive integer $N$, they are independent. By definition, $(L^{N}_j)_{j\in \mathcal{J}}$ is asymptotically $\mathcal{G}(\mathcal{A})$-invariant. By Theorem \ref{th:asympinvetliberte}, the asymptotic $\mathcal{P}$-freeness of $(M^{N}_i)_{i \in \mathcal{I}}$ and $(L_j^{N})_{j \in \mathcal{J}}$ would imply the asymptotic $\mathcal{A}$-freeness of the two families. Thus, we only have to prove that the two families are asymptotically $\mathcal{P}$-free. 

Let $k_1$ and $k_2$ be two integers, $p$ be in $\mathcal{P}_k$ where $k=k_1+k_2$, let $(i_1,...,i_{k_1})$ be in $\mathcal{I}^{k_1}$ and $(j_1,...,j_{k_2})$ be in $\mathcal{J}^{k_2}$. By Theorem \ref{th:freeasymptoautre}, we need to show that the exclusive moment $\mathbb{E}m_{p^{c}}(M^{N}_{i_1},...,M^{N}_{i_{k_1}}, L^{N}_{j_1}, ..., L^{N}_{j_{k_2}})$ converges to: 
\begin{align*}
\delta_{p^{l}_{k_1}\otimes p^{r}_{k_1} \sqsupset p} \mathbb{E}m_{(p_{k_1}^{l})^{c}}\left[M_{i_1}, ..., M_{i_{k_1}}\right] \mathbb{E}m_{(p_{k_1}^{r})^{c}}\left[L_{j_1}, ..., L_{j_{k_2}}\right]. 
\end{align*}
Let $N$ be a positive integer. For sake of clarity, we will use the following notation: $B^{N}_t = M^{N}_{t}$ if $t \in \{i_1,...,i_{k_1}\}$, $B^{N}_t = L^{N}_{t}$ if $t \in \{j_{1}, ..., j_{k_2}\}$ and $\mathcal{T} =  (i_1,...,i_{k_1} , j_{1}, ..., j_{k_2})$. By definition, $\mathbb{E}m_{p^{c}}(M^{N}_{i_1},...,M^{N}_{i_{k_1}}, L^{N}_{j_1}, ..., L^{N}_{j_{k_2}})$ is equal to: 
\begin{align*}
\mathbb{E}m_{p^{c}}\left((B_{t}^{N})_{t \in \mathcal{T}}\right) =\frac{1}{N^{{\sf nc}(p \vee \mathrm{id}_k)}} \sum_{\mathbb{I} \in \{1,...,N\}^{2k} \mid {\sf Ker}(\mathbb{I}) = p} \mathbb{E} \left[ B_{ \mathbb{I}, {\mathcal{T}}}^{N} \right]. 
\end{align*}
Any $\mathbb{I} \in \{1,...,N\}^{2k}$ can be writen as the concatenation of $\mathbb{I}_1 \in \{1,...,N\}^{2k_1}$ and $\mathbb{I}_2 \in \{1,...,N\}^{2k_2}$, that we denote by $\mathbb{I}_1\mathbb{I}_2$. Using the independence property, the right hand side can be written as: 
\begin{align*}
\frac{1}{N^{{\sf nc}(p \vee {\mathrm{id}_k})}} \sum_{\mathbb{I}_1 \in \{1,...,N\}^{2k_1}, \mathbb{I}_2 \in \{1,...,N\}^{2k_2} \mid {\sf Ker}(\mathbb{I}_1\mathbb{I}_2) = p} \mathbb{E} \left[ B_{ \mathbb{I}_1,(i_1,...,i_{k_1})}^{N} \right] \mathbb{E} \left[ B_{\mathbb{I}_2, (j_1,...,j_{k_2})}^{N} \right].
 \end{align*}
Let $\mathbb{I}_2$ be in $ \{1,...,N\}^{2k_2}$. Since $(L^{N}_j)_{j\in \mathcal{J}}$ is asymptotically strongly $\mathcal{G}(\mathcal{A})$-invariant:
\begin{align*}
\mathbb{E}\left[ B_{\mathbb{I}_2, (j_1,...,j_{k_2})}^{N} \right] = N^{{\sf nc}({\sf Ker}(\mathbb{I}_2) \vee {\mathrm{id}_{k_2}}) - {\sf nc}({\sf Ker}(\mathbb{I}_2))} \left[\mathbb{E}m_{{\sf Ker}(\mathbb{I}_2)^{c}}\left[B_{j_1}^{N},...,B_{j_{k_2}}^{N}\right] + o(1)\right]
\end{align*}
where the $o(1)$ is uniform in $\mathbb{I}_2$. Besides, if $\mathbb{I}_1 \in \{1,...,N\}^{2k_1}$ and $\mathbb{I}_2 \in \{1,...,N\}^{2k_2}$ satisfy ${\sf Ker}(\mathbb{I}_1\mathbb{I}_2) = p$, then ${\sf Ker}(\mathbb{I}_1) = p_{k_1}^{l}$ and ${\sf Ker}(\mathbb{I}_2) = p_{k_1}^{r}$. This implies that the exclusive moment  $\mathbb{E}m_{p^{c}}\left((B_{t}^{N})_{t \in \mathcal{T}}\right)$ is equal to: 
\begin{align*}
\frac{1}{N^{{\sf nc}(p \vee \mathrm{id}_k)}}\!\! \sum_{\mathbb{I}_1,\mathbb{I}_2  \mid {\sf Ker}(\mathbb{I}_1\mathbb{I}_2) = p} \!\!\!\!\!\!\!\!\!\!\!\!\!\!\!\!\!\!N^{{\sf nc}(p_{k_1}^{r} \vee \mathrm{id}_{k_2}) - {\sf nc}(p_{k_1}^{r})} \mathbb{E} \left[ B_{\mathbb{I}_1, (i_1,...,i_{k_1})}^{N} \right]\left[\mathbb{E}m_{(p_{k_1}^{r})^{c}}\left[(B_{j_n}^{N})_{n=1}^{k_2}\right] + o(1)\right]
\end{align*}
or to: 
\begin{align*}
&N^{-{\sf nc}(p \vee  \mathrm{id}_k) + {\sf nc}(p_{k_1}^{r} \vee  \mathrm{id}_{k_2}) - {\sf nc}(p_{k_1}^{r})} \\&\ \ \ \ \sum_{\mathbb{I}_1\mid {\sf Ker}(\mathbb{I}_1) = p_{k_1}^{l}} \mathbb{E} \left[ B_{ \mathbb{I}_1,(i_1,...,i_{k_1})}^{N} \right]\sum_{\mathbb{I}_2\mid {\sf Ker}(\mathbb{I}_2) = p_{k_1}^{r}, {\sf Ker}(\mathbb{I}_1\mathbb{I}_2) = p}\left[\mathbb{E}m_{(p_{k_1}^{r})^{c}}\left[(B_{j_n}^{N})_{n=1}^{k_2}\right] + o(1)\right]. 
\end{align*}
When $\mathbb{I}_1 \in \{1,...,N\}^{2k_1}$ is fixed, there exists asymptotically $N^{{\sf nc}(p ) - {\sf nc}(p_{k_1}^{l})}$ elements $\mathbb{I}_2$ in $\{1,...,N\}^{2k_2}$ such that ${\sf Ker}(\mathbb{I}_2) = p_{k_1}^{r}$ and ${\sf Ker}(\mathbb{I}_1 \mathbb{I}_2) = p$. We can go on our calculations and $\mathbb{E}m_{p^{c}}\left((B_{t}^{N})_{t \in \mathcal{T}}\right)$ is equal asymptotically to:  
\begin{align*}
&N^{-{\sf nc}(p \vee {\mathrm{id}_k}) + {\sf nc}(p_{k_1}^{r} \vee \mathrm{id}_{k_2}) - {\sf nc}(p_{k_1}^{r}) + {\sf nc}(p) - {\sf nc}(p_{k_1}^{l})} \\& \ \ \ \ \ \ \ \ \ \ \ \ \ \ \ \ \ \ \ \ \ \ \ \ \ \ \ \ \ \ \ \ \ \ \ \ \ \ \ \left[\mathbb{E}m_{(p_{k_1}^{r})^{c}}\left[(B_{j_n}^{N})_{n=1}^{k_2}\right] + o(1)\right] \sum_{\mathbb{I}_1\mid {\sf Ker}(\mathbb{I}_1) = p_{k_1}^{l}} \mathbb{E} \left[ B_{ \mathbb{I}_1,(i_1,...,i_{k_1})}^{N} \right]\\
\end{align*}
or: 
\begin{align*}
&\!\!\!\!N^{\left({\sf nc}(p)-{\sf nc}(p \vee {\mathrm{id}_k})\right)- \left({\sf nc}(p_{k_1}^{l} \otimes p_{k_1}^{r}) - {\sf nc}((p_{k_1}^{l} \otimes p_{k_1}^{r}) \vee {\mathrm{id}_k})\right)} \\ &\ \ \ \ \ \ \ \ \ \ \ \ \ \ \ \ \ \ \ \ \ \ \ \ \ \ \ \ \ \ \ \ \ \ \left[\mathbb{E}m_{(p_{k_1}^{r})^{c}}\left[L_{j_1}^{N},...,L_{j_{k_2}}^{N}\right] + o(1)\right] \mathbb{E}m_{(p_{k_1}^{l})^{c}}\left[M_{i_1}^{N},...,M_{i_{k_1}}^{N}\right].
\end{align*}
Since $p_{l}^{g} \otimes p_{l}^{d}$ is finer than $p$, by definition of the order $\sqsupset$:
\begin{align*}
N^{\left({\sf nc}(p)-{\sf nc}(p \vee {\mathrm{id}_k})\right)- \left({\sf nc}(p_{k_1}^{l} \otimes p_{k_1}^{r}) - {\sf nc}((p_{k_1}^{l} \otimes p_{k_1}^{r}) \vee {\mathrm{id}_k})\right)} &\underset{N \to \infty}{\longrightarrow} \delta_{p_{k_1}^{l} \otimes p_{k_1}^{r} \sqsupset p}.
\end{align*}
This allows us to conclude the proof of the theorem. 
  \end{proof}
 
 \begin{remark}
 Unlike Theorem \ref{Lemain}, this new Theorem \ref{Main1} is general enough to be used in order to recover the freeness of general  real and compex  Wigner matrices. Yet, let us remark that Theorem \ref{Main1} is not a generalization of Theorem \ref{Lemain} since in this latter, we need that the families of random matrices converge in $\mathcal{P}$-distribution.
 \end{remark}

\section{General theorems for convergence of L\'{e}vy processes }
\label{sec:Levy}
In this section, we give a general theorem about convergence in $\mathcal{P}$-distribution for sequences of $\mathcal{G}(\mathcal{A})$-invariant matricial L\'{e}vy processes. 
\subsection{Generalities about L\'{e}vy processes}
Let $G$ be a topological group, let us recall the notion of (right-) L\'{e}vy processes (\cite{Liao}). 
\begin{definition}
Let $\left(X_{t}\right)_{t \geq 0}$ be a cadlàg process in $G$ which is stochastically continuous such that $X_0$ is the neutral element of $E$. The process $(X_t)_{t \geq 0}$ is a L\'{e}vy process if for any $0<t<s$, $X_{s}X_{t}^{-1}$ is independent of $(X_u)_{u \leq t}$ and $X_{s}X_{t}^{-1}$ has the same law as $X_{s-t}$. 
\end{definition}

From now on, we will only consider groups of matrices. Let $N$ be a positive integer. If $G$ is a subgroup of $(\mathcal{M}_N(\mathbb{C}),+)$, the L\'{e}vy processes are called additive L\'{e}vy processes. If $G$ is a subgroup of $\mathcal{G}l(N)$, the L\'{e}vy processes are called multiplicative L\'{e}vy processes. In order to prove the next statement, we need the insertion operation. 

\begin{definition}
\label{def:insertion}
Let $I=\{i_1,...,i_l\} \subset \{1,...,k\}$ with $i_1 < ...< i_l$. The permutation  $\sigma_I$ is the one which sends $i_j$ on $j$ for any $j \in \{1,...,l\}$ and $i \notin I$ on $l+i-\#\{n,i_n<i\}$. For any  $A \in {\sf End}\left((\mathbb{C}^{N})^{\otimes l}\right)$, $B \in  {\sf End}\left((\mathbb{C}^{N})^{\otimes k-l}\right)$, we define: 
\begin{align*}
\mathcal{I}_{I}\left(A,B\right) = \rho_N\left(\sigma_{I}^{-1} \right) \left(A\otimes B\right) \rho_N\left(\sigma_{I}\right).
\end{align*}
\end{definition}

Let us remark that for any $l \leq k$, any subset $I \subset \{1,...,k\}$ of cardinal $l$ and any partitions $(p,p') \in \mathcal{P}_l \times \mathcal{P}_{k-l}$,\ $\mathcal{I}_{I} (\rho_{N}( p), \rho_{N} (p')) = \rho_{N} \left( \sigma_{I}^{-1} (p \otimes p')  \sigma_{I} \right).$ We will use often this remark without referring to it. 

Let $(X_t)_{t \geq 0}$ be an additive or multiplicative L\'{e}vy process in $\mathcal{M}_N(\mathbb{C})$. From now on, we only consider L\'{e}vy processes such that for any $t \geq 0$, $X_t$ is in $L^{\infty^{-}}\otimes \mathcal{M}_N(\mathbb{C})$. For any $k \in \mathbb{N}$, let:  
\begin{align*}
G_k = \frac{d}{dt}_{\mid t=0} \mathbb{E}\left[X_t^{\otimes k}\right].
\end{align*} 
The family $(G_k)_{k \in \mathbb{N}}$ is the only data one needs in order to compute $\mathbb{E}[X_t^{\otimes k}]$ for any positive real $t$.

\begin{lemma}
\label{geneentemps}
For any positive integer $k$ and any real $t_0 \geq 0$, 
\begin{enumerate}
\item if $\left(X_t\right)_{t \geq 0}$ is an additive L\'{e}vy process: 
\begin{align*}
\frac{d}{dt}_{\mid t=t_0} \mathbb{E}\left[X_t^{\otimes k}\right] = \sum_{l=0}^{k-1} \sum_{I \subset \{1,...,k\}, \#I = l} \mathcal{I}_I\left[\mathbb{E}\left[X_{t_0}^{\otimes l}\right], G_{k-l} \right],
\end{align*}
\item if $\left(X_t\right)_{t \geq 0}$ is a multiplicative L\'{e}vy process:
\begin{align*}
\frac{d}{dt}_{\mid t=t_0} \mathbb{E}\left[X_t^{\otimes k}\right] = G_k \mathbb{E}\left[X_{t_0}^{\otimes k}\right]. 
\end{align*}
\end{enumerate}
\end{lemma}

\begin{proof}
Let $(H_t)_{t \geq 0}$ and $(U_t)_{t \geq 0}$ be two matrix-valued L\'{e}vy processes which are respectively additive and multiplicative. Let us define for any integer~$k$: 
\begin{align*}
G^{H}_k =  \frac{d}{dt}_{\mid t=0} \mathbb{E}\left[H_t^{\otimes k}\right],\ 
G^{U}_k =  \frac{d}{dt}_{\mid t=0} \mathbb{E}\left[U_t^{\otimes k}\right]. 
\end{align*}
For any $t_0 \geq 0$:
\begin{align*}
\frac{d}{dt}_{\mid t=t_0} \mathbb{E}\left[H_t^{\otimes k}\right]
&=  \lim_{s\to 0}\frac{\mathbb{E}\left[H_{t_0+s}^{\otimes k}\right] - \mathbb{E}\left[H_{t_0}^{\otimes k}\right] }{s} \\
&= \lim_{s\to 0}\frac{\mathbb{E}\left[(H_{t_0}+(H_{t_0+s}-H_{t_0}))^{\otimes k}\right] - \mathbb{E}\left[H_{t_0}^{\otimes k}\right] }{s} \\
&=\lim_{ s \to 0} \frac{\sum_{l=0}^{k-1}\sum_{I \subset \{1,...,k\}, \#I = l} \mathcal{I}_I\left[\mathbb{E}\left[H_{t_0}^{\otimes l}\right], \mathbb{E}\left[H_{s}^{\otimes k-l}\right]\right]}{s}\\
&= \sum_{l=0}^{k-1} \sum_{I \subset \{1,...,k\}, \#I = l} \mathcal{I}_I\left[\mathbb{E}\left[H_{t_0}^{\otimes l}\right],\lim_{ s \to 0} \frac{\mathbb{E}\left[H_{s}^{\otimes k-l}\right]}{s}\right]\\
&= \sum_{l=0}^{k-1} \sum_{I \subset \{1,...,k\}, \#I = l} \mathcal{I}_I\left[\mathbb{E}\left[H_{t_0}^{\otimes l}\right], G^{H}_{k-l} \right],
\end{align*}
and if $(U'_t)_{t \geq 0}$ is a process which has the same law as $(U_t)_{t \geq 0}$ and which is independent of $(U_t)_{t \geq 0}$, we also have: 
\begin{align*}
\frac{d}{dt}_{\mid t=t_0} \mathbb{E}\left[U_t^{\otimes k}\right] = \lim_{s \to 0} \frac{\mathbb{E}\left[U_{t_0+s}^{\otimes k}\right] - \mathbb{E}\left[U_{t_0}^{\otimes k}\right]}{s} &= \lim_{s \to 0}  \frac{\mathbb{E}\left[\left(U'_{s}U_{t_0}\right)^{\otimes k}\right] - \mathbb{E}\left[\left(U_{t_0}^{\otimes k}\right)\right]}{s}\\
&= \lim_{s \to 0}  \frac{\mathbb{E}\left[(U'_{s})^{\otimes k} U_{t_0}^{\otimes k}\right] - \mathbb{E}\left[\left(U_{t_0}^{\otimes k}\right)\right]}{s}\\
&=\lim_{s \to 0}  \frac{\mathbb{E}\left[(U'_{s})^{\otimes k}\right]  - Id^{\otimes k}}{s} \mathbb{E}\left[ U_{t_0}^{\otimes k}\right]\\
&= G^{U}_k \mathbb{E}\left[U_{t_0}^{\otimes k}\right]. 
\end{align*}
In these equalities, we used intensively the independence and stationarity properties of the two processes. 
  \end{proof}

Later on, we will use the more general fact that for any multiplicative matrix-valued L\'{e}vy process $(X_t)_{t \geq 0}$, any integers $k$ and $l$, the family $\left(\mathbb{E}\left[X^{\otimes k} \otimes \overline{X}^{\otimes l}\right]\right)_{t \geq 0}$ is a semi-group and, using the same arguments as for the proof of Lemma \ref{geneentemps}, for any $t_0\geq 0$, 
\begin{align}
\label{semimultibar}
\frac{d}{dt}_{\mid t=t_0} \mathbb{E}\left[X_{t}^{\otimes k} \otimes \overline{X_{t}}^{\otimes l}\right] =\left( \frac{d}{dt}_{\mid t = 0} \mathbb{E}\left[X_t^{\otimes k}\otimes \overline{X_{t}}^{\otimes l}\right]\right)\mathbb{E}\left[X_{t_0}^{\otimes k}\otimes \overline{X_{t_0}}^{\otimes l}\right]\!. 
\end{align}

\subsection{General theorem for the convergence of L\'{e}vy processes}
\subsubsection{Convergence in $\mathcal{P}$-distribution}
Using Definition \ref{invariancedef}, we define the notion of $\mathcal{G}(\mathcal{A})$-invariant L\'{e}vy processes. Let $\left((X_t^{N})_{t \geq 0}\right)_{N \geq 0}$ be a sequence of $\mathcal{G}(\mathcal{A})$-invariant L\'{e}vy processes which are either all additive or multiplicative. For any integers $k$ and $N$, let:  
\begin{align*}
G^{N}_k = \frac{d}{dt}_{\mid t=0} \mathbb{E}\left[(X^{N}_t)^{\otimes k}\right].
\end{align*} 
Let us suppose that $N$ is greater than $2k$. The endomorphism $G^{N}_k$ commutes with the tensor action of $\mathcal{G}(\mathcal{A})$ on $\left(\mathbb{C}^{N}\right)^{\otimes k}$. Using the duality stated in Theorem \ref{duality}, $G^{N}_{k}$ belongs to $\mathbb{C}\left[\rho_{N}^{\mathcal{A}_k}\right]$ and since $N\geq 2k$, $\rho_{N}$ is injective: we can consider $G^{N}_{k}$ as an element of $\mathbb{C}[\mathcal{A}_k(N)]$. The cumulants $\left(\kappa_{p}\left(G^{N}_{k}\right)\right)_{p \in \mathcal{A}_k}$ of $G^{N}_{k}$ are defined using Definition 5.3 of \cite{Gab1}. Recall Section $5$ of \cite{Gab1} where the notion of convergence for elements of $\prod_{N \in \mathbb{N}}\mathbb{C}[\mathcal{A}_k(N)]$ was defined and studied. From now on, the sequence $\left(G_k^{N}\right)_{N \geq 2k}$ will always be considered as an element of $\prod_{N \geq 2k} \mathbb{C}[\mathcal{A}_k(N)]$. Thus, for any integer $k$, the sequence $\left(G_k^{N}\right)_{N \in \mathbb{N}}$ converges if $\left(G_k^{N}\right)_{N \geq 2k}$ seen as an element of $\prod_{N \geq 2k} \mathbb{C}[\mathcal{A}_k(N)]$ converges. 

\begin{remark}
\label{rq:convergenceG}
Using Theorem 5.2 and 5.5 in \cite{Gab1}, the three following assertions are equivalent to the convergence of $(G_k^{N})_{N \in \mathbb{N}}$: 
\begin{enumerate}
\item for any $p \in \mathcal{A}_k$, $\kappa_{p}(G_N)$ converges, 
\item for any $p \in \mathcal{A}_k$, $m_{p}(G_N)$ converges, 
\item for any $p \in \mathcal{P}_k$, $m_{p^c}(G_N)$ converges. 
\end{enumerate}
\end{remark}

Recall the notion of $\boxdot$-character defined in Definition 4.8 of \cite{Gab1}.

\begin{definition}\label{definitioncondens}
Let us suppose that for any integer $k$, $\left(G_{k}^{N}\right)_{N \in \mathbb{N}}$ converges. The $\mathcal{R}$-transform of $G$, denoted by $\mathcal{R}(G) $, is the linear form in $\left(\bigoplus_{k=0}^{\infty} \mathbb{C}[\mathcal{P}_k / \mathfrak{S}_k]\right)^{*}$ which sends $p \in \mathcal{P}_k$ on $\kappa_{p}(G_k) := \lim_{N \to \infty} \kappa_{p}\left(G_k^{N}\right)$. We say that $\left(G_k^{N}\right)_{k}$  {\em condensates} (resp. {\em weakly condensates}) if $\mathcal{R}(G)$ is a $\boxplus$-infinitesimal character (resp. a $\boxtimes$-infinitesimal character). 
\end{definition}

\begin{remark}
\label{rq:condens}
Using a slight generalization of Section 4.3.2 of \cite{Gab1}, we have three criterions in order to know if $\left(G_k^{N}\right)_{k}$ condensates or weakly condensates. Indeed, because of Theorem 4.2 and 4.3 of \cite{Gab1}, the three following assertions are equivalent: 
\begin{enumerate}
\item $\left(G_k^{N}\right)_{k}$ condensates, 
\item for any integer $k$, any $p \in \mathcal{A}_k$, $m_{p}(G_k) = 0$ if $p$ is not irreducible, 
\item for any  integer $k$, any $p \in \mathcal{P}_k$, $m_{p^{c}}(G_k) = 0$ if $p$ is not irreducible, 
\end{enumerate}
and the three following assertions are also equivalent: 
\begin{enumerate}
\item $\left(G_k^{N}\right)_{k}$ weakly condensates, 
\item for any integers $k_1$ and $k_2$, any $p_1 \in \mathcal{A}_{k_1}$ and any $p_2 \in \mathcal{A}_{k_2}$, $m_{p_1 \otimes p_2}(G_{k_1+k_2}) = m_{p_1}(G_{k_1})+ m_{p_2}(G_{k_2})$,
\item for any  integer $k$, any $p \in \mathcal{P}_k$, $m_{p^{c}}(G_k) = 0$ if $p$ is not exclusive-irreducible (the definition is given in Definition 3.3 of \cite{Gab1}: this means that there exists a cycle $c_0$ of $p$ such that any other cycle of $p$ is equal to a partition of the form ${\sf 0}_l$)
\end{enumerate}
\end{remark}

We can now state the main result about convergence in $\mathcal{P}$-distribution of L\'{e}vy processes. 

\begin{theorem}\label{convergencegenerale}
Let $\left((X_t^{N})_{t \geq 0}\right)_{N \geq 0}$ be a sequence of $\mathcal{G}(\mathcal{A})$-invariant L\'{e}vy processes which are either all additive or multiplicative. Let $\boxdot$ be either $\boxplus$ in the additive case or $\boxtimes$ in the multiplicative case. For any integers $k$ and $N$, let: 
\begin{align*}
G_k^{N} = \frac{d}{dt}_{\mid t=0} \mathbb{E}\left[\left(X_t^{N}\right)^{\otimes k}\right]. 
\end{align*}
 If for any positive integer $k$, the sequence $\left(G_k^{N}\right)_{N \geq 0}$ converges, then the process $(X_t^{N})_{t \geq 0}$ converges in $\mathcal{P}$-distribution toward the $\mathcal{P}$-distribution of a $\mathcal{G}(\mathcal{A})$-invariant $\boxdot$-$\mathcal{A}$-L\'{e}vy process.  For any real $t_0\geq 0$,  
\begin{align}
\label{equationverifie}
\mathcal{R}[X_{t_0}] = e^{\boxdot t_0\mathcal{R}(G)}
\end{align}
In the additive case, respectively multiplicative case, if $(G_k^{N})_k$ condensates, respectively weakly condensates, then the asymptotic $\mathcal{P}$-factorization property holds for $(X_t^{N})_{t \geq 0}$. 
\end{theorem}

\begin{remark}
\label{rem:conver}
A consequence of Equation (\ref{equationverifie}) is that, in the additive case, for any real $t_0\geq 0$, any integer $k$ and any irreducible partition $p \in \mathcal{P}_k$,  
\begin{align*}
\mathcal{R}\left[X_{t_0}\right](p) = t_0 (\mathcal{R}\left[G\right] (p)).
\end{align*}
Besides, using Proposition $4.3$ of \cite{Gab1}, Equation (\ref{equationverifie}) implies, in the multiplicative case, that for any integer $k$, any $p \in \mathcal{P}_k$ and any real $t_0\geq 0$: 
\begin{align*}
\frac{d}{dt}_{\mid t=t_0}\mathbb{E}m_{p} (X_t) &= \sum_{p_1 \in \mathcal{A}_k | p_1\leq p} \kappa_{p_1} (G_k) \mathbb{E}m_{\!\!\text{ }^{t}p_1 \circ p}(X_{t_0}). 
\end{align*}
\end{remark}

\begin{proof}
Let us use the same notations as in Theorem \ref{convergencegenerale} and let us suppose that for any integer $k$, the sequence $(G_k^{N})_{N \geq 0}$ converges: for any $p \in \mathcal{P}_k$, $\kappa_{p}(G_k^{N})$ converges. According to the case we consider: 
\begin{enumerate}
\item {\textbf{in the additive case,}} using Lemma \ref{geneentemps}, for any integers $N$ and $k$ and any $t_0 \geq 0$: 
 \begin{align}
\label{lequationadditive}
\frac{d}{dt}_{\mid t=t_0} \mathbb{E}\left[(X_t^{N})^{\otimes k}\right] = \sum_{l=0}^{k-1} \sum_{I \subset \{0,...,k\}, \#I = l} \mathcal{I}_I\left[\mathbb{E}\left[(X_{t_0}^{N})^{\otimes l}\right], G_{k-l}^{N} \right].
\end{align}
Let $\mathcal{R}_{(N)}(X_{t}^{N})$, respectively $\mathcal{R}_{(N)}(G^{N})$, be the linear form in $\left(\bigoplus_{k=0}^{\infty} \mathbb{C}[\mathcal{P}_k / \mathfrak{S}_k]\right)^{*}$ which sends any partition $p$ on $\mathbb{E}\kappa_{p}(X_t^{N})$, respectively on $\kappa_{p}(G_k^{N})$. The last system of equations  implies that for any $t_0 \geq 0$:
\begin{align*}
\mathcal{R}_{(N)}[X_{t_0}^{N}] = e^{\boxplus t_0 \mathcal{R}_{(N)}(G^{N})}
\end{align*}
Since  $\mathcal{R}_{(N)}(G^{N})$ converges to $\mathcal{R}(G)$, for any real $t_0 \geq 0$, $\mathcal{R}_{(N)}[X_{t_0}^{N}]$ converges to $e^{\boxplus t_0 \mathcal{R}(G)}$. By Theorem \ref{th:lienlimite}, for any $t_0 \geq 0$, $X_{t_0}^{N}$ converges in $\mathcal{P}$-distribution.  
\item {\textbf{in the multiplicative case,}}  using Lemma \ref{geneentemps}, for any integers $N$ and $l$ and any $t_0\geq 0$: 
\begin{align*}
\frac{d}{dt}_{\mid t=t_0} \mathbb{E}\left[\left(X_t^{N}\right)^{\otimes k}\right] = G^{N}_k \mathbb{E}\left[\left(X_{t_0}^{N}\right)^{\otimes k}\right]. 
\end{align*}
Thus, when $N \geq 2k$, $ \mathbb{E}\left[(X_t^{N})^{\otimes k}\right]$ can be seen as a semi-group in $\mathbb{C}\left[\mathcal{A}_{k}(N)\right]$. The Theorem $5.10$ of \cite{Gab1} allows us to conclude that for any $t_0 \geq 0$, $X_{t_0}^{N}$ converges in $\mathcal{P}$-distribution and $\mathcal{R}[X_{t_0}] = e^{\boxtimes t_0 \mathcal{R}(G)}$. 
\end{enumerate}

In a nutshell, in the two cases, for any $t_0 \geq 0$, $X_{t_0}^{N}$ converges in $\mathcal{P}$-distribution.  Let $s \geq t \geq 0$, and let us define the increment $\mathcal{X}^{N}_{s,t}$ by $X^{N}_s - X^{N}_t$ in the additive case and $X_s^{N}(X_{t}^{N})^{-1}$ in the multiplicative case. Let $0 \leq t_1 <t_2< ... < t_n$ be an increasing sequence of non negative reals. The convergence in $\mathcal{P}$-distribution of $(X_{t_i}^{N})_{i=1}^{n}$ is equivalent to the convergence in $\mathcal{P}$-distribution of $(\mathcal{X}_{t_{i+1},t_{i}}^{N})_{i=0}^{n-1}$, with the convention that $t_0=0$. Since $(X_t^{N})_{t \geq 0}$ is a $\mathcal{G}(\mathcal{A})$-invariant L\'{e}vy process, $(\mathcal{X}_{t_{i+1},t_{i}}^{N})_{i=0}^{n}$ is a vector of independent $\mathcal{G}(\mathcal{A})$-invariant random matrices and  for any $i \in \{0,...,n-1\}$, $\mathcal{X}_{t_{i+1},t_{i}}^{N}$ has the same law as $X^{N}_{t_{i+1}-t_{i}}$. Using Theorem \ref{Lemain}, the family $(\mathcal{X}_{t_{i+1},t_{i}}^{N})_{i=0}^{n-1}$ converges in $\mathcal{P}$-distribution toward the $\mathcal{P}$-distribution a vector of $\mathcal{A}$-free elements. Using again the stationarity of the process $(X^{N}_t)_{t \geq 0}$ and also Theorem \ref{th:invarianceasymp}, this proves that the process $(X_t^{N})_{t \geq 0}$ converges in $\mathcal{P}$-distribution  toward the $\mathcal{P}$-distribution of a $\mathcal{G}(\mathcal{A})$-invariant $\boxdot$-$\mathcal{A}$-L\'{e}vy process. 

The last assertion about the condensation of the generator and the asymptotic $\mathcal{P}$-factorization property is a direct consequence of Theorem \ref{th:deterministicsemigroup}.
  \end{proof}

\subsubsection{$*$-Convergence in $\mathcal{P}$-distribution}
\label{sec:*convergence}
If for any $t \geq 0$ and any integer $N$, the matrices $X_t^{N}$ are complex-valued, the asymptotic $\mathcal{P}$-factorization poperty does not imply the convergence in probability in $\mathcal{P}$-distribution. Using Theorem \ref{theoremconvprob}, in order to show convergence in probability, we need to prove that the family $(X_t)_{t \geq 0} \cup (X_t^{*})_{t \geq 0}$ or $(X_t)_{t \geq 0} \cup (\overline{X_t})_{t \geq 0}$  converge in $\mathcal{P}$-distribution and satisfies the asymptotic $\mathcal{P}$-factorization property. The complex-va\-lued additive L\'{e}vy processes considered in next sections will be Hermitian: the convergence and asymptotic factorization properties of $(X_t)_{t \geq 0} \cup (X_t^{*})_{t \geq 0}$ is a direct consequence of the convergence and  asymptotic factorization properties of $(X_t)_{t \geq 0}$. From now on, until the end of the section, we will suppose that for any integer $N$, $(X_t^{N})_{t \geq 0}$ is a multiplicative L\'{e}vy process. 

Recall the operation ${\sf S}_k$ defined in Definition \ref{def:Sk}. Let $k$, $l$ be two integers, let $N$ be greater than $2(k+l)$ and let $t \geq 0$. The endomorphism $ \mathbb{E}\left[\left(X_t^{N}\right)^{\otimes k}\otimes \left((X_t^{N})^{*}\right)^{\otimes l}\right]$ commutes with the tensor action of $\mathcal{G}(\mathcal{A})(N)$ on $\left(\mathbb{C}^{N}\right)^{\otimes (k+l)}$. According to Theorem \ref{duality}, it is an element of $\mathbb{C}\left[\rho_N^{\mathcal{A}_{k+l}}\right]$, or with a slight abuse of notation it is an element of $\mathbb{C}[\mathcal{A}_{k+l}(N)]$. This implies that $\mathbb{E}\left[\left(X_t^{N}\right)^{\otimes k}\otimes \left(\overline{X_t^{N}}\right)^{\otimes l}\right]$ which is equal to ${\sf S}_{k}\left[ \mathbb{E}\left[\left(X_t^{N}\right)^{\otimes k}\otimes \left((X_t^{N})^{*}\right)^{\otimes l}\right]\right]$ can be seen as an element of $\mathbb{C}[\mathcal{P}_{k+l}(N)]$.

According to Equation (\ref{semimultibar}), $\left(\mathbb{E}\left[(X^{N}_{t})^{\otimes k} \otimes \left(\overline{X_{t}^{N}}\right)^{\otimes l}\right]\right)_{t \geq 0}$ is a semi-group of endomorphisms. For any integers $k$, $l$ and $N$, let: 
\begin{align}
\label{def:genekl}
G_{k,l}^{N} = \frac{d}{dt}_{\mid t=0} \mathbb{E}\left[\left(X_t^{N}\right)^{\otimes k}\otimes \left(\overline{X_t^{N}}\right)^{\otimes l}\right]. 
\end{align}
The endomorphism $G_{k,l}^{N}$ can be seen as an element of $\mathbb{C}[\mathcal{P}_{k+l}(N)]$. The sequence $(G_{k,l}^{N})_{N \in \mathbb{N}}$ converges if $(G_{k,l}^{N})_{N \geq 2(k+l)}$, seen as an element of $\prod_{N \geq 2(k+l)} \mathbb{C}[\mathcal{P}_k(N)]$, converges. 

\begin{definition}
Let us suppose that for any integer $k$ and $l$,  $(G_{k,l}^{N})_{N \in \mathbb{N}}$ converges. We define $\mathcal{R}_{k,l}(G)$ as the linear form on $\mathbb{C}[\mathcal{P}_{k+l}]$ which sends the partition $p$ on $\kappa_p(G_{k,l})\!\!= \lim_{N \to \infty} \kappa_{p}(G_{k,l}^{N}).$

If $(( X_t^{N}, (X_{t}^{N})^{*})_{t \geq 0})_{N \in \mathbb{N}}$ converges in $\mathcal{P}$-distribution, for any positive integers $k$ and $l$ and any real $t \geq 0$, $\mathcal{R}_{k,l}(X_t)$ is the linear form on $\mathbb{C}[\mathcal{P}_{k+l}]$ which sends $p$ on the cumulant $\mathbb{E}\kappa_p( X_{t}, ...,X_{t}, \overline{X_t}, ..., \overline{X_t} )$, where we wrote $k$ times $X_t$ and $l$ times $\overline{X_t}$.

We say that $(G_{k,l}^{N})_{k,l}$ {\em weakly condensates} if for any integers $k$ and $l$, any $p \in \mathcal{P}_{k+l}$, if $p$ is not weakly irreducible, $\kappa^{p}(G_{k,l}) = 0,$ and $$\kappa^{{\mathrm{id}}_{k+l}}(G_{k,l}) = k \kappa^{\mathrm{id}_{1}}(G_{1,0}) + l \kappa^{\mathrm{id}_1}(G_{0,1}).$$
\end{definition}
 
Let us remark that for any integer $k$, the convolution $\boxtimes$ can be restricted to $\left(\left(\mathbb{C}[\mathcal{P}_k / \mathfrak{S}_k]\right)^{*} \right)^{\otimes 2}$ then extended to $\left( \left(\mathbb{C}[\mathcal{P}_k]\right)^{*}\right)^{\otimes 2}$. The proof of the following theorem is similar to Theorem \ref{convergencegenerale}, thus we will omit it.

\begin{theorem}\label{convergence*levy}\label{autreconvergence*levy}
Let $\left((X_t^{N})_{t \geq 0}\right)_{N \geq 0}$ be a sequence of $\mathcal{G}(\mathcal{A})$-invariant multiplicative L\'{e}vy processes. For any integers $k$, $l$ and $N$, let: 
\begin{align*}
G_{k,l}^{N} = \frac{d}{dt}_{\mid t=0} \mathbb{E}\left[\left(X_t^{N}\right)^{\otimes k}\otimes \left(\overline{X_t^{N}}\right)^{\otimes l}\right]. 
\end{align*}
If for any positive integers $k$ and $l$, the sequence $(G_{k,l}^{N})_{N \geq 0}$ converges, then the family $(X_t^{N})_{t \geq 0} \cup ((X_{t}^{N})^{*})_{t \geq 0}$ converges in $\mathcal{P}$-distribution. Let $k$ and $l$ be two integers, for any real $t_0 \geq 0$: 
\begin{align*}
\mathcal{R}_{k,l}(X_{t_0}) = e^{\boxtimes t_0 \mathcal{R}_{k,l}(G)}, 
\end{align*}
and for any $p \in \mathcal{P}_{k+l}$: 
\begin{align*}
\frac{d}{dt}_{\mid t=t_0}\!\!\!\!\!\!\mathbb{E}m_{p}(X_t,...,X_t,\overline{X_t}, ..., \overline{X_t})&=\!\!\!\!\! \sum_{p_1 \in \mathcal{P}_{k+l} | p_1\leq p} \kappa_{p_1} (G_{k,l})\mathbb{E}m_{\!\!\text{ }^{t}p_1 \circ p}(X_{t_0},...,X_{t_0},\overline{X_{t_0}}, ..., \overline{X_{t_0}}), 
\end{align*}
 where we wrote $k$ times $X_{t_0}$ and $l$ times $\overline{X_{t_0}}$.

If $(G_{k,l}^{N})_{k,l}$ weakly condensates then the asymptotic $\mathcal{P}$-factorization property holds for $(X_t)_{t \geq 0} \cup (X_t^{*})_{t \geq 0}$ and the $\mathcal{P}$-moments of $(X^{N}_t)_{t \geq 0} \cup ((X^{N}_t)^{*})_{t \geq 0}$ converge in probability to the limit of their expectation.

\end{theorem}

\section{Some examples of L\'{e}vy processes and consequences}
Theorems \ref{convergencegenerale} and \ref{convergence*levy} assert that the convergence of the generator of a L\'{e}vy process implies the convergence, possibly in probability, in $\mathcal{P}$-distribution. In the Brownian case, one can compute the finite-dimensional cumulants of the generator: it implies the convergence in probability in $\mathcal{P}$-distribution of Brownian motions but also a matricial Wick formula. In the general L\'{e}vy case, the finite-dimensional cumulants of the generator are complicated to compute, but we can compute the $\mathcal{P}$-moments of the generator and their limits. This section illustrates also the fact that convergence of additive and multiplicative L\'{e}vy processes can be handle in the same setting: the convergence of multicative L\'{e}vy processes are actually implied by the convergence of their additive counterpart. 

\subsection{Brownian processes}
In this section, we apply Theorems \ref{convergencegenerale} and \ref{convergence*levy} to recover and extend in the setting of $\mathcal{P}$-distributions the known convergence of Hermitian and unitary Brownian motions. A matricial Wick formula is given as a consequence of Equation (\ref{equationverifie}) for the Hermitial Brownian motion. We also show that any central $\mathcal{P}$-Gaussian (Definition \ref{def:Pgaussian}) can be approximated by the $\mathcal{P}$-distribution of random matrices.  

\subsubsection{Convergence of Brownian motions}
\label{sec:convergencebrownien}
Following the presentation of L\'{e}vy in \cite{Levymaster}, we define some useful space of matrices. Let $\mathbb{K}$ be either $\mathbb{R}$ or $\mathbb{C}$, let $N$ be a positive integer. The spaces of skew-symmetric and symmetric real matrices of size $N$ are respectively $\mathfrak{a}_{N} = \{M \in \mathcal{M}_N(\mathbb{R}),\!\text{ }^{t}M = -M\} \text{ and } 
\mathfrak{s}_N = \{M \in \mathcal{M}_N(\mathbb{R}),\!\text{ }^{t}M = M\}.$
The space of skew-Hermitian matrices of size $N$ is $\mathfrak{u}_N = \{M \in \mathcal{M}_N(\mathbb{C}),M^{*} = -M\}.$ The space of Hermitian matrices of size $N$ is $i\mathfrak{u}_N$. We will use the conventions: 
$$
 \mathfrak{u}(N,\mathbb{K})  = \left\{
   \begin{array}{ll}
      \mathfrak{a}_N& \text{ if } \mathbb{K} = \mathbb{R}, \\
      \mathfrak{u}_N& \text{ if } \mathbb{K} = \mathbb{C}. 
    \end{array}
\right.
$$
and:  
$$
 \mathfrak{h}(N,\mathbb{K})  = \left\{
   \begin{array}{ll}
      \mathfrak{s}_N& \text{ if } \mathbb{K} = \mathbb{R} ,\\
      i\mathfrak{u}_N& \text{ if } \mathbb{K} = \mathbb{C}. 
    \end{array}
\right.
$$ 
More generally, we will consider: 
\begin{align*}
 \mathfrak{g}^{\epsilon} (N,\mathbb{K}) = \left\{
   \begin{array}{ll}
      \mathfrak{u}(N, \mathbb{K}) & \text{ if } \epsilon = -1, \\
      \mathfrak{h}(N, \mathbb{K})& \text{ if } \epsilon=1. 
    \end{array}
    \right.
\end{align*}
Besides, we will also use the following notation: 
\begin{align*}
\beta_{\mathbb{K}} = \left\{
   \begin{array}{ll}
      1& \text{ if } \mathbb{K} = \mathbb{R}, \\
      2& \text{ if } \mathbb{K} = \mathbb{C}. 
    \end{array}
\right.
\end{align*}
Let us consider $\epsilon \in \{-1,1\}$. On $\mathfrak{g}^{\epsilon}(N,\mathbb{K})$, we consider the  scalar product: 
\begin{align}
\label{eq:scalar}
<\!X,Y\!> = \frac{\beta_{\mathbb{K}}N}{2} Tr(X^{*}Y). 
\end{align}
Besides, we set: 
\begin{align*}
U(N,\mathbb{K}) = \left\{
   \begin{array}{ll}
      O(N) & \text{ if } \mathbb{K} = \mathbb{R}, \\
      U(N) & \text{ if } \mathbb{K} = \mathbb{C}. 
    \end{array}
    \right.
\end{align*}
The group $U(N,\mathbb{K})$ is a Lie group whose Lie algebra is $\mathfrak{u}(N,\mathbb{K})$. Let us recall a result of L\'{e}vy in \cite{Levymaster} which will be needed later. 
\begin{definition}
\label{casimircal}
Let $d$ be the dimension of $\mathfrak{g}^{\epsilon}(N,\mathbb{K})$. Let $(X_i)_{i=1}^{d}$ be an orthonormal basis of $\mathfrak{g}^{\epsilon}(N,\mathbb{K})$. The Casimir of $\mathfrak{g}^{\epsilon}(N,\mathbb{K})$ is equal to: 
\begin{align*}
C_{\mathfrak{g}^{\epsilon}(N,G)} = \sum_{i=1}^{d} X_i \otimes X_i = \frac{1}{N} \rho_{N} \big( \epsilon (1,2) + (2-\beta_{\mathbb{K}}) [1,2]\big), 
\end{align*}
where $(1,2) = \{ \{1,2'\},\{2,1'\}\}$ and $[1,2] = \{\{1,1'\},\{2,2'\}\}$. A consequence is that
\begin{align*} 
\sum_{i=1}^{d} X_i X_i = c_{\mathfrak{g}^{\epsilon}(N,\mathbb{K})} {\mathrm Id}_N
\end{align*} 
with $c_{\mathfrak{g}^{\epsilon}(N,\mathbb{K})} = \epsilon + \frac{2-\beta_{\mathbb{K}}}{N}.$
\end{definition}

As L\'{e}vy did in \cite{Levymaster}, we now define the notion of Brownian motion on $\mathfrak{g}^{\epsilon}(N,\mathbb{K})$.  Let $d$ be the dimension of $\mathfrak{g}^{\epsilon}(N,\mathbb{K})$, and let $(X_i)_{i=1}^{d}$ be an orthonormal basis of $\mathfrak{g}^{\epsilon}(N,\mathbb{K})$.

 \begin{definition}
\label{brownien}
Let $\left((B_t^{i})_{t \geq 0}\right)_{i=1}^{d}$ be a $d$-tuple of independent real Brownian motions. The process $(H_t = \sum_{i=1}^{d} B_t^{i} X_i)_{t \geq 0}$ is a Brownian motion on $\mathfrak{g}^{\epsilon}(N,\mathbb{K})$. Any process which has the same law as $(H_t)_{t \geq 0}$ is called a Brownian motion on $g^{\epsilon}(N,\mathbb{K})$. 
\end{definition} 

The law of the process $(H_t)_{t \geq 0}$ does not depend on the choice of the orthonormal basis of $\mathfrak{g}^{\epsilon}(N,\mathbb{K})$. This fact implies the $U(N,\mathbb{K})$-invariance property for $(H_t)_{t \geq 0}$. Any Brownian motion in $\mathfrak{g}^{\epsilon}(N,\mathbb{K})$ is an additive L\'{e}vy process. Besides, we can compute the bracket of $H_t$ with itself: 
\begin{align}
\label{bracket}
dH_t \otimes dH_t &= \sum_{i,j=1}^{d}dB_t^{i}dB_t^{j} X_i \otimes X_j  = \left(\sum_{i=1}^{d} X_i \otimes X_i \right)dt = C_{\mathfrak{g}^{\epsilon}(N,\mathbb{K})} dt, 
\end{align}
due to the fact that for any $i,j \in \{0,...,d\}$, $dB_t^{i}dB_t^{j} = \delta_{i,j} dt$. 

\begin{remark}
Recall the operation ${\sf S}_k$ in Definition \ref{def:Sk}. One can easily compute the bracket of $H_t$ with $^{t}H_t$ and $^{t}H_t$ with itself since $dH_{t} \otimes\ \!\! ^{t}dH_t = {\sf S}_{1}\left[dH_{t} \otimes dH_{t}\right]$ and $^{t}dH_{t} \otimes\ \!\! ^{t}dH_t = {\sf S}_{0}\left[dH_{t} \otimes dH_{t}\right]$. Thus, when $\mathbb{K} = \mathbb{C}$ and $\epsilon =-1$: 
\begin{align}
\label{prodavect} dH_{t} \otimes\ \!\! ^{t}dH_t = -\frac{1}{N} \rho_{N}\big([1,2]\big)dt \ \ \ \ \ \text{ and }\ \ \ \ \ ^{t}dH_{t} \otimes\ \!\! ^{t}dH_t = -\frac{1}{N}  \rho_{N}(1,2)dt.
\end{align}
\end{remark}

Using Itô's formula, one has the easy following lemma.

\begin{lemma}
\label{geneadd}
Let $(H_{t})_{t \geq 0}$ be a Brownian motion on $\mathfrak{g}^{\epsilon}(N,\mathbb{K})$. For any integer $k \neq 2$, $G_k =\frac{d}{dt}_{\mid t=0} \mathbb{E}\left[ H_t^{\otimes k} \right]= 0$ and: 
\begin{align*}
G_2 = \frac{d}{dt}_{\mid t=0} \mathbb{E}\left[ H_t^{\otimes 2} \right] = C_{\mathfrak{g}^{\epsilon}(N,\mathbb{K})}.  
\end{align*}
\end{lemma}

Let us now define the $U(N,\mathbb{K})$ Brownian Motion. 
\begin{definition}
\label{definitionbrowni}
Let $(H_t)_{t \geq 0}$ be a Brownian motion on $\mathfrak{u}(N,\mathbb{K})$. The solution of the Stratonovich stochastic equation: 
$$
 \left\{
    \begin{array}{ll}
      dU_t &= dH_t \circ U_t,  \\
      U_0 &= I_N,  
    \end{array}
\right.
$$ 
is a Brownian motion on $U(N,\mathbb{K})$. Any process which has the same law as $(U_t)_{t \geq 0}$ is a Brownian motion on $U(N,\mathbb{K})$. 
\end{definition}

\begin{remark}
We can reformulate this Stratonovich stochastic equation in terms of an Itô stochastic equation: 
$$
 \left\{
    \begin{array}{ll}
      dU_t &= dH_t U_t +\frac{c_{\mathfrak{u}(N,\mathbb{K})}}{2} U_t dt,  \\
      U_0 &= I_N.
    \end{array}
\right.
$$ 
\end{remark}

Any Brownian motion on $U(N,\mathbb{K})$ is $U(N,\mathbb{K})$-invariant: this is a consequence of the fact that the linear Brownian motion is also $U(N,\mathbb{K})$-invariant. Let us compute the bracket of $U_t$ with itself. Let $(H_t)_{t \geq 0}$ be a Brownian motion on $\mathfrak{u}(N,\mathbb{K})$ and $(U_t)_{t \geq 0}$ be the Brownian motion on $U(N,\mathbb{K})$ associated with $(H_t)_{t \geq 0}$. Using Equation (\ref{bracket}): 
\begin{align*}
dU_t \otimes dU_t =\! dH_t U_t\otimes dH_t U_t =\! (dH_t \otimes dH_t)(U_t \otimes U_t) =\!  C_{\mathfrak{u}(N,\mathbb{K})}(U_t \otimes U_t) dt. 
\end{align*}
Using this equality, the Itô's formula and Lemma \ref{casimircal}, we can compute the action of the infinitesimal generator on the tensor product, already given in \cite{Levymaster}. Recall the insertion operator $\mathcal{I}_I$ defined in Definition \ref{def:insertion}.
\begin{lemma}
\label{genemult} Let $(U_t)_{t \geq 0}$ a Brownian motion on $U(N,\mathbb{K})$. For any integer $k$:
\begin{align*}
G_k^{N} =\frac{d}{dt}_{\mid t=0} \mathbb{E}\left[U_{t}^{\otimes k}\right] = \frac{k c_{\mathfrak{u}_{N,\mathbb{K}}}}{2} Id^{\otimes k}+  \sum_{1\leq i<j\leq k} \mathcal{I}_{\{i,j\}}\left[C_{\mathfrak{u}(N,\mathbb{K})}, Id^{\otimes k-2}\right]. 
\end{align*}
\end{lemma}

 Let us suppose, just for this discussion, that $\mathbb{K}=\mathbb{C}$. The matrix $H_{t}$ is skew-Hermitian, $\overline{dH_t} = -\!\text{ }^{t}dH_{t}$, and $dU_t = dH_t U_t - \frac{1}{2} U_t dt$, thus $d\overline{U_t}= -\!\text{ }^{t}dH_{t}\overline{U_t} -  \frac{1}{2} \overline{U_t} dt.$ Using Equation (\ref{prodavect}), this implies that: 
\begin{align*}
dU_t \otimes d\overline{U_t} &= - dH_t U_t \otimes \ \!\!^{t}dH_{t} \overline{U_t} =\frac{1}{N}\rho_{N}\big([1,2]\big)\left(U_t \otimes \overline{U_t}\right)dt , \\
d\overline{U_t} \otimes d\overline{U_t} &= \ \!\!^{t}dH_{t} \overline{U_t}\otimes \ \!\!^{t}dH_{t} \overline{U_t} = -\frac{1}{N} \rho_{N}\big((1,2)\big) \left( \overline{U_t}\otimes  \overline{U_t} \right)dt. 
\end{align*}
Using the Itô's formula, we recover the following result already proved by T. L\'{e}vy in \cite{Levymaster}, and by A. Dahlqvist in \cite{Dahlqvist}. 

\begin{lemma}
\label{calculgenera*}
Let $(U_t)_{t \geq 0}$ be a Brownian motion on $U(N,\mathbb{C})$. For any positive integers $k$ and $l$: 
\begin{align*}
G_{k,l}^{N} &= \frac{d}{dt}_{\mid t = 0}\mathbb{E}\left[U_{t}^{\otimes k} \otimes \overline{U_t}^{\otimes l}\right]\\&= -\frac{k+l}{2} {\mathrm{Id}_N}^{\otimes {k+l}} - \frac{1}{N}\rho_{N}\left[ \sum_{1\leq i<j\leq k+l}(i,j) + \sum_{1\leq i\leq k < j \leq k+l} \left([i,j]-(i,j)\right)\right],
\end{align*}
where $(i,j)$ and $[i,j]$ were defined in Section \ref{sec:basic}.
\end{lemma}

In order to state the main theorem of convergence for additive and multiplicative Brownian motions, we need the following notation. 

\begin{notation}
\label{not:formduale}
Let $k$ be an integer and $p$ be in $\mathcal{P}_k$. The form $p^{*} \in \left(\bigoplus_{k=0}^{\infty} \mathbb{C}[\mathcal{P}_k / \mathfrak{S}_k]\right)^{*}$ sends $[p]$, the equivalence class of $p$, on $1$ and the other equivalence classes on $0$.
\end{notation}

 Let $\epsilon$ be in $\{-1,1\}$. For any integer $N$, let $(H_t^{N})_{t \geq 0}$ be a Brownian motion on $\mathfrak{g}^{\epsilon}(N,\mathbb{K})$ and let $(U_t^{N})_{t \geq 0}$ be a Brownian motion on $U(N, \mathbb{K})$. Let $\mathcal{A}_{\mathbb{K}}$ be $\mathfrak{S}$ if $\mathbb{K}= \mathbb{C}$ or $\mathcal{B}$ if $\mathbb{K} = \mathbb{R}$.

\begin{theorem}
\label{convbrowniens1}
The process $(H_t^{N})_{t \geq 0}$ converges in $\mathcal{P}$-distribution toward a $\mathcal{G}(\mathcal{A}_{\mathbb{K}})$-invariant additive $\mathcal{A}_{\mathbb{K}}$-L\'{e}vy process. It satisfies the asymptotic $\mathcal{P}$-factorization property: the $\mathcal{P}$-moments of $(H^{N}_t)_{t \geq 0}$ converges in probability to the limit of their expectation. Moreover, for any real $t_0 \geq 0$, $\mathcal{R}\left(H_{t_0}\right) = e^{\boxplus t_0 \mathcal{R}(G)},$ where $\mathcal{R}(G) = \epsilon (1,2)^{*} + (2-\beta_{\mathbb{K}}) [1,2]^{*}.$

When $\epsilon=1$, the mean empirical eigenvalues distribution of $H_{1}^{N}$ converges in probability to the Wigner semicircular distribution: $\mu_{sc} = \frac{1}{2\pi} \sqrt{4-\mid\!x\!\mid^{2} } 1\!\!1_{[-2,2]} dx. $
\end{theorem}

Let us state the theorem for the convergence of $U(N,\mathbb{K})$ Brownian motions. This theorem extends the results of Biane in \cite{Biane} and L\'{e}vy in \cite{Levy2}.
\begin{theorem}
\label{convbrowniens2}
The process $\left(U_t^{N}\right)_{t \geq 0}$ converges in $\mathcal{P}$-distribution toward a $\mathcal{G}(\mathcal{A}_{\mathbb{K}})$-invariant multiplicative  $\mathcal{A}_{\mathbb{K}}$-L\'{e}vy process. Moreover, for any positive real $t _0$, $\mathcal{R}(U_{t_0}) = e^{ \boxtimes t_0 \mathcal{R}(G)},$ where:
\begin{align}
\label{lastequ}
\mathcal{R}(G)\! =\! -\frac{1}{2} \mathrm{id}_{1}^{*} + \sum_{k \geq 2} \left[ -\frac{k}{2} \mathrm{id}_k^{*} - ((1,2)\otimes {\mathrm{id}_{k-2}})^{*} \!+\! (2-\beta_{\mathbb{K}}) ([1,2]\otimes {\mathrm{id}_{k-2}})^{*}\right]\!\!.
\end{align}
Besides, for any $t_0\geq 0$, any positive integer $k$ and any $p\in \mathcal{P}_k$: 
\begin{align}
\label{eq:momentbrownienuni}
\frac{d}{dt}_{\mid t=t_0} \mathbb{E}m_{p}(U_t) = \sum_{p_1\in \mathcal{P}_k | p_1 \leq p} \mathcal{R}(G)(p_1) \mathbb{E}m_{^{t}p_1\circ p}(U_{t_0}). 
\end{align}

Moreover, the family $\left(U^{N}_{t}, \left(U^{N}_{t}\right)^{*}\right)_{t \geq 0}$ converges in $\mathcal{P}$-distribution and satisfies the asymptotic $\mathcal{P}$-factorization property. In particular, the family $\left(U_t^{N},\left(U_t^{N}\right)^{*}\right)_{t \geq 0}$ converges in probability in $\mathcal{P}$-distribution. For any $t \geq 0$, the empirical eigenvalues distribution of $U_{t}^{N}$ converges in probability to a measure $\nu_{t}$ as $N$ goes to infinity and for any integer~$k$: 
\begin{align*}
\int_{\mathbb{U}} z^{k} \nu_{t}(dz) = \int_{\mathbb{U}} z^{-k} \nu_{t}(dz) =  e^{-\frac{kt}{2}}\sum_{l=0}^{k-1} \frac{(-t)^{l}}{l!} k^{l-1} \binom{k}{l+1}. 
\end{align*}

Let us suppose that $\mathbb{K} = \mathbb{C}$. Let $k$ and $l$ be two integers and $t_0\geq 0$. Using the same notations as for Theorem \ref{autreconvergence*levy}, $\mathcal{R}_{k,l}(U_{t_0}) = e^{\boxtimes t_0 \mathcal{R}_{k,l}(G)}$, where 
$$\mathcal{R}_{k,l}(G) = -\frac{k+l}{2} \mathrm{id}_{k+l}^{*} - \sum_{1\leq i<j\leq k+l}(i,j)^{*} - \sum_{1\leq i\leq k < j \leq k+l} \left([i,j]^{*}-(i,j)^{*}\right),$$
and for any $p \in \mathcal{P}_{k+l}$: 
\begin{align*}
\frac{d}{dt}_{\mid t=t_0}\!\!\!\!\!\!\!\!\!\!\!\mathbb{E}m_{p}(U_t,...,U_t,\overline{U_t}, ..., \overline{U_t})&=\!\!\!\! \sum_{p_1 \in \mathcal{P}_{k+l} | p_1\leq p} \!\!\!\!\!\!\!\!\!\!\mathcal{R}_{k,l}(G)(p_1)\mathbb{E}m_{\!\!\text{ }^{t}p_1 \circ p}(U_{t_0},...,U_{t_0},\overline{U_{t_0}}, ..., \overline{U_{t_0}}), 
\end{align*}
 where we wrote $k$ times $U_{t_0}$ and $l$ times $\overline{U_{t_0}}$.
\end{theorem}

\begin{proof}[Proof of Theorem \ref{convbrowniens1}]
For any integers $k$, $N \geq 2k$, let $G^{N}_k =\frac{d}{dt}_{\mid t=0} \mathbb{E}\left[ (H^{N}_t)^{\otimes k} \right]$. By Lemma \ref{geneadd}, the value of $G^{N}_k$ is known and we can compute its finite-dimensional cumulants: $\kappa_{(1,2)}(G_2^{N}) = \epsilon$, $\kappa_{[1,2]}(G_2^{N}) = 2-\beta_{\mathbb{K}}$ and for any other $p$, $\kappa_{p}(G_{k}^{N}) = 0$. For any integer $k$, $(G^N_{k})_{N \in \mathbb{N}}$ converges and $\mathcal{R}(G) = \epsilon (1,2)^{*} + (2-\beta_{\mathbb{K}}) [1,2]^{*}.$ Since the two partitions $(1,2)$ and $[1,2]$ are irreducible, $(G^{N}_k)_{k}$ condensates. Let us recall also that $(H^{N}_{t})_{t \geq 0}$ is a $U(N,\mathbb{K})$-invariant L\'{e}vy process. We can apply Theorem \ref{convergencegenerale}: it only remains to prove that: 
\begin{enumerate}
\item $\mathcal{P}$-moments of $(H_t^{N})_{t \geq 0}$ converges in probability to the limit of their expectation,
\item the mean empirical eigenvalue distribution of $H_{1}^{N}$ converges in probability to the Wigner semicircular distribution when $\epsilon = -1$.
\end{enumerate} 

Since the family $(H_t^{N})_{t \geq 0} \cup (-H_t^{N})_{t \geq 0}$ is stable by the conjugate operation, the first assertion is a consequence of the $\mathcal{P}$-asymptotic factorization property of $(H_{t}^{N})_{t \geq 0}$ and Theorem \ref{theoremconvprob}.

For the second assertion, since we have proved the convergence in probability, using Theorem \ref{convergencemean}, it is enough to prove that for any integer $k$, $\mathbb{E}m_{(1,...,k)}(H_{t}) = \int_{[-2,2]} x^{k} \mu_{sc}(dx)$, where $(1,...,k)$ is the cycle of size $k$. Actually, one needs also to prove some uniformity subgaussianity property which can be proved using concentration of measure for the operator norm of a random matrix. We refer to the discussion at the beginning of Section 2.4.2 in \cite{Tao}. Let $k$ be an integer, the moment $\int_{[-2,2]} x^{k} \mu_{sc}(dx)$ is known to be equal to the number of non-crossing pair-partitions of $\{1,...,k\}$. Now, 
\begin{align*}
\mathbb{E}m_{(1,...,k)}(H_{1}) = \sum_{p \in \mathcal{P}_k | p \leq (1,...,k)} \mathcal{R}(H_{1})(p) = \sum_{p \in \mathcal{P}_k | p \leq (1,...,k)} e^{\boxplus \mathcal{R}(G)}(p).
\end{align*}
Yet, using Lemma $3.1$ of \cite{Gab1}, for any permutation $\sigma \in \mathfrak{S}_k$ and any $b \in \mathcal{B}_k \setminus \mathfrak{S}_k$, $b \nleq \sigma$. This shows that one can forget about the duals of Brauer elements in $\mathcal{R}(G)$: 
\begin{align*}
\mathbb{E}m_{(1,...,k)}(H_{1}) = \sum_{\sigma \in \mathfrak{S}_k | \leq (1,...,k)} e^{\boxplus (1,2)^{*}}(\sigma). 
\end{align*}
For any permutation $\sigma$, $e^{\boxplus (1,2)^{*}}(\sigma)$ is equal to $1$ if $\sigma$ is only composed of cycles of size $2$ and $0$ if not. Besides the set of involutions $\sigma$ without fixed points such that $\sigma \leq (1,...,k)$ is in bijection with the set ${\sf PNC}_k$ of non-crossing pair-partitions of $\{1,...,k\}$, thus $\mathbb{E}m_{(1,...,k)}(H_{1}) = \# {\sf PNC}_{k} =\int_{[-2,2]} x^{k} \mu_{sc}(dx)$. 
  \end{proof}

The proof of Theorem \ref{convbrowniens2} follows the same structure as the proof of Theorem \ref{convbrowniens1}.

\begin{proof}[Proof of Theorem \ref{convbrowniens2}]
For any integers $k$, $N \geq 2k$, let $G^{N}_k =\frac{d}{dt}_{\mid t=0} \mathbb{E}\left[ (U^{N}_t)^{\otimes k} \right]$. By Lemma \ref{genemult}, the value of $G^{N}_k$ is known and we can compute its finite-dimensional cumulants: $\kappa_{\mathrm{id}_k}(G_{k}^{N}) =  \frac{k c_{\mathfrak{u}_{N,\mathbb{K}}}}{2}$, $\kappa_{(1,2) \otimes {\mathrm{id}_k}}(G_{k+2}^{N}) = -1$, $\kappa_{[1,2] \otimes {\mathrm{id}_k}}(G_{k+2}^{N}) = 2-\beta_{\mathbb{K}}$ and for any $p$ which is not in the equivalence class of these partitions for the action of $\mathfrak{S}$ on $\mathcal{P}$, $\kappa_{p}(G) = 0$. For any integer $k$, $(G_k^{N})_{N \in \mathbb{N}}$ converges and $\mathcal{R}(G)$ is given by Equation~(\ref{lastequ}). Let us recall also that $(U_t^{N})_{t \geq 0}$ is a $U(N,\mathbb{K})$-invariant L\'{e}vy process. We can apply Theorem \ref{convergencegenerale} and Remark \ref{rem:conver} in order to get the assertions up to the one about the family $(U^N_t)_{t \geq 0} \cup ((U_t^{N})^{*})_{t \geq 0}$. 

Now, if $\mathbb{K} = \mathbb{R}$, the general assertion about the family $(U^N_t)_{t \geq 0} \cup ((U_t^{N})^{*})_{t \geq 0}$ is a consequence of the one about $(U^N_t)_{t \geq 0}$, the fact that $(G_k^{N})_{k \in \mathbb{N}}$ weakly condensates and the usual arguments. Indeed, convergence in probability in $\mathcal{P}$-distribution of $(U^N_t)_{t \geq 0}$ implies the convergence in  probability in $\mathcal{P}$-distribution of $(U^N_t)_{t \geq 0}\cup ((U_t^{N})^{*})_{t \geq 0}$. 

Let us suppose for this paragraph that $\mathbb{K} = \mathbb{C}$. The assertions about the family $(U^N_t)_{t \geq 0} \cup ((U_t^{N})^{*})_{t \geq 0}$ is a direct consequence of Lemma \ref{calculgenera*}, Theorems \ref{convergence*levy} and \ref{theoremconvprob}. 

It remains to prove the assertion about the convergence of the eigenvalue distribution of $U_t^{N}$. Let $t \geq 0$, since $U_t^{N}$ satisfies the asymptotic $\mathcal{P}$-factorization, for any permutation $\sigma \in \mathfrak{S}_{k}$, $\mathbb{E}m_{\sigma}(U_t) = \prod_{c \in \sigma\vee \mathrm{id}_k} \mathbb{E}m_{(1,...,\frac{\#c}{2})}(U_t)$. Let us denote $m_{k}(t) = \mathbb{E}m_{(1,...,k)}(U_t)$. Let us consider Equation (\ref{eq:momentbrownienuni}). Using Lemma $3.1$ of \cite{Gab1}, for any permutation $\sigma \in \mathfrak{S}_k$ and any $b \in \mathcal{B}_k\setminus \mathfrak{S}_k$, $b \nleq \sigma$. This shows that one can forget about the duals of Brauer elements in $\mathcal{R}(G)$. One obtains that  for any integer $k$ and any $t_0 \geq 0$,
\begin{align*}
\frac{d}{dt}_{| t = t_0} m_k(t) = -\frac{k}{2}m_{ k}(t_0) -\frac{k}{2} \sum_{l=1}^{k-1} m_{l}(t_0)m_{k-l}(t_0). 
\end{align*} 
In \cite{Levymaster}, L\'{e}vy proved in Lemma $2.4$ that the moments of $\nu_t$ also satisfy the same system of linear differential equations with the same initial conditions. By unicity, for any integer $k$ and any $t \geq 0$, $\mathbb{E}m_{(1,...,k)}(U_{t}) = \int_{\mathbb{U}}z^{k} \nu_{t}(dz)$. Since we have proved the convergence in probability of the $\mathcal{P}$-distribution of $U_t^{N}$, Theorem \ref{convergencemean} allows us to conclude. 
  \end{proof}

\subsubsection{Matricial Wick formula}

A matricial Wick formula is given in this section. Let $N$ be an integer. A standard Gaussian $\mathcal{M}_{N}(\mathbb{K})$ matrix $M$ is a random matrix in $\mathcal{M}_{N}(\mathbb{K})$ which has the same law as $\sqrt{N} H_1$ where $\left(H_t\right)_{t \geq 0}$ is a Brownian motion on $\mathfrak{g}^{\epsilon}(N,\mathbb{K})$. When $\epsilon = 1$, this defines the G.O.E and G.U.E ensembles. The standard Gaussian matrices satisfy the matricial Wick formula. In order to state it, we introduce the notion of one- and two-specie pairings of an integer. Let $n$ be a positive integer.

\begin{definition}
A {\em one-specie pairing} $\pi$ of $n$ is a partition of $\{1,...,n\}$ into pairs. 
A {\em two-species pairing} $\pi$ of $n$ is a partition of $\{1,...,n\}$ into pairs, with a partition in two sets  $(T_\pi,W_\pi)$ of these pairs. For $i = 1$ or $2$, we denote by $\mathcal{F}_n(i)$ the set of $i$-specie(s) pairings of $n$. 
\end{definition}
Any one-specie pairing $\pi$ of $n$ will be seen as a two-species pairing with $W_\pi = \emptyset$. Recall the notions of transpositions and Weyl contractions defined in Section \ref{sec:basic}.

\begin{definition}
Let $\pi$ be a two-species pairing of $n$. The Brauer element $b_{\pi} \in \mathcal{B}_n$ is equal to: 
\begin{align*}
b_\pi = \prod_{\{i,j\} \in T_{\pi}} (i,j) \prod_{\{i,j\} \in W_{\pi}} [i,j]. 
\end{align*}
\end{definition}
The elements in the product defining $b_\pi$ commute, thus the order is not important. We can state the Wick matricial formula. 

\begin{theorem}
\label{Wickmat}
Let $k_1$ and $k_2$ be two integers, $A$ be a $k_1 \times k_2$ complex matrix, and $(\tilde{M}_1,..., \tilde{M}_{k_2})$ be a $k_2$-tuple of independent standard Gaussian $\mathfrak{g}^{\epsilon}(N,\mathbb{K})$ matrices. For any $i, j \in \{1,...,k_1\}$, we define $M_i = \sum_{l=1}^{k_2}A_{i,l}\tilde{M}_l$ and $C(M_i,M_j) = (A\!\text{ }^{t}\!A)_{i,j}$. Then: 
\begin{align*}
\mathbb{E}[M_{1}\otimes ...\otimes M_{k_1}] = \sum_{{\pi} \in \mathcal{F}_{k}(3-\beta_{\mathbb{K}}) } \left(\epsilon^{\#T_{\pi}} \prod_{\{i,j\} \in {\pi}} C(M_{i},M_{j}) \right) \rho_N(b_{{\pi}}), 
\end{align*}
\end{theorem}

\begin{proof}
Let $M$ be a standard Gaussian of size $N$ in $\mathfrak{g}^{\epsilon}(N,\mathbb{K})$. It has the same law than $\sqrt{N} H_1$ where $\left(H_t\right)_{t \geq 0}$ is a Brownian motion on $\mathfrak{g}^{\epsilon}(N,K)$. Using Lemma \ref{geneadd}, for any positive integer~$k$: 
\begin{align*}
G_k=\frac{d}{dt}_{\mid t=0} \mathbb{E}[H_t^{\otimes k}] = \delta_{k=2} \left[\frac{1}{N} \rho_N\left( \epsilon (1,2) + (2-\beta_{\mathbb{K}})[1,2] \right)\right].
\end{align*}
Thus, using Lemma \ref{geneentemps}, for any positive real $t_0$: 
\begin{align*}
\frac{d}{dt}_{\mid t=t_0} \mathbb{E}\left[H_t^{\otimes k}\right] = \sum_{I \subset \{1,...,k\}, \#I = k-2} \mathcal{I}_I\left[\mathbb{E}\left[H_{t_0}^{\otimes (k-2)}\right], G_{2} \right].
\end{align*}
This system of equations can be solved and, for any $t \geq 0$ and any integer $k$: 
\begin{align*}
\mathbb{E}[H_t^{\otimes k}]= N^{-\frac{k}{2}} \rho_N\left[ \sum_{{\pi} \in \mathcal{F}_k(3-\beta_{\mathbb{K}})} \epsilon^{\#T_{\pi}}b_{\pi}\right]. 
\end{align*}
This leads us to the conclusion: $\mathbb{E}\left[M^{\otimes k}\right] = N^{\frac{k}{2}}\mathbb{E}[H_t^{\otimes k}] =\sum_{{\pi} \in \mathcal{F}_k(3-\beta_{\mathbb{K}})} \epsilon^{\#T_{\pi}} \rho_{N}(b_{\pi}).$

Now, let $(M_1,...,M_{k_1})$ be given as in Theorem \ref{Wickmat}. An easy computation gives: 
\begin{align*}
\mathbb{E}\left[M_1\otimes...\otimes M_{k_1}\right] &= \sum_{i_1,...,i_{k_1}=1}^{k_2} A_{1}^{i_1}...A_{k_1}^{i_{k_1}}\mathbb{E}[\tilde{M}_{i_1}\otimes...\otimes \tilde{M}_{i_{k_1}}]\\
&=\sum_{p \in {\sf P}_{k_1}} \sum_{{\sf Ker}((i_1,...,i_{k_1})) = p} A_{1}^{i_1}...A_{{k_1}}^{i_{k_1}} \mathbb{E}[\tilde{M}_{i_1}\otimes...\otimes \tilde{M}_{i_{{k_1}}}], 
\end{align*}
where we recall that ${\sf P}_{{k_1}}$ is the set of partitions of $\{1,...,{k_1}\}$. Using the independence $(\tilde{M}_{i})_{i=1}^{k_2}$ and using the result already proved for one matrix: 
\begin{align*}
\mathbb{E}&\left[M_1\otimes...\otimes M_{k_1}\right]
&=\sum_{p \in {\sf P}_{k_1}} \left(\sum_{{\sf Ker}((i_1,...,i_{k_1})) = p} A_{1}^{i_1}...A_{{k_1}}^{i_{k_1}}\right) \sum_{\pi \in \mathcal{F}_{{k_1}}(3-\beta_{\mathbb{K}}) | \pi \trianglelefteq p} \epsilon^{\#T_{\pi}} b_{\pi},
\end{align*}
where ${\pi} \trianglelefteq p$ means that ${\pi}$ is finer than $p$. We can go on the calculations: 
\begin{align*}
\mathbb{E}\left[M_1\otimes...\otimes M_{k_1}\right] 
=\sum_{{\pi} \in \mathcal{F}_{k_1}(3-\beta_{\mathbb{K}})}\left(\sum_{p \in {\sf P}_{k_1}, {\pi} \trianglelefteq p}\ \ \sum_{{\sf Ker}((i_1,...,i_{k_1})) = p} A_{1}^{i_1}... A_{{k_1}}^{i_{k_1}} \right) \epsilon^{\#T_{\pi}} b_{\pi}, 
\end{align*}
and the result follows from the easy equality which holds for any $m\in \mathcal{F}_{k_1}(3-\beta_{\mathbb{K}})$: 
\begin{align*}
\sum_{p \in {\sf P}_{k_1}, {\pi} \trianglelefteq p}\ \  \sum_{{\sf Ker}((i_1,...,i_{k_1}))=p} A_{1}^{i_1}...A_{{k_1}}^{i_{k_1}} = \prod_{\{i,j\} \in {\pi}} \sum_{l=1}^{{k_2}} A_{i}^{l} A_{j}^{l} = \prod_{\{i,j\} \in {\pi}} C(M_i,M_j).
\end{align*}
Hence the matricial Wick formula given in Theorem \ref{Wickmat}. 
  \end{proof}

\begin{remark}
For any integers $l<k$, there exists also a matricial Wick formula for $\mathbb{E}\left[M_1\otimes...\otimes M_{l} \otimes M_l \otimes \overline{M_{l+1}} \otimes ...\otimes \overline{M_{k}}\right]$. This Wick formula is a consequence of Theorem \ref{Wickmat} and the fact that $M^{*}$ is either equal to $M$ or $-M$. 
For example, let $M$ be a standard Gaussian  $\mathfrak{g}^{1}(N,\mathbb{C})$ matrix. Recall the operation ${\sf S}_k$ defined in Definition \ref{def:Sk}. Using Theorem \ref{Wickmat}:
\begin{align*}
\mathbb{E}\left[M^{\otimes k} \otimes \overline{M}^{\otimes k}\right] = {\sf S}_k\left(\mathbb{E}[M^{\otimes k}\otimes (M^{*})^{\otimes k}]\right) =  \sum_{m \in \mathcal{F}_{k}(1) } \epsilon^{\#T_m} {\sf S}_{k}\left(b_{m} \right).
\end{align*}
\end{remark}

\subsubsection{Matricial approximations of centered $\mathcal{P}$-Gaussians}
In this section, based on the intuition we developed in the last sections, we prove that any centered $\mathcal{P}$-Gaussian can be approximated in $\mathcal{P}$-distribution by a sequence of random matrices. 

\begin{theorem}
\label{th:approximationgaussienne}
For any $\phi \in (\mathbb{C}[\mathcal{P}_2/\mathfrak{S}_2])^{*}$, there exists a sequence of matrices $(M_{N})_{N \in \mathbb{N}}$ such that $M_N$ converges in $\mathcal{P}$-distribution toward a $\mathcal{P}$-Gaussian whose $\mathcal{R}$-transform is given by $e^{\boxplus \phi}$.
\end{theorem}

\begin{proof}
Let $\phi_1$ and $\phi_2$ be two elements of $(\mathbb{C}[\mathcal{P}_2/\mathfrak{S}_2])^{*}$. Let us suppose that $(M_N)_{ N \in \mathbb{N}}$ and $(L_N)_{ N \in \mathbb{N}}$ converge in $\mathcal{P}$-distribution toward a $\mathcal{P}$-Gaussian, with $\mathcal{R}(M) = e^{\boxplus \phi_1}$ and $\mathcal{R}(L) = e^{\boxplus \phi_2}$. We can always suppose that for any integer $N$, $M_N$ and $L_N$ are independent and $\mathfrak{S}$-invariant. As a consequence of Theorems \ref{Lemain} and \ref{sommeetproduit}, the matrix $M_N + L_N$ converges also to a $\mathcal{P}$-Gaussian and $\mathcal{R}(M+L) = e^{\boxplus (\phi_1+\phi_2)}$. Besides, the matrix $iM_{N}$ converges also to a $\mathcal{P}$-Gaussian and $\mathcal{R}(M) = e^{\boxplus (-\phi_1)}$. 

Recall Notation \ref{not:formduale} where we defined the form $p^{*}$. It remains to prove the theorem when $\phi \in \{ p^{*} | [p] \in \mathcal{P}_2/ \mathfrak{S}_2\}$. For any integer $N$, we define a process $(H^{N}_t)_{t \geq 0}$ when $p$ is equal to: 
\begin{itemize}
\item ${\mathrm{id}_2} = \{\{1,1'\},\{2,2'\}\}$: let $(B_t)_{t \geq 0}$ be a real Brownian motion. For any integer $N$ and any $t \geq 0$, $H^N_t = B_t {\sf Id}_N$.
\item ${\sf 0}_2 = \{\{1,1'\},\{2,2'\}\}$: Let $((B^{(i)}_{t})_{t \geq 0})_{i \in \mathbb{N}}$ be a sequence of independent real Brownian motions. For any integer $N$ and any $t \geq 0$, $H^N_t = {\sf Diag}(B_{t}^{(1)}, ...,B^{(N)}_t)$.
\item  ${\sf 1}_{2} = \{ \{1\},\{1'\},\{2\},\{2'\}\}$: let $(B_t)_{t \geq 0}$ be a real Brownian motion. Let $\mathbb{J}_N$ be the matrix of size $N$ such that $(\mathbb{J}_{N})_{i}^{j} = 1/N$ for any $i$ and $j$ in $\{1,...,N\}$.  For any integer $N$ and any $t \geq 0$, $H^N_t= B_t \mathbb{J}_N$. 
\item $[1,2] = \{\{1,2\},\{1',2'\}\}$: let $((B^{(i,j)}_{t})_{t \geq 0})_{i,j \in \mathbb{N}}$ be independent real Brownian motions. For any integer $N$ and any $t \geq 0$, $H^N_t = (\frac{1}{\sqrt{N}}B_{t}^{(i,j)})_{i,j =1}^{N}$. 
\item $(1,2) = \{\{1,2'\},\{1',2\}\}$: for any integer $N$, we consider $(H_t^{N})_{t \geq 0}$ the Hermitian Brownian motion.
\item $\{\{1'\},\{2'\}, \{1,2\}\}$: let $((B_t^{(j)})_{t \geq 0})_{j \in \mathbb{N}}$ be a sequence of independent real Brownian motion. For any integer $N$, any $t \geq 0$ and any $i,j \in \{1,...,N\}$, $(H^N_t)_{i,j} = \frac{1}{N}B_{t}^{(j)}$.
\item $\{\{1\},\{2\}, \{1',2'\}\}$: let $((B_t^{(i)})_{t \geq 0})_{i \in \mathbb{N}}$ be a sequence of independent real Brownian motions. For any integer $N$, any $t \geq 0$ and any $i,j \in \{1,...,N\}$, $(H^{N}_t)_{i,j} = \frac{1}{N}B_{t}^{(i)}$.
\item $\{ \{1,2'\}, \{2\}, \{1'\}\}$: let $((B^{(i)}_{1,t})_{t \geq 0})_{i \in \mathbb{N}}$, $((B^{(i)}_{2,t})_{t \geq 0})_{i \in \mathbb{N}}$, $((B^{(i)}_{3,t})_{t \geq 0})_{i \in \mathbb{N}}$ be three independent sequences of  independent real Brownian motions. For any integer $N$, any $t \geq 0$ and any $i_0,j_0 \in \{1,...,N\}$, $(H^N_t)_{i_0,j_0} = \frac{1}{N}\left[B_{1,t}^{(i_0)} + B_{1,t}^{(j_0)} - i B_{2,t}^{(i_0)} - i B_{3,t}^{(j_0)}\right]$. The matrix $H^N_t$ can also be written as: 
\begin{align*}
H^N_t =\frac{1}{N} \left[\sum_{i_0,j_0=1}^{N} B_{1,t}^{(i_0)}(E_{i_0}^{j_0} + E_{j_0}^{i_0}) - i  \sum_{i_0,j_0=1}^{N} B_{2,t}^{(i_0)} E_{i_0}^{j_0} - i  \sum_{i_0,j_0=1}^{N} B_{3,t}^{(j_0)} E_{i_0}^{j_0} \right].
\end{align*}
\item $\{\{1,2,2'\},\{1'\}\}$: let $((B^{(i)}_{1,t})_{t \geq 0})_{i \in \mathbb{N}}$, $((B^{(i)}_{2,t})_{t \geq 0})_{i \in \mathbb{N}}$ be two independent sequences of independent real Brownian motions. For any integer $N$, any $t \geq 0$ and any $i_0,j_0 \in \{1,...,N\}$, $(H^N_t)_{i_0,j_0} =  \delta_{i_0 = j_0} (B_{1,t}^{(i_0)} - i B_{2,t}^{(i_0)}) + \frac{1}{N}B_{1,t}^{(j_0)} $. The matrix $H^N_t$ can also be written as: 
\begin{align*}
H^N_t = \sum_{i_0,j_0 = 1}^{N} \frac{B_{1,t}^{(j_0)}}{N} (E_{j_0}^{j_0} + E_{i_0}^{j_0}) - i \sum_{i_0=1}^{N} B_{2,t}^{(i_0)} E_{i_0}^{i_0}.
\end{align*}

\item $\{\{1,2,1'\},\{2\}\}$: we consider the transpose of the matrices used for the partition $\{\{1,2,2'\},\{1'\}\}$. 
\item $\{\{1,1'\},\{2\},\{2'\}\}$:  let $((B_{i}^{(1)})_{t \geq 0})_{i=1}^{3}$ be three independent real Brownian motions. For any integer $N$, any $t \geq 0$ and any $i_0,j_0 \in \{1,...,N\}$, we consider $(H^N_t)_{i_0,j_0}$ equal to $\delta_{i_0=j_0} (B_{t}^{(1)} - i B_{t}^{(2)}) + \frac{1}{N} (B_{t}^{(1)} - i B_{t}^{(3)})$. The matrix $H^N_t$ can also be written as: 
\begin{align*}
H^N_t= B_{t}^{(1)} (\mathrm{Id}_{N} + \mathbb{J}_N) - i B_{t}^{(2)} \mathrm{Id}_N - iB_t^{(3)} \mathbb{J}_{N}. 
\end{align*}
\end{itemize}
For any integer $N$, the process $(H_t^{N})_{t \geq 0}$ is an additive L\'{e}vy process and its generator $(G^{N}_k)_{k \in \mathbb{N}}$ can be computed: it converges and in each case $\mathcal{R}[G] = p^{*}$ where $p$ is the partition considered. Thus if, for any integer $N$, we set $M_N = H_1^{N}$, we have a sequence of random matrices which converges in $\mathcal{P}$-distribution toward a centered $\mathcal{P}$-Gaussian whose $\mathcal{R}$-transform is given by $e^{\boxplus p^{*}}$.

Let us illustrate the convergence of $(G_{k}^{N})_{k \in \mathbb{N}}$ by an example: let us consider the case $p = \{\{1,2,2'\}, \{1'\} \}$. For any $k \neq 2$, $(dH_t^{N})^{\otimes k} = 0$: this implies that $G_k^{N} = 0$. Let us consider $k=2$, then: 
\begin{align*}
G_{2}^{N} = \frac{1}{dt}dH_t^{N} \otimes dH_t^{N} &= \frac{1}{N^{2}} \sum_{i_0,j_0, i_1=1}^{N}(E_{j_0}^{j_0} + E_{i_0}^{j_0}) \otimes (E_{j_0}^{j_0} + E_{i_1}^{j_0}) - \sum_{i_0=1}^{N} E_{i_0}^{i_0} \otimes E_{i_0}^{i_0}\\
&=\frac{1}{N^{2}}\left[ N^{2} {\sf 0}_2 + N p + N p' + p_2 \right] - {\sf 0}_2 = \frac{1}{N} (p+p') + \frac{1}{N^{2}} p_2,\\
\end{align*}
where $p' = \{\{1,2,2'\}, \{1'\}\}$ and $p_2 = \{\{1,2\}, \{1'\}, \{2'\}\}$. The partition $p'$ is in the equivalence class of $p$ in $\mathcal{P}_{2}/\mathfrak{S}_2$. Thus $\kappa_{p}(G_{2}^{N}) = 1$ and for any partition $\tilde{p}$ which are not in the equivalence class of $p$ in $\mathcal{P}_{2}/\mathfrak{S}_2$, $\lim_{N \to \infty}\kappa_{p}(G_{2}^{N})=0$. Thus for any $k$, $(G_k^{N})_{N \in \mathbb{N}}$ converges and $\mathcal{R}(G) = p^{*}$.
  \end{proof}

\subsection{L\'{e}vy processes and approximation of free $\boxdot$-infinitely divisible probability measures}
\label{sec:infinitelydivisible}
\subsubsection{Generalities about free $\boxdot$-infinitely divisible probability measures }

In the articles \cite{Flor}, \cite{Caba} and \cite{Cebron2}, it is shown that there exists a matricial approximation for any semi-group of free $\boxdot$-infinitely divisible measures either by Hermitian or unitary L\'{e}vy processes. The proofs are different in the additive an multiplicative cases, using either Fourier transform or Weingarten calculus, and the convergence in probability of the approximations used concentration of measures. 

In the next section, we will see that these results are consequences of Theorems \ref{convergencegenerale}  and \ref{convergence*levy}. Besides, the convergence in probability of the approximations can be dealt using only the elementary notion of $\mathcal{P}$-factorization. At last, the arguments generalize easily in order to show that approximations by symmetric and orthogonal L\'{e}vy processes also exist. 

In order to explain the results of  \cite{Flor}, \cite{Caba} and \cite{Cebron2}, let us first introduce free $\boxdot$-infinitely divisible measures. If $\mu$ is a measure on a compact set of $\mathbb{C}$, let $\mathcal{M}(\mu)$ be the character of $\bigoplus_{k}\mathbb{C}[\mathfrak{S}_k]$ which, for any integer $k$, sends any cycle of size $k$ on $\int_{\mathbb{C}}z^{k} \mu(dz)$. The $\mathcal{R}$-transform of $\mu$, denoted by $\mathcal{R}(\mu)$, is equal to $\mathcal{R}_{\mathfrak{S}}(\mathcal{M}(\mu))$ where $\mathcal{R}_{\mathfrak{S}}$ is defined in Definition 4.13 of \cite{Gab1}. Actually, from now on, we will see $\mathcal{R}(\mu)$ as an element of $(\bigoplus_{k} \mathbb{C}[\mathcal{P}_k/\mathfrak{S}_k])^{*}$. In the following, $\mathcal{M}_{\boxdot}$ denotes either the set of measures supported by $\mathbb{R}$ when $\boxdot = \boxplus$, or by the unit circle $\mathbb{U}$ when  $\boxdot = \boxtimes$. For any probability measures $\mu$ and $\nu$ in $\mathcal{M}_{\boxdot}$, there exists a measure in $\mathcal{M}_{\boxdot}$, denoted by $\mu \boxdot \nu$ such that $\mathcal{R}(\mu \boxdot \nu) = \mathcal{R}(\mu) \boxdot \mathcal{R}(\nu)$: it is the free $\boxdot$-convolution of $\mu$ with $\nu$. The free $\boxdot$-convolution is a bilinear continuous operation (\cite{Voicu1}, \cite{voicu2}).

\begin{remark}
Let us suppose that $\mu$ and $\nu$ are compactly supported probability measures. Let $(X_i)_{i \in \mathbb{N}}$ (resp. $(Y_i)_{i \in \mathbb{N}}$) be a sequence of independent random variables of law $\mu$ (respectively of law $\nu$), for any integer $N$ let $M_N = {\sf Diag}(X_1,...,X_n)$ (respectively $L_N = {\sf Diag}(Y_1,...,Y_n)$) and $U_N$ be a Haar random matrix on $U(N)$. Let us suppose that for any integer $N$, $M_N$, $L_N$ and $U_N$ are independent. Using Theorems \ref{Lemain} and \ref{convergencemean}, if $\mu, \nu \in \mathcal{M}_{\boxplus}$ (resp. $\mu, \nu \in \mathcal{M}_{\boxtimes}$), the eigenvalues distribution of $M_N + U_N L_N U_N^{*}$  (resp. $M_N U_N L_N U_N^{*}$) converges to $\mu\boxplus \nu$ (resp. $\mu \boxtimes \nu$). This can be taken as a definition for the convolution $\boxdot$ for compactly supported probability measures. 
\end{remark}

A probability measure $\mu \in \mathcal{M}_{\boxdot}$ is a free $\boxdot$-infinitely divisible probability measure if for any integer $n \geq 1$ there exists a probability measure $\mu_{\frac{1}{n}} \in \mathcal{M}_{\boxdot}$ such that $\mu = \mu_{\frac{1}{n}}^{\boxdot n}.$ Let $\lambda_{\mathbb{U}}$ be the uniform probability measure on $\mathbb{U}$. Let us denote by $e_{\boxdot}$ the number $0$ if $\boxdot=\boxplus$ and $1$ if $\boxdot = \boxtimes$. Let $\mu$ be a probability measure in $\mathcal{M}_{\boxdot} \setminus \{ \lambda_{\mathbb{U}}\}$ be a free $\boxdot$-infinitely divisible probability measure. There exists a continuous one parameter semi-group of probability measures $(\mu_t)_{t \geq 0}$ in $\mathcal{M}_{\boxdot}$ for the $\boxdot$-convolution such that $\mu = \mu_1$ and $\mu_0 = \delta_{e_{\boxdot}}$. This semi-group is unique in the additive case and in the multiplicative case there is a canonical way to construct it. By definition, $\left(\mathcal{R}(\mu_t)\right)_{t \geq 0}$ is a continuous one parameter semi-group in $\left( (\bigoplus_{k}\mathbb{C}[\mathcal{P}_k/\mathfrak{S}_k])^{*}, \boxdot \right)$. There exists $\mathcal{LR}_{\boxdot}(\mu)$ such that for any $t_0 \geq 0$: 
\begin{align}
\label{semiadd}\frac{d}{dt}_{\mid t = t_0} \mathcal{R}(\mu_{t}) &= \mathcal{LR}(\mu) \boxdot \mathcal{R}(\mu_{t_0}).
\end{align}
Thus, the probability measure $\mu$ is characterized by $\mathcal{LR}(\mu)$ which is a $\boxdot$-infinitesi\-mal character on $\mathbb{C}[\mathfrak{S}_k/\mathfrak{S}_k]$ (here $\mathfrak{S}_k/\mathfrak{S}_k$ is the notation for the equivalence classes of the action of $\mathfrak{S}_k$ by conjugation on itself). 
\begin{remark}\label{rq:LretR}
In the additive case, $\mathcal{LR}(\mu)$ is the unique $\boxplus$-infinitesimal character such that for any irreducible permutation $\sigma$, $\mathcal{LR}(\mu)(\sigma) = \mathcal{R}(\mu)(\sigma)$.
\end{remark}

Let us recall that a measure $\rho$ in $\mathcal{M}_{\boxdot}$ is a L\'{e}vy measure if $\rho(\{e_{\boxdot}\}) = 0$ and $\int \min(|z-e_{\boxdot}|^{2},1) \rho(dx) <\infty. $ In \cite{Gab1}, Lemma 4.3, it was explained that the set of characters of $\bigoplus_{k}\mathbb{C}[\mathfrak{S}_k]$ is isomorphic to the affine space $\mathbb{C}_{1}[[z]]$ of formal series with constant coefficient equal to $1$.  Let $S$ be an element of $\mathbb{C}_{1}[[z]]$ and let us suppose that for a given complex number $z$ the evaluation of this formal series converges then we denote this evaluation by $S(z)$. The linear form $\mathcal{LR}(\mu)$ was actually described in \cite{BarnNiel}, \cite{BV} and \cite{Cebron2}.
\begin{theorem}
\label{caractaddfreeinf}
Let $\mu$ be a free $\boxplus$-infinitely divisible probability measure. For any complex number $z$ such that $Im(z) <0$, $\mathcal{R}(\mu) (z)$ is defined, and there exist $\eta \in \mathbb{R}$, $a \in \mathbb{R}^{+}$ and $\rho$ a L\'{e}vy measure on $\mathbb{R}$ which are unique, such that: 
\begin{align}
\label{Rtransfoinfini}
\mathcal{R}(\mu) (z)= 1+ \eta z + a z^{2} + \int_{\mathbb{R}} \left(\frac{1}{1-tz} - 1 - tz 1\!\!1_{[-1,1]}(t) \right) \rho (dt). 
\end{align}
Conversely, for any $\eta \in \mathbb{R}$, $a \in \mathbb{R}^{+}$ and any L\'{e}vy measure $\rho$ on $\mathbb{R}$, there exists a free $\boxplus$-infinitely divisible probability measure $\mu \in \mathcal{M}_{\boxplus}$ such that for any complex number $z$ such that $Im(z) <0$,  Equality (\ref{Rtransfoinfini}) is satisfied. The  triplet $(\eta, a, \rho)$ is the $\boxplus$-characteristic triplet of $\mu$.
\end{theorem}

In particular, using Remark \ref{rq:LretR}, if the L\'{e}vy measure of $\mu$ is compactly supported then there exist $\eta \in \mathbb{R}$, $a \in \mathbb{R}^{+}$ and $\rho$ a L\'{e}vy measure on $\mathbb{R}$ which are unique, such that $\mathcal{LR}(\mu)$ is characterized by the fact that for any integer $k$: 

 $$\mathcal{LR}(\mu)((1,...,k)) =  \left\{
   \begin{array}{ll}
      \eta& \text{ if } k=1, \\
      a+ \int_{\mathbb{R}} x^{2} \rho(dx)& \text{ if } k=2, \\
      \int_{\mathbb{R}} x^{n} \rho(dx)& \text{ if } k \geq 3. 
    \end{array}
\right.$$
 Let us state the characterization of  free $\boxtimes$-infinitely divisible probability measures.

\begin{theorem}
\label{caractmultfreeinf}
Let $\eta$ be a free $\boxtimes$-infinitely divisible probability measure. There exist $\omega \in \mathbb{U}$, $b \in \mathbb{R}^{+}$ and $\nu$ a L\'{e}vy measure on $\mathbb{U}$ which are unique and such that for any integer $k$:
\begin{align}
\label{LRtransfoinfini}
 \mathcal{LR}(\mu)((1,...,k)) =  \left\{
   \begin{array}{ll}
       i {\sf arg}(\omega)- \frac{b}{2} + \int_{\mathbb{U}} (\Re(\zeta)-1) \nu(d\zeta),& \text{ if } k=1, \\
      -b + \int_{\mathbb{U}} (\zeta-1)^{2} \nu (d\zeta), & \text{ if } k=2, \\
       \int_{U}(\zeta-1)^{n} \nu(d\zeta)& \text{ if } k \geq 3. 
    \end{array}
\right.
\end{align}
Conversely, for any $\boxtimes$-infinitesimal character $\mathcal{LR}$ which satisfies (\ref{LRtransfoinfini}), there exists a free $\boxtimes$-infinitely divisible probability measure $\mu \in \mathcal{M}_{\boxtimes}$ such that $\mathcal{LR}(\mu) = \mathcal{LR}$. The triplet $(\omega, b, \nu)$ is the $\boxtimes$-characteristic triplet of $\nu$. 
\end{theorem}

In the articles \cite{Flor}, \cite{Caba}, Benaych-Georges and Cabanal-Duvillard showed that for any free $\boxplus$-infinitely divisible measure $\mu$, if $(\mu_{t})_{t \geq 0}$ is the continuous semi-group associated with $\mu$, there exists a sequence of Hermitian U-invariant L\'{e}vy processes $((X_{t}^{N})_{t \geq 0})_{N \in \mathbb{N}}$ such that the empirical eigenvalues distribution of $X_t^{N}$ converges in probability to $\mu_t$ as $N$ goes to infinity. In \cite{Cebron2}, C\'{e}bron extended this result to free $\boxtimes$-infinitely divisible probability measures $\mu \ne \lambda_{\mathbb{U}}$. If  $(\mu_{t})_{t \geq 0}$ is the canonical continuous semi-group associated with $\mu$, there exists a sequence of unitary U-invariant L\'{e}vy processes $((X_{t}^{N})_{t \geq 0})_{N \in \mathbb{N}}$ such that the empirical eigenvalues distribution of $X_t^{N}$ converges in probability to $\mu_t$ as $N$ goes to infinity.

In the additive case, we show that the L\'{e}vy process $(X_{t}^{N})_{t \geq 0}$ can be a symmetric $O$-invariant L\'{e}vy process, and in the multiplicative case, if $(\omega, b, \eta)$ is the characteristic triplet of $\mu$, if $\omega = 0$ and if for any continuous function $f: \mathbb{U} \to \mathbb{R}$, one has:  
\begin{align}
\label{eq:paire}
\int_{\mathbb{U}} f(\overline{z}) \nu(dz) = \int_{\mathbb{U}}f(z) \nu(dz), 
\end{align}
then the L\'{e}vy process $(X_{t}^{N})_{t \geq 0}$ can be an orthogonal $O$-invariant L\'{e}vy process. These results are consequences of the results in Section \ref{sec:convLevyex}.

\subsubsection{Generalities about L\'{e}vy processes: the generator}
In Section \ref{sec:convergencebrownien}, the convergence of the Brownian motions was proved by computing, using Itô's formula, the action of the generator at time $t=0$ on the application $M \mapsto M^{\otimes k}$ (Lemmas \ref{geneadd} and  \ref{genemult}). There exists a general result in order to compute the generator of a L\'{e}vy process given by Theorem $31.5$ in \cite{Sato} and Hunt's Theorem, Theorem $1.1$ in \cite{Liao}. The first theorem applies to additive L\'{e}vy processes. 

\begin{theorem}
\label{generateuradd}
Let $E$ be a finite dimensional vector space of dimension $d$, $(Y_i)_{i=1}^{d}$ be a basis of $E$ and $(X_t)_{t \geq 0}$ be an additive L\'{e}vy process in $E$. There exist: 
\begin{enumerate}
\item $Y_0 \in E$, 
\item a symmetric positive semidefinite matrix $(y_{i,j})_{i,j = 1}^{d}$, 
\item a L\'{e}vy measure $\Pi$ on $E$, that is a measure on $E$ such that $\Pi(\{0\})=0$ and such that, if $B$ is the ball of center $0$ and radius $1$ in $E$: 
\begin{align*}
\int_{B} \mid\mid x \mid\mid_{E}^{2} \Pi(dx) \leq \infty \text{ and } \Pi(B^{c})< \infty, 
\end{align*}
\end{enumerate}
such that the generator $G$ of $(X_t)_{t \geq 0}$ is given for any $f \in C^{2}_{0}(E)$ and any $y\in E$ by $Gf(y)= \frac{d}{dt}_{\mid t=0} \mathbb{E}[f(y+X_t)]$ which is equal to: 
\begin{align*}
  \partial_{X_0}f(y) + \frac{1}{2} \sum_{i,j=1}^{N^{2}} y_{i,j} \partial_{Y_i}\partial_{Y_j} f(y) + \int_{E}\left[ f(y+x)-f(y)-1\!\!1_{B}(x) \partial_{x}f(y)\right] \Pi(dx). 
\end{align*}
Conversely, every operator of this form is the generator of a unique L\'{e}vy process $(X_t)_{t \geq 0}$. 

Besides, let us suppose that $E$ is equal to the Lie algebra $\mathfrak{g}$ of a compact Lie group $G$. Let $H$ be a Lie subgroup of $G$. Let us suppose that $Y_0$, the operator $ \sum_{i,j=1}^{N^{2}} y_{i,j} \partial_{Y_i}\partial_{Y_j}$ and $\Pi$ are invariant by conjugation by any element of $H$, then the L\'{e}vy process associated is invariant by conjugation by $H$. 
\end{theorem}

\begin{remark}
\label{rm:voisinage}
One can change $1\!\!1_{B}$ in the form of the generator given by the last theorem by anything of the form $1\!\!1_{V}$ where $V$ is a neighborhood of $0$. This operation only changes the drift $X_0$. This remark will be important latter as we will work with a L\'{e}vy measure which is compactly supported: it is then easier to suppose that ${\sf Supp}(\Pi) \subset B$. 
\end{remark}

A similar result exists for compact Lie groups. Let $G$ be a compact Lie group, let $\mathfrak{g}$ be the Lie algebra of $G$ and ${\mathrm{Id}}$ be the neutral element of $G$. Let $A: G \to \mathfrak{g}$ be a smooth mapping such that $A(\mathrm{Id}) = 0\text{ and } d_{\mathrm{Id}}A ={\mathrm{ id}}_{\mathfrak{g}}.$ We recall also that any element $X$ in the Lie algebra $\mathfrak{g}$ induces a right invariant vector field which is defined for any $g$ in $\mathfrak{g},$ by $X^{r}(g) = DR_{g}(X),$ with $DR_{g}$ being the diffential map of the right multiplication operation $R_g : h \mapsto gh. $ Hunt's theorem, see \cite{Liao}, allows to compute the generator of a L\'{e}vy process in a compact Lie group. 
\begin{theorem}
\label{genemultilie}
Let $(X_t)_{t \geq 0}$ be a L\'{e}vy process on $G$, let $d$ be the dimension of the Lie algebra $\mathfrak{g}$ and $(Y_i)_{i=1}^{d}$ be a basis of $\mathfrak{g}$. There exist: 
\begin{enumerate}
\item $Y_0 \in \mathfrak{g}_N$, 
\item a symmetric positive semidefinite matrix $(y_{i,j})_{i,j = 1}^{d}$, 
\item a L\'{e}vy measure $\Pi$ on $G$, that is a measure on $G$ such that $\Pi(\{Id\})=0$ and for any neighborhood $V$ of $Id$ in $G$, we have: 
\begin{align*}
\int_{V} \mid\mid\! A(x)\! \mid\mid_{\mathfrak{g}}^{2} \Pi(dx) \leq \infty \text{ and } \Pi(V^{c}) < \infty, 
\end{align*}
\end{enumerate}
such that the generator $G$ of $(X_t)_{t \geq 0}$ is given for any $f \in C^{2}(G)$ and any $h\in {U}_N$ by $Gf(h) = \frac{d}{dt}_{\mid t=0} \mathbb{E}[f(X_th)]$ which is equal to:
\begin{align*}
Y_0^{r}f(h) + \frac{1}{2} \sum_{i,j=1}^{d} y_{i,j} Y_i^{r}Y_j^{r} f(h) + \int_{G}\left[ f(gh)-f(h)- A(g)^{r}f(h)\right] \Pi(dg). 
\end{align*}
Conversely, every operator of this form is the generator of a unique L\'{e}vy process $(X_t)_{t \geq 0}$. 

Let $H$ be a Lie subgroup of $G$. Let us suppose that $Y_0$, the measure $\Pi$ and the operator $\sum_{i,j=1}^{d} y_{i,j}Y_{i}^{l} Y_{j}^{l}$ are invariant by conjugation by any element of $H$, then the L\'{e}vy process associated is invariant by conjugation by $H$. 
\end{theorem}

Recall the notations in Section \ref{sec:convergencebrownien}. As explained by G. C\'{e}bron in \cite{Cebron2}, one can use on $U(N,\mathbb{K})$ the mapping: 
\begin{align*}
A: M \mapsto \frac{M - M^{*}}{2}. 
\end{align*}

We will denote this application $i\Im$, even when we are working on the group $O(N)$. It is invariant by conjugation by $U(N,\mathbb{K})$ since for any $U \in U(N, \mathbb{K})$, for any $M \in \mathcal{M}(N,\mathbb{K})$: $U(i\Im(M))U^{-1}\! =\! i\Im(UMU^{-1}).$ We also define $\Re (M) = \frac{M+M^{*}}{2}$. 

In Section \ref{sec:convergencebrownien}, we defined a scalar product on the Lie algebras $\mathfrak{g}^{\epsilon}(N,\mathbb{K})$ (Equation~(\ref{eq:scalar})). From now on, when $\mathfrak{g} = \mathfrak{g}^{\epsilon}(N,\mathbb{K})$, we will always assume that the basis $(Y_i)_{i=1}^{d}$ used in Theorems \ref{generateuradd} and \ref{genemultilie} is an orthonormal basis for this scalar product and we will not specify it anymore.

\subsubsection{Approximation of $\boxplus$-infinitely divisible probability measures and convergence of additive L\'{e}vy processes }
\label{sec:convLevyex}

Recall the notations taken in Section \ref{sec:convergencebrownien}, in particular~$\mathcal{A}_{\mathbb{K}}$.

\begin{theorem}
\label{convergenceLevy}
Let $\mu$ be a free $\boxplus$-infinitely divisible probability measure on $\mathbb{R}$ with associated semi-group $(\mu_t)_{t \geq 0}$ and associated characteristic triplet $(\eta, a,\rho)$. Let us suppose that the measure $\rho$ has a compact support. Let $d_{\mathfrak{h}(N,\mathbb{K})}$ be the dimension of $\mathfrak{h}(N,\mathbb{K})$. Let us define for any integer $N$: 
\begin{align*}
a_N &=a {\mathrm{Id}}_{d_{\mathfrak{h}(N,\mathbb{K})}}, \\
\rho_N (f) &= N \int_{\mathbb{R}} \int_{U(N,\mathbb{K})} f \left( g\begin{pmatrix}
  x & 0& \cdots &  0 \\
  0 & 0 & \ddots & \vdots  \\
 \vdots  & \ddots & \ddots &0  \\
  0 &  \cdots & 0& 0
 \end{pmatrix}\!\!\text{ }g^{*}\right) dg \rho(dx). 
\end{align*}
For any positive integer $N$, let $(X^{N}_t)_{t \geq 0}$ be a L\'{e}vy process on $\mathfrak{h}(N,\mathbb{K})$ with characteristic triplet $\big(\eta {\mathrm{Id}}_N, a_N, \rho_{N}\big)$. The process $(X^{N}_t)_{t \geq 0}$ converges in $\mathcal{P}$-distribution toward a $U(N,\mathbb{K})$-invariant additive $\mathcal{A}_{\mathbb{K}}$-L\'{e}vy process. It satisfies the asymptotic $\mathcal{P}$-factorization property: the $\mathcal{P}$-moments of $(H_{t}^{N})_{t \geq 0}$ converge in probability to the limit of their expectation. Besides, for any integer $k$ and any $t \geq 0$: 
\begin{align*}
m_{(1,...,k)}[X_t] =  \int_{\mathbb{R}} x^{k}  \mu_t(dx). 
\end{align*}
\end{theorem}

\begin{proof}
Let us consider the L\'{e}vy process $(X_t^{N})_{t \geq 0}$ specified in the theorem. Using the last assertion of Theorem \ref{generateuradd}, for any positive integer $N$, $(X^{N}_t)_{t \geq 0}$ is $U(N,\mathbb{K})$-invariant. Using Remark \ref{rm:voisinage}, we can suppose, without loss of generality, that for any $x\in {\sf Supp}(\rho)$, the matrix: $$\begin{pmatrix}
  x & 0& \cdots &  0 \\
  0 & 0 & \ddots & \vdots  \\
 \vdots  & \ddots & \ddots &0  \\
  0 &  \cdots & 0& 0
 \end{pmatrix}$$ is in the ball of center 0 and radius 1. For any integers $k$ and $N$, using Theorem \ref{generateuradd}, we have: 
  $$G_{k}^{N} =  \frac{d}{dt}_{| t=0} \mathbb{E}\left[(X_{t}^{N})^{\otimes k}\right] = \left\{
   \begin{array}{ll}
      \eta {\mathrm{Id}_N}& \text{ if } k=1, \\
      {a}C_{\mathfrak{h}(N,\mathbb{K})} + \int_{\mathfrak{h}(N,\mathbb{K})} g^{\otimes 2} \rho_{N}(dg), & \text{ if } k=2, \\
      \int_{\mathfrak{h}(N,\mathbb{K})}g^{\otimes k} \rho_{N}(dg) & \text{ if } k \geq 3. 
    \end{array}
\right.$$
According to Theorem \ref{convergencegenerale}, we have to show that $G_{k}^{N} \in \mathbb{C}[\rho_{N}(\mathcal{A}_{\mathbb{K},k})]$ converges for any integer $k$ as $N$ goes to infinity. Using Remark \ref{rq:convergenceG}, it is enough to prove that for any $p \in \mathcal{A}_{\mathbb{K}}$, the $p$-moment of $G_{k}^{N}$ converges. When $\mathbb{K} = \mathbb{R}$, $X_{t}^{N}$ is symmetric for any $t \geq 0$ and any integer $N$: the convergence of the $\mathfrak{S}$-moments of $G_{k}^{N}$ implies the convergence of the $\mathcal{B}$-moments of $G_{k}^{N}$. Thus, in the two cases, $\mathbb{K} = \mathbb{C}$ or $\mathbb{R}$, it remains to prove that the $\mathfrak{S}$-moments of $G_{k}^{N}$ converges to infinity in order to prove that $(X_t^{N})_{t \geq 0}$ converges in $\mathcal{P}$-distribution.

The convergence of $G_{1}^{N}$ is obvious and we have already understood the convergence of $C_{\mathfrak{h}(N,\mathbb{K})}$ in Section \ref{sec:convergencebrownien}. Let $k$ be an integer greater than $2$ and let $\sigma \in \mathfrak{S}_k$, the moment $m_{\sigma}\left( \int_{\mathfrak{h}(N,\mathbb{K})} g^{\otimes k} \rho_{N}(dg)\right)$ is equal to: 
\begin{align*}
\frac{1}{N^{{\sf nc}(\sigma \vee \mathrm{id}_k)}} {\sf Tr}\left[  \left(\int_{\mathfrak{h}(N,\mathbb{K})} g^{\otimes k} \rho_{N}(dg)\right) \!\text{ }^{t}\sigma\right] = \frac{1}{N^{{\sf nc}(\sigma \vee \mathrm{id}_k)-1}} \int_{\mathbb{R}} x^{k} \rho(dx). 
\end{align*}
where we have omit to write the representation $\rho_{N}$ in order not to confuse the reader. The number of cycles of $\sigma$ is ${\sf nc}(\sigma \vee \mathrm{id}_k)$, thus, 
\begin{align}
\label{convergencedesG}
m_{\sigma}\left( \int_{\mathfrak{h}(N,\mathbb{K})} g^{\otimes k} \rho_{N}(dg)\right) \underset{N \to \infty}{\longrightarrow} \delta_{\sigma \in [(1,...,k)]} \int_{\mathbb{R}}x^{k} \rho(dx), 
\end{align}
where $[(1,...,k)]$ is the set of all $k$-cycles in $\mathfrak{S}_k$. This implies that for any integer $k$, $\left(G_k^{N}\right)_{N \in \mathbb{N}}$ converges. Besides, for any integer $k$, the $k$-cycles are the only irreducible permutations in $\mathfrak{S}_k$. Again, using the fact that, when $\mathbb{K} = \mathbb{R}$, the matrices are symmetric, this proves in the two cases, $\mathbb{K} = \mathbb{C}$ or $\mathbb{R}$, that $(G_{k}^{N})_{N \in\mathbb{N}}$ condensates for any integer $k$. By  Theorem \ref{convergencegenerale}, the process $(X_t^{N})_{t \geq 0}$ converges in $\mathcal{P}$-distribution toward a $U(N,\mathbb{K})$-invariant additive $\mathcal{A}_{\mathbb{K}}$-L\'{e}vy process. It satisfies the asymptotic $\mathcal{P}$-factorization, and since $(X_t^{N})_{t \geq 0}$ is stable by the adjoint operation, by Theorem \ref{theoremconvprob} the $\mathcal{P}$-moments of $(H_{t}^{N})_{t \geq 0}$ converge in probability to the limit of their expectation. 

Let $k$ be an integer and $t \geq 0$. In order to prove that $m_{(1,...,k)}[X_t] = \int_{\mathbb{R}}z^{k} \mu_{t}(dz)$ we only have to prove that $\mathcal{R}[\mu_t] = \mathcal{R}[X_t]$. Using the relation between cumulants and moments, Theorem 5.6 of \cite{Gab1}, we can compute $\mathcal{R}[G]$. We already know that it is a $\boxplus$-infinitesimal character and: 
\begin{align*}
\mathcal{R}[G]( (1,...,k)) = \left\{
   \begin{array}{ll}
      \eta & \text{ if } k=1, \\
      a + \int_{\mathbb{R}} x^{2} \rho(dx),  & \text{ if } k=2, \\
     \int_{\mathbb{R}} x^{k} \rho(dx) & \text{ if } k\geq 3,  
    \end{array}
\right.
\end{align*}
Using Theorem \ref{caractaddfreeinf}, the restriction of $\mathcal{R}[G]$ to $\bigoplus_{k} \mathbb{C}[\mathfrak{S}_k/\mathfrak{S}_k]$, denoted by $\mathcal{R}[G]_{| \mathfrak{S}}$ is equal to $\mathcal{LR}(\mu)$. Using Theorem \ref{convergencegenerale}, for any real $t \geq 0$, $\mathcal{R}[X_t] = e^{\boxplus t \mathcal{R}[G]}$. Using Lemma $3.1$ of \cite{Gab1}, for any permutation $\sigma \in \mathfrak{S}_k$ and any $b \in \mathcal{B}_k\setminus \mathfrak{S}_k$, $b \nleq \sigma$. Thus: 
\begin{align*}
m_{(1,...,k)}(X_t) = \sum_{\sigma \in \mathfrak{S}_k | \sigma \leq (1,...,k)} \mathcal{R}[X_t](\sigma) &= \sum_{\sigma \in \mathfrak{S}_k | \sigma \leq (1,...,k)} e^{\boxplus t \mathcal{R}[G]}(\sigma)  \\&= \sum_{\sigma \in \mathfrak{S}_k | \sigma \leq (1,...,k)} e^{\boxplus t \mathcal{R}[G]_{| \mathfrak{S}}}(\sigma) \\
&=\sum_{\sigma \in \mathfrak{S}_k | \sigma \leq (1,...,k)} e^{\boxplus t \mathcal{LR}(\mu)}(\sigma) \\
&= \sum_{\sigma \in \mathfrak{S}_k | \sigma \leq (1,...,k)} \mathcal{R}(\mu_t)(\sigma) = \int_{\mathbb{R}} x^{k} \mu_t(dx). 
\end{align*}
This allows us to conclude the proof. 
  \end{proof}

\subsubsection{Approximation of $\boxtimes$-infinitely divisible probability measures and convergence of multiplicative L\'{e}vy processes }
\begin{theorem}
\label{convergenceunitaire}
Let $\mu$ be a free $\boxtimes$-infinitely divisible probability measure on the circle $\mathbb{U}$ with associated canonical semi-group $(\mu_t)_{t \geq 0}$ and associated characteristic triplet $(\omega, b, \nu)$. If $\mathbb{K} = \mathbb{R}$, let us suppose that $\nu$ satisfies Equation (\ref{eq:paire}) for any continuous function $f : \mathbb{U} \to \mathbb{R}$ and that $\omega = 0$. Let $d_{{\mathfrak{u}(N,\mathbb{K})}}$ be the dimension of $\mathfrak{u}(N,\mathbb{K})$. Let us define for any integer $N$: 
\begin{align*}
b_{N} &= b \mathrm{Id}_{d_{{\mathfrak{u}(N,\mathbb{K})}}},\\
\nu_{N}(f) &=  \left\{
   \begin{array}{ll}
      N\int_{\mathbb{U}}  \int_{U(N)} f\left(g \begin{pmatrix}
  \zeta & 0& \cdots &  0 \\
  0 & 1 & \ddots & \vdots  \\
 \vdots  & \ddots & \ddots &0  \\
  0 &  \cdots & 0& 1
 \end{pmatrix} g^{*} \right)dg \nu(d\zeta) & \text{ if } \mathbb{K}=\mathbb{C}, \\
       N\int_{[-\pi,\pi]}  \int_{O(N)} f\left(g  \begin{pmatrix}
  cos\theta & -sin\theta & 0 & \cdots & 0 \\
  sin\theta & cos\theta & 0 &  \cdots &0\\
  0&0&1&\ddots & \vdots \\
  \vdots  & \vdots  & \ddots & \ddots & 0  \\
  0 & 0 & \cdots &  0& 1
 \end{pmatrix}\!\!\text{ }^{t}\!g \right)dg \nu(d\theta), & \text{ if } \mathbb{K}=\mathbb{R}, \\
    \end{array}
\right. \end{align*}
where in the last equality, we considered $\nu$ as a measure on $[-\pi,\pi]$.

For any positive integer $N$, let $(Y_t^{N})_{t \geq 0}$ be a L\'{e}vy process on $U(N,\mathbb{K})$ with characteristic triplet $(i {\sf arg}(\omega){\mathrm{Id}}_N, b_N, \nu_{N})$. The process $(Y^{N}_t)_{t \geq 0}$ converges in $\mathcal{P}$-distribution toward a $U(N,\mathbb{K})$-invariant $\mathcal{A}_{k}$-L\'{e}vy process. Moreover, the family $(Y_t^{N}, (Y_{t}^{N})^{*})_{t \geq 0}$ converges in $\mathcal{P}$-distribution and satisfies the asymptotic $\mathcal{P}$-factorization property. In particular, it converges in probability in $\mathcal{P}$-distribution toward its $\mathcal{P}$-distribution. For $t\geq 0$ and any $k \in \mathbb{N}$,
\begin{align*}
m_{(1,...,k)}[Y_t] = \int_{\mathbb{U}} z^{k} d \mu_t.
\end{align*}
\end{theorem}

\begin{proof}
The difference between the two cases $\mathbb{K} = \mathbb{C}$ and $\mathbb{K} = \mathbb{R}$ are handled using a similar argument as in the proof of Theorem \ref{convergenceLevy}: we saw before the Remark \ref{rq:convsigma} that for a sequence of orthogonal matrix the convergence in $\mathfrak{S}$-moments implies the convergence in $\mathcal{B}$-moments. 

Thus, we will focus on the case $\mathbb{K} = \mathbb{C}$ which is the more complicated in our setting since we will have to prove the assertion about the family $(Y^{N}_t,(Y^{N}_t)^{*})_{t \geq 0}$. The structure of the proof follows the one of Theorem \ref{convergenceLevy}. Let us consider the L\'{e}vy process $(Y_t^{N})_{t \geq 0}$ specified in the theorem. Using the last assertion of Theorem \ref{genemultilie}, for any positive integer $N$, $(Y^{N}_t)_{t \geq 0}$ is a L\'{e}vy process invariant by conjugation by $U(N)$.  Let $k$, $N$ be two integers. Applying Theorem \ref{genemultilie}: 
\begin{align*}
G_k^{N} &= \frac{d}{dt}_{\mid t = 0} \mathbb{E}\left[\left(Y_t^{N}\right)^{\otimes k}\right] \\&= k i {\sf arg}(w)\mathrm{Id}^{\otimes k} + \frac{b}{2} k  c_{\mathfrak{g}} \mathrm{Id}^{\otimes k} +\frac{1}{2} \sum_{i,j=1}^{k} \mathcal{I}_{\{i,j\}} (C_{\mathfrak{g}^{1}(N,\mathbb{K})}, \mathrm{Id}^{\otimes k-2}) \\
&+ \int_{U(N)} \left(M^{\otimes k} - \mathrm{Id}^{\otimes k} - \sum_{i=1}^{k} \mathcal{I}_{\{i\}} (i \Im(M), \mathrm{Id}^{\otimes k-1})\right)\nu_{N}(dM). 
\end{align*}
G. C\'{e}bron noticed in \cite{Cebron2} that one should use, in the last term, the following equality: 
$$M^{\otimes k}\!=\! (M-\mathrm{Id} +\mathrm{Id})^{\otimes k} = \mathrm{Id}^{\otimes k} + \sum_{m=1}^{k}\sum_{1\leq i_1 < i_2 < ... <i_m\leq k} \!\!\mathcal{I}_{\{i_1,...,i_m\}}\left[(M-\mathrm{Id})^{\otimes m}, \mathrm{Id}^{\otimes k-m}\right]\!\!.$$
Since $M - \mathrm{Id} - i \Im(M) = \Re(M) - \mathrm{Id}$, one gets: 
\begin{align*}
G_k^{N} & = k i {\sf arg}(w)\mathrm{Id}^{\otimes k} + \sum_{i=1}^{k} \mathcal{I}_{\{i\}}\left[\int_{U(N)}(\Re(M)-\mathrm{Id}) \nu_{N} (dM), \mathrm{Id}^{\otimes k-1} \right] \\
&+ \frac{b}{2} k  c_{\mathfrak{g}} \mathrm{Id}^{\otimes k}+  \frac{b}{2} \sum_{i\neq j, i,j=1}^{k} \mathcal{I}_{\{i,j\}} (C_{\mathfrak{g}^{1}(N,\mathbb{K})}, \mathrm{Id}^{\otimes k-2})\\
&+ \sum_{2\leq m \leq l} \sum_{1 \leq i_1 < ... <i_m \leq k}\!\!\! \mathcal{I}_{\{i_1,..., i_m\}}\left[\int_{U(N)} (M-\mathrm{Id})^{\otimes m} \nu_{N}(dM), \mathrm{Id}^{\otimes k-m}\right]\!\!, 
\end{align*} 
result obtained in Proposition 5.3. of \cite{Cebron2}. The proof now differs from the one of G. C\'{e}bron and this is what allows us to have the result for orthogonal matrices at no cost. According to Theorem \ref{convergencegenerale}, we have to show that $G_k^{N} \in \mathbb{C}[\mathfrak{S}_k]$ converges for any integer $k$. Using  Remark \ref{rq:convergenceG}, it is enough to prove that for any permutation $\sigma \in \mathfrak{S}_k$, the $p$-moment of $G_k^{N}$ converges. 

Using our work on the Brownian motion and on the Casimir element, it remains to understand, for any  $m  \in \{2,..., k\}$ and any $1 \leq i_1 < ... <i_m \leq k$, the limit of the $\mathfrak{S}$-moments of: 
\begin{align*}
\left(A_N = \sum_{i=1}^{k} \mathcal{I}_{\{i\}}\left[\int_{U(N)}(\Re(M)-\mathrm{Id}) \nu_{N} (dM), \mathrm{Id}^{\otimes k-1} \right] \right)_{N \in \mathbb{N}}
\end{align*}
and 
\begin{align*}
\left(B_N=\mathcal{I}_{\{i_1,..., i_m\}}\left[\int_{U(N)} (M-\mathrm{Id})^{\otimes m} \nu_{N}(dM), \mathrm{Id}^{\otimes k-m}\right] \right)_{N \in \mathbb{N}}.
\end{align*}
 Let $\sigma$ be in $\mathfrak{S}_k$: $m_{\sigma}(A_N) = k \int_{\mathbb{U}} (\Re(\zeta)-1) \nu(d\zeta)$. Thus $A_N$ converges and using the cumulant-moment relation, $\kappa_{\sigma}(A) =\left(k  \int_{\mathbb{U}} (\Re(\zeta)-1) \nu(d\zeta) \right) \delta_{\sigma = id_k}.$ Let $m$ be a positive integer. In order to study the convergence of $(B_N)_{N \in \mathbb{N}}$, we only need to study the sequence $(\tilde{B}_{N} = \int_{U(N)}  (M-\mathrm{Id})^{\otimes m} \nu_{N}(dM) )_{N \in \mathbb{N}}.$ Since: 
 \begin{align*}
m_{\sigma}(\tilde{B}_N) = \frac{N}{N^{{\sf nc}(\sigma\vee id)}}\int_{\mathbb{U}} ( \zeta-1)^{m} \nu(d\zeta), 
\end{align*}
the $\sigma$-moment of $ \tilde{B}_N$ converge to  $m_{\sigma}( \tilde{B})=\delta_{\sigma \in [(1,...,m)]}\int_{\mathbb{U}} (\zeta - 1)^{m} \nu(d\zeta),$ where we recall that $ [(1,...,m)]$ is the set of $m$-cycles in $\mathfrak{S}_m$. Thus $\tilde{B}_N$ converges and using the cumulant-moment relation, for any $\sigma \in \mathfrak{S}_m$, $\kappa_{\sigma}(\tilde{B})  =  \delta_{\sigma \in [(1,...,m)]}\int_{\mathbb{U}} (\zeta - 1)^{m} \nu(d\zeta)$.

This discussion allows us to asser that for any integer $k$, $G_k^{N}$ converges as $N$ goes to infinity. Besides, we have already computed $\mathcal{R}[G]$: 
\begin{align*}
\mathcal{R}[G] = &\sum_{k \geq 1} k\left( i {\sf arg}(\omega) - \frac{b}{2} + \int_{\mathbb{U}} (\Re(\zeta)-1) \nu(d\zeta) \right)\mathrm{id}_k^{*} \\ &- \sum_{k \geq 2} b ((1,2)\otimes \mathrm{id}_{k-2} )^{*} + \sum_{2 \leq m \leq k} \left(\int_{\mathbb{U}} (\zeta - 1)^{m} \nu(d\zeta)\right)( (1,...,m)\otimes \mathrm{id}_{k-m})^{*}, 
\end{align*}
where we recall that $(1,...,m)$ is the usual $m$-cycle in $\mathfrak{S}_m$. By Theorem \ref{convergencegenerale}, the process $\left(Y_t^{N}\right)_{t \geq 0}$ converges in $\mathcal{P}$-distribution toward a $U(N,\mathbb{K})$-invariant multiplicative $\mathcal{A}_{\mathbb{K}}$-L\'{e}vy process. Besides, for any $t \geq 0$, $e^{\boxplus t \mathcal{R}[G]} = \mathcal{R}[Y_{t}]$. Using Theorem \ref{caractmultfreeinf}, we see that $\mathcal{R}[G] = \mathcal{LR}(\mu)$: for any $t \geq 0$, $\mathcal{R}[Y_{t}] = \mathcal{R}[\mu_t]$, and then $\mathbb{E}m_{(1,...,k)}[Y_t] = \int_{\mathbb{U}}z^{k} \mu_t(dz)$ for any integer $k$. Using Theorem \ref{theoremconvprob}, the assertion about $m_{(1,...,k)}[Y_t]$ is a consequence of this discussion and the assertion about the family $(Y_t^{N})_{t \geq 0} \cup (Y_t^{N})^{*})_{t \geq 0}$.

It remains to prove the assertion about the family $(Y_t^{N})_{t \geq 0} \cup (Y_t^{N})^{*})_{t \geq 0}$. Equivalently, we need to prove that the family  $(Y_t^{N})_{t \geq 0} \cup (\overline{Y_t^{N}}))_{t \geq 0}$ converges in $\mathcal{P}$-distribu\-tion and satisfies the asymptotic $\mathcal{P}$-factorization property. Let $k$ and  $k'$ be two integers. Let us define:
\begin{align*}
G^{N}_{k,k'} = \frac{d}{dt}_{\mid t = 0} \mathbb{E}\left[ (Y_{t}^{N})^{\otimes k} \otimes (\overline{Y_{t}^{N}})^{\otimes k'}\right].
\end{align*}
Let us remark that $G_{k,k'}^{N}$ is actually in $\mathbb{C}\left[\rho_{N}^{\mathcal{B}_{k+k'}}\right]$ since for any permutation $\sigma \in \mathfrak{S}_{k+k'}$, ${\sf S}_{k}(\sigma) \in \mathcal{B}_{k+k'}$. Using Theorem \ref{genemultilie}, we can compute $G^{N}_{k,k'}$ and we can see that it is composed of three parts: a drift and Brownian parts which were already studied, and a third part which remains to be understood:  
\begin{align*}
\int_{U(N)} \bigg( M^{\otimes k} \otimes \overline{M}^{\otimes k'} - \mathrm{Id}^{\otimes k+k'} &- \sum_{i=1}^{k} \mathcal{I}_{\{i\}} (i \Im(M), \mathrm{Id}^{\otimes k+k'-1}) \\&+ \sum_{i=k+1}^{k+k'} \mathcal{I}_{(i)}(i ^{t}\Im(M), \mathrm{Id}^{\otimes k+k'-1})\bigg) \nu_{N}(dM), 
\end{align*}
where we used the fact that $\overline{i \Im(M)} = - ^{t}(i\Im(M))$. Using the usual argument, we know that $M^{\otimes k} \otimes \overline{M}^{\otimes k'} - \mathrm{Id}^{\otimes k+k'}$ is equal to: 
\begin{align*}
\sum_{m=1, l=1}^{k,k'}\sum_{1\leq i_1<...<i_m\leq k < j_1 \leq ... \leq j_l \leq k+k'} \mathcal{I}&_{\{i_1,...,i_m, j_1,...,j_l \}} \\&\left[ (M-\mathrm{Id})^{\otimes m} \otimes (\overline{M}-\mathrm{Id})^{\otimes l}, \mathrm{Id}^{\otimes k+k' - m- l}\right].
\end{align*}
Let us remark also that $\overline{M} +\ \!\! ^{t}(i\Im(M)) =\ \!\! ^{t} \Re(M)$. As for the study of $G_k^{N}$, after some simple calculations, we see that we need to study, for any $(m,l) \in \{1,...,k\}\times\{1,...,k'\}$, the limit~of: 
\begin{align*}
&\int_{U(N)} (\Re(M) - \mathrm{Id}) \nu_{N} (dM), \\
&\int_{U(N)}( { }^{t}\Re(M) - \mathrm{Id}) \nu_{N} (dM),\\
&\int_{U(N)} (M-\mathrm{Id})^{\otimes m} \otimes \left(\overline{M}-\mathrm{Id}\right)^{\otimes l} \nu_{N} (d M).
\end{align*}
The two first sequences are easy to study, let us focus only on the last one. Let $(m,l)$ in $\{1,...,k\}\times\{1,...,k'\}$ and let us study: $$C_N = \int_{U(N)} (M-\mathrm{Id})^{\otimes m} \otimes \left(\overline{M}-\mathrm{Id}\right)^{\otimes l} \nu_{N} (d M). $$
Recall the operation ${\sf S}_{k}$ defined in Definition \ref{def:Sk}. With an obvious abuse of notations, the element ${\sf S}_m(C_N)$ is equal to $\int_{U(N)} (M-\mathrm{Id})^{\otimes m} \otimes (M^{*}-\mathrm{Id})^{\otimes l} \nu_{N}(dM)$ which commutes with the tensor action of $U(N)$: it is an element of $\mathbb{C}[\mathfrak{S}_{m+l}(N)]$.  Since the convergence of ${\sf S}_{m}(C_N)$ is equivalent to the convergence of $C_N$, the convergence of the $\mathfrak{S}$-moments of ${\sf S}_{m}(C_N)$ implies the convergence of $C_N$. For any $\sigma \in \mathfrak{S}_{m+l}$, it is easy to see that $m_{\sigma} ({\sf S}_{m}(C_N))$ converges and: 
\begin{align*}
m_{\sigma}({\sf S}_{m}(C)) = \delta_{\sigma \in [(1,...,m+l)]} \int_{\mathbb{U}} (\zeta-1)^{m}(\overline{\zeta}-1)^{l} \nu(d\zeta),
\end{align*}
where we remind that $[(1,...,m+l)]$ is the set of $m+l$ cycles in ${\mathfrak{S}_{m+l}}$. Using the cumulant-moment relation,  for any $\sigma \in \mathfrak{S}_{m+l}$: 
\begin{align*}
 \kappa_{\sigma}({\sf S}_m(C)) = \delta_{\sigma \in [(1,...,m+l)]}\int_{\mathbb{U}} (\zeta-1)^{m}(\overline{\zeta}-1)^{l} \nu(d\zeta).
\end{align*}
Hence for any $\sigma \in \mathfrak{S}_{m+l}$: 
\begin{align*}
\kappa_{\sigma}(C) =\delta_{\sigma \in {\sf S}_{m}([(1,...,m+l)])} \left(\int_{\mathbb{U}} (\zeta-1)^{m}(\overline{\zeta}-1)^{l} \nu(d\zeta)\right).
\end{align*}
At the end, we see that for any integer $k$ and $k'$, $G^{N}_{k,k'}$ converges as $N$ goes to infinity and $(G^{N}_{k,k'})_{k,k'}$ weakly condensates. Using Theorem \ref{convergence*levy}, this implies that the family $\left(Y_{t}^{N}, (Y_t^{N})^{*}\right)_{t\geq 0}$ converges in $\mathcal{P}$-distribution as $N$ goes to infinity and the $\mathcal{P}$-asymptotic factorization property holds for this family. In particular, the $\mathcal{P}$-moments of $\left(Y_{t}^{N}, (Y_t^{N})^{*}\right)_{t\geq 0}$ converges in probability to the limit of their expectation. 
  \end{proof}

Using the up-coming Theorem \ref{theoremconvprob2} and the fact that the defect of any permutation with any Brauer element is an even integer, one can easily show that all the convergences we proved on Brownian an L\'{e}vy processes hold almost surely.

\section{Algebraic fluctuations}
\label{sec:algfluctu}
\subsection{Basic definitions}
For any integer $N$, let $(M_i^{N})_{i \in \mathcal{I}}$ be a family of random matrices which converges in $\mathcal{A}$-distribution.  In this section, we study the asymptotic developments of the $\mathcal{A}$-distribution of $(M_i^{N})_{i \in \mathcal{I}}$. We will often ommit the proof of the results since they are similar to their $0$-order counterpart.  

\begin{definition}
The family $(M_i^{N})_{i \in \mathcal{I}}$ converges in $\mathcal{A}$-distribution up to order $n$ of fluctuations if for any integer $k$, any $p \in \mathcal{A}_k$, any $i \in \{0,...,n\}$ and any $(i_1,...,i_k) \in \mathcal{I}^k$, there exists a real $\mathbb{E}m_{p}^{i}(M_{i_1},...,M_{i_k})$, called the $i^{th}$-order fluctuations of the $p$-moment, such that: 
\begin{align*}
N^{n}\left( \mathbb{E}m_p(M^{N}_{i_1},...,M^{N}_{i_k})- \sum_{i=0}^{n-1} \frac{ \mathbb{E}m_p^{i}(M_{i_1},...,M_{i_k})}{N^{i}} \right) \underset{N \to \infty}{\longrightarrow} \mathbb{E}m^{n}_{p}( M_{i_1},...,M_{i_k} ).
\end{align*}
\end{definition}

Let us remark that the convergence in  $\mathcal{A}$-distribution up to order $n$ of fluctuations of $(M_i^{N})_{i \in \mathcal{I}}$ or of the algebra generated by $(M_i^{N})_{i \in \mathcal{I}}$ are equivalent. Recall the notion of defect ${\sf df}(p',p)$ defined in Equation (1) of \cite{Gab1} and recall the notion of finite dimensional $\mathcal{A}$-cumulants. The following theorem is a generalization of part of Theorem \ref{th:lienlimite} and is a consequence of Theorem 6.1 of \cite{Gab1}.

\begin{theorem}\label{ma-cum-mom-fluctu}
The family $(M_i^{N})_{i \in \mathcal{I}}$ converges in $\mathcal{A}$-distribution up to order $n$ of fluctuations if and only if for any integer $k$, any $p \in \mathcal{A}_k$, any $i \in \{0,...,n\}$ and any $(i_1,...,i_k) \in \mathcal{I}^k$, there exists a real $\mathbb{E}\kappa_{p}^{i,\mathcal{A}}(M_{i_1},...,M_{i_k})$, called the $i^{th}$-order fluctuations of the $p$-$\mathcal{A}$-cumulant, such that: 
\begin{align*}
N^{n}\left( \mathbb{E}\kappa_p^{\mathcal{A}}(M^{N}_{i_1},...,M^{N}_{i_k})- \sum_{i=0}^{n-1} \frac{ \mathbb{E}\kappa_p^{i,\mathcal{A}}(M_{i_1},...,M_{i_k})}{N^{i}} \right) \underset{N \to \infty}{\longrightarrow} \mathbb{E}\kappa^{n,\mathcal{A}}_{p}( M_{i_1},...,M_{i_k} ).
\end{align*}
Let us suppose that $(M_i^{N})_{i \in \mathcal{I}}$ converges in $\mathcal{A}$-distribution up to order $n$ of fluctuations, then, for any integer $k$, any $p \in \mathcal{A}_k$, any $i \in \{0,...,n\}$ and any $(i_1,...,i_k) \in \mathcal{I}^{k}$: 
\begin{align*}
\mathbb{E}m^{i}_{p}\left(M_{i_1},...,M_{i_k}\right) = \sum_{p' \in \mathcal{A}_k, {\sf df}(p',p)\leq i} \mathbb{E}\kappa_{p'}^{i-{\sf df}(p',p),\mathcal{A}}\left[M_{{i_1}},..., M_{{i_k}}\right]. 
\end{align*}
\end{theorem}

From now on, let us suppose that $(M_i^{N})_{i \in \mathcal{I}}$ converges in $\mathcal{A}$-distribution up to order $n$ of fluctuations.  

\begin{remark}
\label{remarque:fluctuhigher}
Let us denote by $\mathbb{E}\kappa_{p}^{n,\mathcal{A}}(M_{i_1}^{N}, ..., M_{i_k}^{N})$  the difference  $ \mathbb{E}\kappa_p^{\mathcal{A}}(M^{N}_{i_1},...,M^{N}_{i_k})- \sum_{i=0}^{n-1} \frac{ \mathbb{E}\kappa_p^{i,\mathcal{A}}(M_{i_1},...,M_{i_k})}{N^{i}} $. By definition, $\int_{\mathcal{G}(\mathcal{A})(N)} g^{\otimes k} \mathbb{E}[M_{i_1} \otimes...\otimes M_{i_k}] (g^{*})^{\otimes k} $ is equal to:
\begin{align*}
\sum_{p \in \mathcal{A}_k}  \left[ \sum_{i=1}^{n-1} \frac{\mathbb{E}\kappa_{p}^{i, \mathcal{A}}(M_{i_1}, ..., M_{i_k})}{N^{i}} + \frac{\mathbb{E}\kappa_{p}^{n, \mathcal{A}}(M^{N}_{i_1}, ..., M^{N}_{i_k})}{N^{n}}\right] \rho_N(p).
\end{align*}
\end{remark}

Let $i_0$ be in $\mathcal{I}$. We can define the $\mathcal{R}^{(n)}_{\mathcal{A}}$-transform of $(M_{i_0}^{N})_{N \in \mathbb{N}}$.

\begin{definition}
The $\mathcal{R}_{\mathcal{A}}^{(n)}$-transform of $(M_{i_0}^{N})_{N \in \mathbb{N}}$ is the linear form which sends $(p,i) \in \mathcal{A}\times \{0,...,n\}$ on $\mathbb{E}\kappa_{p}^{i,\mathcal{A}}[M_{i_0},...,M_{i_0}]$.
\end{definition}

Let us remark that Theorem \ref{theoremconvprob} can be generalized in order to have almost sure convergence. For any integer $N$, let $(M^{N}_i)_{i \in \mathcal{I}}$ be a family of random matrices which converges in $\mathcal{A}$-distribution up to order $1$ of fluctuations. 
\begin{definition}
The family  $(M^{N}_i)_{i \in \mathcal{I}}$ satisfies the strong asymptotic $\mathcal{A}$-factorization property if it satisfies the asymptotic $\mathcal{A}$-factorization property and for any integers $k_1$ and $k_2$, any $(i_1,...,i_{k_1+k_2}) \in \mathcal{I}^{k_1+k_2}$, any partitions $p_1\in \mathcal{P}_{k_1}$ and $p_2\in \mathcal{P}_{k_2}$:
\begin{align*}
\mathbb{E}m^{1}_{p_1 \otimes p_2} (M_{i_1},...,M_{i_{k_1+k_2}}) = \sum_{i=0}^{1}\mathbb{E}m^{i}_{p_1} (M_{i_1},...,M_{i_{k_1}}) \mathbb{E}m^{1-i}_{p_2} (M_{i_{k_1+1}},...,M_{i_{k_1+k_2}}).
\end{align*} 
\end{definition}
With this definition, we can state the generalization of Theorem \ref{theoremconvprob}, whose proof is totaly similar, except that one can prove that the variances are summable hence the almost-sure convergence result. 

\begin{theorem}
\label{theoremconvprob2}
For any integer $N$, let $(M^{N}_i)_{i \in \mathcal{I}}$ be a family of {\em real} random matrices which converges in $\mathcal{A}$-distribution. If $(M^{N}_i)_{i \in \mathcal{I}}$ satisfies the strong asymptotic $\mathcal{A}$-factorization property then it converges almost surely in $\mathcal{A}$-distribution and for any integer $k$, any $p \in \mathcal{A}_k$, any $i_1,...,i_k \in I$, $m_{p}(M_{i_1}, ..., M_{i_k}) = \mathbb{E}m_{p}(M_{i_1}, ..., M_{i_k})$.  

If for any integer $N$, $(M^{N}_i)_{i \in \mathcal{I}}$ is a family of complex random matrices, the same result holds if we suppose that $(M^{N}_i)_{i \in \mathcal{I}}$ is stable by the conjugate or adjoint operations.
\end{theorem}

\begin{remark}
Using the forthcoming Theorem \ref{convergencegeneralefluctua}, all the examples of L\'{e}vy processes we considered in this paper satisfy the strong asymptotic $\mathcal{A}$-factorization property. This is due to the fact that for any permutation $\sigma$,  and any Brauer element $b$, the defect ${\sf df}(b,\sigma) \in 2\mathbb{N}$: this implies that for any L\'{e}vy processes we considered the $1$-order fluctuation of the $\mathfrak{S}$-moments were equal to zero: the strong asymptotic $\mathfrak{S}$-factorization property holds.
\end{remark}

When $(M_{i}^{N})_{i \in \mathcal{I}}$ is $\mathcal{G}(\mathcal{A})$-invariant, there exists a generalization of Theorem \ref{th:invarianceasymp}.

\begin{theorem}
\label{AentrainePfluctu}
Let us suppose that  $(M_{i}^{N})_{i \in \mathcal{I}}$ is $\mathcal{G}(\mathcal{A})$-invariant and converges in $\mathcal{A}$-distribution up to order $n$ of fluctuations, then it converges in $\mathcal{P}$-distribution up to order $n$ of fluctuations. For any integer $k$, any $p \in \mathcal{P}_k$, any $i\in \{0,...,n\}$ and any $(i_1,...,i_k) \in \mathcal{I}^{k}$: 
\begin{align*}
 \mathbb{E}m_{p}^{i}(M_{i_1},..., M_{i_k}) = \sum_{p' \in \mathcal{A}_k, {\sf df}(p', p)\leq i}  \mathbb{E}\kappa^{i-{{\sf df}(p',p)}, \mathcal{A}}_{p'}\left(M_{i_1},..., M_{i_k}\right). 
\end{align*}
More generally for any $\mathcal{A}'$ such that $\mathcal{A} \subset \mathcal{A}'$, for any integer $k$, any $p \in \mathcal{A}'_k$, any $(i_1,...,i_k) \in \mathcal{I}^{k}$ and any $i\in \{0,...,n\}$:
\begin{align*}
\mathbb{E}\kappa_p^{i,\mathcal{A}'} [M_{i_1}, ..., M_{i_k}]= \delta_{p \in \mathcal{A}} \mathbb{E}\kappa_{p}^{i,\mathcal{A}} [M_{i_1}, ..., M_{i_k}].
\end{align*} 
\end{theorem}

\subsection{Freeness up to higher order of fluctuations}
A notion of asymptotic $\mathcal{A}$-freeness up to higher order of fluctuations can be defined for families which converge in $\mathcal{A}$-distribution up to order $n$ of fluctuations. For any integer $N$, let $(M^{N}_i)_{i \in \mathcal{I}}$ and $(L^{N}_j)_{j \in \mathcal{J}}$ be two families of random matrices. We recall that we always supposed, for sake of simplicity, that $\mathcal{I}\cap \mathcal{J} = \emptyset$. Let us suppose that the family  $(M^{N}_i)_{i \in \mathcal{I}} \cap (L^{N}_j)_{j \in \mathcal{J}}$ converges in $\mathcal{A}$-distribution up to order $n$ of fluctuations.

\begin{definition}
The families $(M^{N}_i)_{i \in \mathcal{I}}$ and $(L^{N}_j)_{j \in \mathcal{J}}$ are asymptotically $\mathcal{A}$-free up to order $n$ of fluctuations if the two following conditions hold: 
\begin{description}
\item[$\bullet$ \textbf{compatibility condition: }] for any integer $k$, any $i \in \{0,...,n\}$, any $p \in \mathcal{A}_k$, any $(B^{N}_1, ...,B^{N}_k) \in \left((M_i^{N})_{i \in \mathcal{I}}\cup (L_j^{N})_{j \in \mathcal{J}}\right)^{k}$, if there exists $i$ and $j$ in the same block of $p$ such that $\{ B_i^{N},B_j^{N}\} \subset (M_i^{N})_{i \in \mathcal{I}}$ of $\{ B_i^{N},B_j^{N}\} \subset (L_j^{N})_{i \in \mathcal{J}}$, then:
\begin{align*}
\mathbb{E}\kappa_{p}^{i,\mathcal{A}}(B_1,...,B_k)=0, 
\end{align*}

\item[$\bullet$ \textbf{factorization property: }] for any integers $k$ and $k'$, any $i \in \{0,...,n\}$, any $p_1 \in \mathcal{P}_{k}$, any $p_2 \in \mathcal{P}_{k'}$, any $(i_1,...,i_{k}) \in \mathcal{I}^{k}$ and any $(j_{1},...,j_{k'}) \in \mathcal{J}^{k'}$, 
\begin{align*}
\mathbb{E}\kappa_{p_1\otimes p_2}^{i, \mathcal{A}}(M_{i_1},...,M_{i_{k}}, L_{j_1}, ..., L_{j_{k'}})\!\!=\!\! \sum_{i'=0}^{i}\!\mathbb{E}\kappa_{p_1}^{i', \mathcal{A}}(M_{i_1},...,M_{i_{k}}) \mathbb{E}\kappa_{p_2}^{i-i', \mathcal{A}}(L_{j_1},...,L_{j_{k'}}).
\end{align*} 
\end{description}
\end{definition}

Actually, it is easy to see that the  families $(M^{N}_i)_{i \in \mathcal{I}}$ and $(L^{N}_j)_{j \in \mathcal{J}}$ are asymptotically $\mathcal{A}$-free up to order $n$ of fluctuations if and only if  the algebras generated by $(M^{N}_i)_{i \in \mathcal{I}}$ and by $(L^{N}_j)_{j \in \mathcal{J}}$ are asymptotically $\mathcal{A}$-free up to order $n$ of fluctuations. All the theorems we proved in the zero order case can be easily generalized for the notion of $\mathcal{A}$-freeness up to order $n$ of fluctuations. The generalization of Theorem \ref{sommeetproduit} is given by the following theorem. Recall the notion of $\prec$-defect, denoted by $\eta$, defined in Definition 3.13 of \cite{Gab1}, and the notion of defect, denoted by ${\sf df}$, defined in Equation $(1)$ of \cite{Gab1}.  Recall the notions of $\boxplus$ and $\boxtimes$ convolutions in Notation 6.2 of~\cite{Gab1}.

\begin{theorem}
Let us suppose that $(M^{N}_i)_{i \in \mathcal{I}}$ and $(L^{N}_j)_{j \in \mathcal{J}}$ are asymptotically $\mathcal{A}$-free up to order $n$ of fluctuations. Let $k$ be an integer, $p$ be in $\mathcal{A}_k$, $(i_1,...,i_k)$ be in $\mathcal{I}^{k}$, $(j_1,...,j_{k})$ be in $\mathcal{J}^{k}$ and $i\in \{0,...,n\}$, then: \\
$\bullet$ $\mathbb{E}\kappa_{p}^{i,A}[M_{i_1}+L_{j_1},...,M_{i_k}+L_{j_k}]$ is equal to: 
\begin{align*}
\sum_{(p_1, p_2, I) \in \mathfrak{F}_{2}(p )} \sum_{i'=0}^{i} \mathbb{E}\kappa^{i',\mathcal{A}}_{p_1}(M_{i_1},...,M_{i_k}) \mathbb{E}\kappa^{i-i',\mathcal{A}}_{p_2}(L_{j_1},...,L_{j_k}),
\end{align*}
$\bullet$ $\mathbb{E}\kappa_{p}^{i,A}[M_{i_1}L_{j_1},...,M_{i_k}L_{j_k}]$ is equal to:
\begin{align*}
\sum_{p', p'' \in \mathcal{A}_k, i', j' |p'\circ p''=p,\ i'+j'+\eta(p',p'') =i} \mathbb{E}\kappa^{i',\mathcal{A}}_{p'}(M_{i_1},...,M_{i_k}) \mathbb{E}\kappa_{p''}^{j',\mathcal{A}}(L_{j_1},...,L_{j_k}),
\end{align*}
$\bullet$ $\mathbb{E}m_{p}^{i}[B_1C_1,...,B_kC_k]$ is equal to:
\begin{align*}
 \sum_{p' \in \mathcal{A}_k, i', j' | i'+j'+{\sf df}(p', p) = i} \mathbb{E}\kappa^{i',\mathcal{A}}_{p'}(M_{i_1},...,M_{i_k})\mathbb{E}m_{^{t}p'\circ p}^{j'}(L_{j_1},...,L_{j_k}).
\end{align*}

Thus, for any $i \in \mathcal{I}$ and $j\in \mathcal{J}$, 
\begin{align*}
\mathcal{R}_{\mathcal{\mathcal{A}}}^{(n)}\left[M_{i}+L_j\right] = \mathcal{R}^{(n)}_{\mathcal{A}}\left[M_i \right] \boxplus \mathcal{R}^{(n)}_{\mathcal{A}} \left[L_j\right] \ \ \text{   and   }\ \ 
\mathcal{R}_{\mathcal{A}}^{(n)}\left[M_iL_j\right]  = \mathcal{R}^{(n)}_\mathcal{A}\left[M_i\right]  \boxtimes \mathcal{R}^{(n)}_\mathcal{A} \left[L_j\right]. 
\end{align*} 
\end{theorem}

\begin{proof}
One could do a combinatorial proof as we did for Theorem \ref{th:calculloi}, using results such as Theorem $3.4$ in \cite{Gab1}. Actually, using the up-coming Theorem \ref{Lemainfluct}, one can suppose that the families we consider satisfy the hypotheses of Theorem \ref{Lemainfluct}. Then the result we have to prove is a consequence of Remark \ref{remarque:fluctuhigher}, simple calculations and Theorem $6.3$ of~\cite{Gab1}.
  \end{proof}

\subsection{$\mathcal{G}(\mathcal{A})$-invariance and independence imply $\mathcal{A}$-freeness of higher order}

The following theorem allows us to construct examples of families of sequences of matrices which are asymptotically $\mathcal{A}$-free up to order $n$ of fluctuations: it is a generalization of Theorem~\ref{Lemain}.

\begin{theorem}
\label{Lemainfluct}
Let us suppose that $(M^{N}_i)_{i \in \mathcal{I}}$ and $(L^{N}_j)_{j \in \mathcal{J}}$ converge in $\mathcal{A}$-distribution up to order $n$ of fluctuations. Let us suppose that $(L^{N}_j)_{j \in \mathcal{J}}$ is $\mathcal{G}(\mathcal{A})$-invariant and that for every integer $N$, the two families $(M^{N}_i)_{i \in \mathcal{I}}$ and $(L^{N}_j)_{j \in \mathcal{J}}$ are independent. Then the two families $(M^{N}_i)_{i \in \mathcal{I}}$ and $(L^{N}_j)_{j \in \mathcal{J}}$ are asymptotically $\mathcal{A}$-free up to order $n$ of fluctuations. 
\end{theorem}

\subsection{Convergence of L\'{e}vy processes in $\mathcal{P}$-distribution up to higher order of fluctuations}

In this section, one can find the generalization of Section \ref{sec:Levy} to the higher orders of fluctuations. Let $\left((X_t^{N})_{t \geq 0}\right)_{N \geq 0}$ be a sequence of $\mathcal{G}(\mathcal{A})$-invariant L\'{e}vy processes which are either all additive or multiplicative. For any integers $k$ and $N$, let:  
\begin{align*}
G^{N}_k = \frac{d}{dt}_{\mid t=0} \mathbb{E}\left[(X^{N}_t)^{\otimes k}\right].
\end{align*} 
Recall Section 6 of \cite{Gab1} where the convergence up to higher orders of fluctuations are defined for elements in $\prod_{N \in \mathbb{N}}\mathbb{C}[\mathcal{P}_k(N)]$. 
\begin{definition}
Let us suppose that for any integer $k$, $(G_k^{N})_{N \in \mathbb{N}}$, seen as an element of $\prod_{N \in \mathbb{N}}\mathbb{C}[\mathcal{P}_k(N)]$, converges up to order $n$ of fluctuations. The $\mathcal{R}^{(n)}$-transform of $G$, denoted by $\mathcal{R}^{(n)}(G)$, is the linear form which sends $(p,i) \in \mathcal{P}_k \times \{0,...,n\}$ on $\kappa_{p}^{i}(G)$.  
\end{definition}

Recall the notation $\boxdot$ which stands either for $\boxplus$ or $\boxtimes$.

\begin{theorem}
\label{convergencegeneralefluctua}
Let us suppose that for any integer $k$, $(G_k^{N})_{N \in \mathbb{N}}$ converges up to order $n$ of fluctuations. Then $(X_t^{N})_{t \geq 0}$ converges in $\mathcal{P}$-expectation up to order $n$ of fluctuations as $N$ goes to infinity. For any real $t_0 \geq 0$: 
\begin{align*}
\mathcal{R}^{(n)}[X_{t_0}] = e^{\boxdot t_0 \mathcal{R}^{(n)}(G)}.
\end{align*} 
Besides, in the multiplicative case, for any integer $k$, any $p \in \mathcal{P}_k$, any $t_0 \geq 0$ and any $i \in \{ 0,...,n\}$:
\begin{align*}
\frac{d}{dt}_{\mid t=t_0} \mathbb{E}m_{p}^{i} [X_t] &= \sum_{p_1 \in \mathcal{A}_k} \sum_{i+j+{\sf df}(p_1,p) = i_0} \kappa_{p_1}^{i}(G_k)\mathbb{E}m^{j}_{\!\!\text{ }^{t}p_1 \circ p}[X_{t_0}].
\end{align*}
\end{theorem}

The convergence of $(X_t^{N})_{t \geq 0} \cup ((X_t^{N})^{*})_{t \geq 0}$ can also be studied with the same tools. Let us suppose that we are in the multiplicative case, for any integers $k$, $l$ and $N$, let: 
\begin{align}
G_{k,l}^{N} = \frac{d}{dt}_{\mid t=0} \mathbb{E}\left[\left(X_t^{N}\right)^{\otimes k}\otimes \left(\overline{X_t^{N}}\right)^{\otimes l}\right]. 
\end{align}

\begin{theorem}
If for any positive integers $k$ and $l$, the sequence $(G_{k,l}^{N})_{N \in \mathbb{N}}$ converges up to order $n$ of fluctuations, then the family $(X_t^{N})_{t \geq 0} \cup ((X_t^{N})^{*})_{t \geq 0}$ converges in $\mathcal{P}$-distribution up to order $n$ of fluctuations.
\end{theorem}

One can use these theorems in order to prove that all the convergences, for the L\'{e}vy processes we considered, actually hold  up to any order of fluctuations. 

\section{Conclusion}

In the next article \cite{Gab3}, we apply these results to general random walks on the symmetric groups. This allows us to define the first $\mathcal{P}$-L\'{e}vy processes which can be approximated by random matrices and which are not free L\'{e}vy processes. We also prove that for general random walks on the symmetric group, the asymptotic factorization property does not hold in general: the moments $m_p$ can converge in law and not in probability to a random variable. This proves that, in general, the eigenvalues distributions can converge in law to a random probability measure. 

In the article \cite{GabCebronGuill} in preparation, we explain the link between the theory of $\mathcal{P}$-tracial algebras and the theory of traffics of C. Male described in \cite{Camille}. 

Let us finish with some open questions concerning $\mathcal{P}$-L\'{e}vy processes and fluctuations: 
\begin{enumerate}
\item A natural notion of positivity can be added to the theory of $\mathcal{P}$-tracial algebras in order to have $\mathcal{P}$-tracial probability algebras. Does it exist a characterization of  $\mathcal{P}$-L\'{e}vy processes in $\mathcal{P}$-tracial probability algebras like Theorem 13.16 of \cite{Speic}? 
\item The characterization of $\mathcal{P}$-L\'{e}vy processes in $\mathcal{P}$-tracial probability algebra might need the definition of a general Fock space: which structure would replace the usual Fock space ? 
\item Do asymptotic of moments of the entries, like the one given in Theorem \ref{th:entrees}, can provide some insight in order to understand approximations of non-commutative distributions by random matrices~? 
\item The classical cumulants can be computed as finite dimensional cumulants: this shows that there exists a Schur-Weyl interpretation of the fluctuations of the moments of random matrices. It is possible to recover the fluctuations of the unitary Brownian motion easily in this setting, but a general study has still to be done. Can we apply the combinatorial techniques used for the convergence of $\mathcal{P}$-distribution and obtain a theorem similar to Theorem 5.1 of \cite{Gab1} ?  
\item Does it exist a good notion which generalizes the notion of asymptotic factorization for higher orders, which would be stable by the asymptotic freeness property of higher orders and would imply probabilistic fluctuations of the moments when $N$ goes to infinity. One can prove the convergence of the fluctuations of Hermitian Brownian motions in this way, yet, a general notion is still missing. \\
\end{enumerate}

{\emph{Acknoledgements: }}The author would like to gratefully thank A. Dahlqvist, G. C\'{e}bron, C. Male, R. Speicher and J. Mingo for the useful discussions, his PhD advisor Pr. T. L\'{e}vy for his helpful comments and his postdoctoral supervisor, Pr. M. Hairer, for giving him the time to finalize this article. Also, he wishes to acknowledge the help provided by C. Male who encouraged the author to define the abstract concept of $\mathcal{P}$-tracial algebras, concept which was not present in the first version of the article and who brought also to his knowledge the first negative assertion in Theorem \ref{th:lesdifferentesliberte}. 

The first version of this work has been made during the PhD of the author at the university Paris 6 UPMC. This final version of the paper was completed during his postdoctoral position at the University of Warwick where the author is supported by the ERC grant, “Behaviour near criticality”, held by Pr. M. Hairer.

\bibliographystyle{plain}
\bibliography{biblio}

\end{document}